\documentclass[11pt]{amsart}

\usepackage{amsmath,amsfonts,amssymb,amsthm,amscd,epsfig,color,euscript,mathrsfs, texdraw,multicol,rotating,pifont}
\usepackage[vcentermath,enableskew]{youngtab}
\usepackage[all]{xy}

\newtheorem{theorem}{Theorem}[section]
\newtheorem{proposition}[theorem]{Proposition}
\newtheorem{corollary}[theorem]{Corollary}

\newtheorem{lemma}[theorem]{Lemma}

\theoremstyle{definition}
\newtheorem{definition}[theorem]{Definition}
\newtheorem{example}[theorem]{Example}
\newtheorem{remark}[theorem]{Remark}

\numberwithin{equation}{section}

\newenvironment{red}{\relax\color{red}}{\hspace*{.5ex}\relax}
\newcommand{\ber}{\begin{red}}
\newcommand{\er}{\end{red}}

\newenvironment{blue}{\relax\color{blue}}{\hspace*{.5ex}\relax}
\newcommand{\beb}{\begin{blue}}
\newcommand{\eb}{\end{blue}}

\newcommand{\soplus}{\mathop{\mbox{\normalsize$\bigoplus$}}\limits}

\newcommand{\Cat}{\mathscr{C}}
\newcommand{\field}{\mathbf{k}}
\newcommand{\um}{\underline{m}}
\newcommand{\tb}{\mathtt{b}}

\def\seq{\mathbin{:=}}
\def\g{\mathfrak g}

\def\Z{\mathbb Z}
\def\C{\mathbb C}
\def\F{\mathcal F}

\def\al{\alpha}
\def\be{\beta}
\def\ga{\gamma}
\def\va{\varpi}
\def\calQ{\mathcal Q}
\def\Q{\mathbb Q}
\def\AR{\Gamma_Q}
\def\Hom{{\rm Hom}}
\def\msN{\mathsf{N}}

\def\Prt{\Phi^+_n}
\def\SQbe{S_Q(\be)}

\def\SQal{S_Q(\al)}
\def\SQga{S_Q(\ga)}
\def\VQbe{V_Q(\be)}
\def\VQal{V_Q(\al)}
\def\VQga{V_Q(\ga)}
\def\Up{U_q'(\g)}
\def\Ug{U_q(\g)}

\def\het{{\rm ht}}
\def\Vi{V(\varpi_i)}
\def\Vj{V(\varpi_j)}
\def\eqcw{\widetilde{w}_0}

\def\aff{{\rm aff}}
\def\Maff{M_\aff}
\def\Naff{N_\aff}
\def\inp{\phi^{-1}}
\def\mtM{\mathtt{M}}
\def\mtS{\mathtt{S}}
\def\mtI{\mathtt{I}}
\def\mtP{\mathtt{P}}
\def\dim{{\rm dim}}
\def\Rep{{\rm Rep}}
\def\Supp{{\rm Supp}}
\def\Dim{\underline{\dim}}

\newcommand{\nc}{\newcommand}
\nc{\hs}{\hspace*}
\nc{\Rnorm}{R^{\rm{norm}}}
\nc{\Runiv}{R^{\rm{univ}}} \nc{\ve}{\varepsilon}
\nc{\cvl}{\mathbin{\mbox{\large $\circ$}}}

\nc{\tens}{\mathop\otimes}
\nc{\wl}{\mathsf{P}}
\nc{\rl}{\mathsf{Q}}
\nc{\cm}{\mathsf{A}}
\nc{\redex}{\widetilde{w}}
\nc{\redez}{{\widetilde{w}_0}}

\newcommand{\dct}[2]{\overset{#2}{\underset{#1}{\cvl}} }
\newcommand{\rmat}[1]{{\mathbf r}_{\mspace{-2mu}\raisebox{-.5ex}{${\scriptstyle{#1}}$}}}
\newcommand{\Lto}{\longrightarrow}

\newcommand{\lf}{\{}
\newcommand{\rf}{\}}

\begin{document}

\title[AR quiver of type D and generalized quantum affine Schur-Weyl duality] {Auslander-Reiten quiver of type D and \\
generalized quantum affine
Schur-Weyl duality}

\author[Se-jin Oh]{Se-jin Oh}
\address{School of Mathematics, Korea Institute for Advanced Study Seoul 130-722, Korea}
\email{sejin092@gmail.com}


\begin{abstract}
We first provide an explicit combinatorial description of the Auslander-Reiten quiver $\Gamma^Q$ of finite type $D$. Then
we can investigate the categories of finite dimensional representations over the quantum affine algebra $U_q'(D^{(i)}_{n+1})$ $(i=1,2)$ and
the quiver Hecke algebra  $R_{D_{n+1}}$ associated to $D_{n+1}$ $(n \ge 3)$, by using the combinatorial description and the
generalized quantum affine Schur-Weyl duality functor. As applications, we can prove that Dorey's rule holds for the category $\Rep(R_{D_{n+1}})$ and
prove an interesting difference between multiplicity free positive roots and multiplicity non-free positive roots.
\end{abstract}

\maketitle

\section*{Introduction}

The quiver Hecke algebras, introduced by Khovanov-Lauda \cite{KL09,KL11} and Rouquier \cite{R08}, are in the limelight among the people in the
representation research area because the algebras are related to categorification of quantum groups. Recently, the quiver Hecke algebras attract the people's
attention once again because the algebras can be understood as a generalization of the affine Hecke algebra of type $A$ in the context of the quantum affine
Schur-Weyl duality, which makes bridge between the representations of quiver Hecke algebras and the quantum affine algebras $U'_q(\g)$,
by the results of Kang, Kashiwara and Kim \cite{KKK13a,KKK13b}.

For the quantum affine algebra $\Up$, the finite dimensional integrable representations over $\Up$ have been investigated by many authors during the past twenty years from different
perspectives (see \cite{AK,CH,CP94,FR,GV,Kas02,Nak}). Among these aspects, the theory of $R$-matrix provides crucial information
for constructing the quantum affine Schur-Weyl duality functor in \cite{KKK13a,KKK13b} (see also \cite{CP96B,Che,GRV94,Jimbo86}).

As a continuation of the previous series paper \cite{Oh14A}, we
provide an explicit combinatorial description of the
Auslander-Reiten(AR) quiver $\AR$ of finite type $D$ and apply the
combinatorial description to investigate
\begin{itemize}
\item[(i)] the category $\Cat_Q^{(i)}$ $(i=1,2)$, consisting of finite dimensional integrable modules over the quantum affine algebra $U_q'(D^{(i)}_{n+1})$
depending on the AR-quiver $\AR$ (\cite{Her10,HL11}),
\item[(ii)] the category $\Rep(R_{D_{n+1}})$,  consisting of finite dimensional graded modules over the quiver Hecke algebra $R_{D_{n+1}}$ associated to $D_{n+1}$
$(n \ge 3)$,
\end{itemize}
with the exact quantum affine Schur-Weyl duality functor
$$ \mathcal{F}^{(1)}_Q: R_{D_{n+1}} \longrightarrow  \Cat_Q^{(1)}.$$
Here $Q$ is any Dynkin quiver of finite type $D_{n+1}$ by orienting edges of Dynkin diagram of finite type $D_{n+1}$.

Let $\Phi_{n}^+$ be the set of all positive roots associated to finite Dynkin diagram of finite type $A_n$, $D_n$, $E_6$, $E_7$ or $E_8$.
Then it is well-known (\cite{Gab80}) that
\begin{itemize}
\item[(i)] the vertices of $\AR$ can be identified with the set $\Phi_{n}^+$ and the set of all isomorphism classes of
indecomposable modules over the path algebra $\C Q$,
\item[(ii)] the dimension vector
of indecomposable module corresponding $\be \in \AR$ is indeed $\be$,
\item[(iii)]
arrows in $\AR$ present the irreducible morphisms between the indecomposables,
\item[(iv)] $\AR$ provides
{\it the unique convex partial order} $\prec_Q$ on $\Phi_{n}^+$ which is compatible with paths in $\AR$ (\cite{B99}).
\end{itemize}

Note that each positive root $\be$ in $\Prt$ of finite type $D$ can be expressed by the tuple of integers $\lf a,\pm b \rf$ $(1 \le a<b \le n)$ where
$\be=\ve_a \pm \ve_b$. We say $\ve_a$ and $\pm \ve_b$ as summands of $\be$. Identifying $\be$ with $\lf a,\pm b \rf$,
every positive root appearing in the {\it maximal $N$-sectional} (resp. {\it $S$-sectional}) path and {\it the maximal swing} in $\AR$
share the same summand as $\ve_a$ or $\pm \ve_b$ (Theorem \ref{Thm: V-swing}, Theorem \ref{thm: short path}).

With the explicit combinatorial description of $\AR$ of finite type $D_n$, we can prove that the {\it Dorey's rule} in \cite{CP96} always holds
 for all $\al \prec_Q \be \in \Prt$ with
$\ga=\al+\be \in \Prt$ (Section \ref{sec: Dorey minimal}); i.e., the following surjective homomorphisms
exist:
$$\VQbe \otimes \VQal  \to \VQga \quad \text{ and } \quad \SQbe \cvl \SQal  \to \SQga, \quad \text{ where}$$
\begin{itemize}
\item $\VQbe \seq \Vi_{(-q)^{p}}$ is the fundamental $U_q'(D_n^{(1)})$-module for $\inp(\be,0)=(i,p),$
\item  $\SQbe$ is the simple $R_{D_n}$-module which is the preimage of $\VQbe$
under the functor $\mathcal{F}^{(1)}_Q$, which is studied in \cite{KKK13b}.
\end{itemize}

For a total order $<$ on $\Prt$, a pair $(\al,\be)$ with $\al<\be$ is called a {\it minimal pair} of $\ga \in \Prt$ if
(i) $\ga=\al+\be$, (ii) there exists {\it no} pair $(\al',\be')$ such that $\ga=\al'+\be' \text{ and }\al<\al'<\ga<\be'<\be.$

We call a positive root $\be=\sum_{k}n_k\al_k \in \Phi^+_{n}$ {\it multiplicity free} if $n_k \le 1$ for all $1 \le k \le n$. Note that
there exist positive roots of $\Prt$ of finite type $D$ which are not multiplicity free. Abstracting this notion, we can define a notion of {\it multiplicity} on
{\it non-simple} positive roots as follows:

\medskip

\noindent
\textbf{Definition.} For every non-simple positive root $\ga\in \Prt$ associated to the finite Dynkin diagram of finite type $A$, $D$ and $E$,
the {\it multiplicity} of  $\gamma = \sum_{i \in I} n_i \alpha_i \in \Phi_n^+$ is the integer ${\rm mul}(\gamma)$ defined as follows:
$$ {\rm mul}(\gamma) := \max\{ n_i \ | \ i \in I \}.$$

Using the combinatorial properties of $\AR$, we can prove \cite[Conjecture]{Oh14A} when $Q$ is of finite type $D$:

\medskip

\noindent
\textbf{Theorem.} For any Dynkin quiver $Q$ of finite type $D$, every pair $(\al,\be)$ of a non-simple positive root $\ga=\al+\be \in \Prt$ is
 minimal with respect to a suitable total order compatible with
$\prec_Q$ if and only if $\ga$ is multiplicity free.

\medskip

In other word, for a non-simple positive root $\gamma \in \Phi^+_{D_n}$ with ${\rm mul}(\gamma)\ge 2$ has a pair $(\alpha,\beta)$ which can not be minimal for {\it any} total order compatible with
$\prec_Q$. We also prove that there exist$(n-b-1)$-pairs $(\al,\be)$ of multiplicity non-free positive root $\ve_a+\ve_b=\al+\be$ ($1 < b \le n-2$) such that
they can not be {\it minimal} (Theorem \ref{thm: non-free}). Hence, by regarding {\it the height} and {\it the $k \ge $-supports} of $\gamma \in \Phi^+$
$(k \in\Z_{\ge 1})$, we know the number of (minimal, non-minimal) pairs of $\gamma$ very easily (Corollary \ref{cor: num min nonmin}).

In Remark \ref{rmk: non-adapted}, we give an example explaining that the interesting difference between
multiplicity free positive roots and multiplicity non-free positive roots happens only when the reduced expression is
{\it adapted to some Dynkin quiver $Q$.}

In Theorem \ref{thm: surj free}, for a pair $(\al,\be)$ with $\al+\be \in \Phi_n^+$, $\phi^{-1}(\al,0)=(k,p)$ and $\phi^{-1}(\be,0)=(l,q)$,  we explain that {\rm (i)}
$(-q)^{|p-q|}$ is a root of the denominator $d_{k,l}(z)$ between fundamental representations over $U_q'(D_n^{(1)})$, {\rm (ii)} the
{\it minimality} of a pair $(\al,\be)$ can be interpreted as the {\it multiplicity} of $(-q)^{|p-q|}$ as a root of $d_{k,l}(z)$.
In other words, we can extract factors of $d_{k,l}(z)$ by reading {\it any} $\Gamma_Q$ with respect to the pairs $(\al,\be)$ such that
$\al+\be \in \Phi_n^+$, $\phi^{-1}(\al,0)=(k,p)$ and $\phi^{-1}(\be,0)=(l,q)$.

\medskip

The outline of this paper is as follows. In Section 1, we first
recall the definition AR-quivers $\AR$ and their basic
theories, review the various orders on $\Prt$, and study new
combinatorial characterization of $\AR$ of finite type $D$. In Section
2, we briefly recall the backgrounds, theories of the generalized
quantum affine Schur-Weyl duality, and observe the similarity arising from the denominators and
Dorey's type morphisms of $\Rnorm_{k,l}(z)$ the quantum affine algebra
$\g=D^{(1)}_{n+1}$ and $\g=D^{(2)}_{n+1}$. In Section \ref{sec:
Dorey minimal}, we prove how the structure of $\AR$
reflects the Dorey's rule for $D^{(1)}_{n}$ and study the minimality
of given pair $(\al,\be)$ of $\gamma = \al+\be \in \Phi^+_{D_n}$ depending on ${\rm mul}(\gamma)$. In
the last section, we apply the results in the previous sections to
describe the category $\Cat^{(2)}_Q$ and the certain
conditions for $\al,\be \in \Prt$ such that $\SQbe
\cvl \SQal \simeq \SQal \cvl \SQbe$ is simple.
\bigskip

\noindent
{\bf Acknowledgements.} The author would like to express his gratitude to
Professor Masaki Kashiwara, Professor Kyungyong Lee and Myungho Kim for many fruitful discussions.
The author gratefully acknowledge the hospitality of RIMS (Kyoto) during his visit in 2013 and 2014.

\section*{Comments for readers}
In this paper, the readers encounter several cases which the author does not cover in a proof. The reason is that
all of the non-covered cases can be proved in the similar way of the cases given in the proof. If the author covers the all cases,
this paper will be tedious and long than necessary.

\section{Combinatorial characterization of AR-quivers of finite type $D$} \label{sec: Combinatorics}
In this section, we provide an explicit combinatorial description of AR-quivers of finite type $D$
(see \cite{Oh14A} for type $A$). To do this, we need to recall the background of
Auslander-Reiten theory, Gabriel theory and various orders on $\Prt$ briefly (see \cite{ARS,ASS,Bour,Gab80} also).

\subsection{Gabriel's Theorem.} \label{subsec: G Theorem}
We denote by $\Delta_n$ a rank $n$ Dynkin diagram of a finite simple Lie algebra $\g_0$ with vertices $I_0=\{ 1,2,\ldots,n\}$.
Let $Q=(Q_0=I_0,Q_1)$ be a Dynkin quiver by orienting edges of $\Delta_n$. For a given $\g_0$, we denote by
(i) $\Pi_n=\{ \al_i \ | \ i \in I_0 \}$ the set of simple roots,
(ii) $\Phi_n$ (resp. $\Prt$) the set of (resp. positive) roots
(iii) $W_0$ the Weyl group generated by simple reflections $\{ s_i \ | \ i \in I_0 \}$ and
(iv) $w_0$ the longest element of $W_0$.

It is well-known that $W_0$ acts on $\Phi_n$ and $w_0$ induces an involution on $I_0$, $i \mapsto i^*$, as follows:
$$ w_0(\al_i) = -\al_{i^*}.$$

The finite dimensional representation $\mtM$ of the path algebra $\C Q$ consists of
\begin{itemize}
\item $\{ \mtM_i \ | \ i \in Q_0 \}$ the set of finite dimensional vector spaces labeled by $Q_0$ such that
$\mtM= \soplus_{i \in Q_0}\mtM_i$,
\item $\{ \varsigma_{\mathsf{a}} : \mtM_{s(\mathsf{a})} \to \mtM_{t(\mathsf{a})} \ | \ \mathsf{a} \in Q_1  \}$ the set of linear maps labeled by
the set of arrows $Q_1$,
\end{itemize}
where $s,t: Q_1 \to Q_0$ are the {\it source} and {\it target} maps, respectively.

A {\it path} $\mathsf{p}$ of $Q$ is a product of arrows
$\mathsf{a}_1\mathsf{a}_2 \cdots \mathsf{a}_r$ such that
$$ t(\mathsf{a}_k)=s(\mathsf{a}_{k+1}) \quad \text{ for all } 1 \le k < r.$$
We set $s(\mathsf{p})\seq s(\mathsf{a}_1)$ and $t(\mathsf{p})\seq t(\mathsf{a}_r)$.

We denote by ${\rm Mod} \C Q$ the category of finite representations of the path algebra $\C Q$. We define the {\it dimension vector} $ \Dim \mtM$
of $\mtM$ in ${\rm Mod} \C Q$ as follows:
$$ \Dim \mtM = \sum_{i \in Q_0} {\rm dim}_\C (\mtM_i)\al_i \in \sum_{i \in I_0} \Z_{\ge 0}\al_i.$$

Then it is known that the set of all simple modules $\mathsf{Irr}(Q)$ in ${\rm Mod} \C Q$ (up to isomorphism) can be labeled by $Q_0$. We write
$\mathsf{Irr}(Q):= \{ \mtS(i) \ | \ i \in Q_0\}$. Moreover, the dimension vector of $\mtS(i)$ is the same as $\al_i$.

If $Q$ is a Dynkin quiver of finite type $A$, $D$ or $E$, ${\rm Mod} \C Q$ has more interesting relation with the finite simple Lie algebra $\g_0$, which is
 known as {\it Gabriel's theorem}.

\begin{theorem} \label{thm: Ga} Assume that $Q$ is a Dynkin quiver of finite type $A$, $D$ or $E$. Let us denote by $\mathsf{Ind}(Q)$ the set of all indecomposable modules in ${\rm Mod} \C Q$ (up to isomorphism).
Then the map $\mathsf{Ind}(Q) \ni \mtM \mapsto \Dim \mtM$ gives a bijection between $\mathsf{Ind}(Q)$ and $\Prt$. Thus we can write
$$ \mathsf{Ind}(Q) = \{ \mtM(\be) \ | \ \Dim\mtM(\be)=\be \in \Prt   \}.$$
\end{theorem}

\subsection{Auslander-Reiten quiver.} \label{subsec: AR-quiver}
For a vertex $i$ of $Q_0$, we say that $i$ is a {\it source} (resp. {\it sink}) if
an arrow $\mathsf{a}$ is connected with $i$, then $s(\mathsf{a})=i$ (resp. $t(\mathsf{a})=i$). For a quiver $Q$ and $i \in Q_0$, we define the quiver
$s_iQ$ by reversing the orientation of each arrow $\mathsf{a} \in Q_0$ with $s(\mathsf{a})=i$ or $t(\mathsf{a})=i$.

\medskip

{\it For the rest of this section, we assume that $Q$ is a Dynkin quiver of finite type $A$, $D$ or $E$.}

\medskip

For a reduced expression $\tilde{w}=s_{i_1}s_{i_2}\cdots s_{i_r}$ of an element $w \in W_0$, we say that it is {\it adapted} to $Q$ if
$$  \text{ $i_k$ is a source of the quiver $s_{i_{k-1}}\cdots s_{i_2} s_{i_1}Q$ for all $1 \le k \le r$.} $$
It is well-known that there exists a unique {\it Coxeter element} $\tau$ adapted to $Q$.

For a Dynkin quiver $Q$, we say that a map $\xi:Q \to \Z$ is a {\it height function} on $Q$ if
$$ \xi_j=\xi_i -1 \quad \text{ for } i \to j \in Q_1.$$
Since $Q$ is connected, all pair $(\xi,\xi')$ of height functions differ by constant. Let us fix a height function $\xi$.

Set
$$ \Z Q \seq \{ (i,p) \in \{ 1,2,\ldots,n\} \times \Z \ | \ p -\xi_i \in 2\Z \}.$$
We view $\Z Q$ as the quiver with arrows
$$ (i,p) \to (j,p+1), \ (j,q) \to (i,q+1) \ \text{ for which $i$ and $j$ are adjacent in $\Delta_n$.}$$
and call it the {\it repetition quiver} of $Q$. Note that $\Z Q$ does not depend on the orientation of the quiver $Q$.
It is well-known that the quiver $\Z Q$ itself has an isomorphism with the AR-quiver of $D^b(\C Q)$-${\rm mod}$,
the bounded derived category of $\C Q$-${\rm mod}$. In our convention, the injective module $\mathtt{I}(i)$ is located on the vertex
$(i,\xi_i)$ of $\Z Q$.

For a vertex $i \in Q_0$, we define positive roots $\eta_i$ and $\zeta_i$ defined as follows:
\begin{equation} \label{eq: dim zeta eta}
\eta_i \seq \sum_{j \in B(i)} \al_j \quad \text{ and } \quad \zeta_i \seq \sum_{j \in C(i)} \al_j,
\end{equation}
where $B(i)$ (resp. $C(i)$)  is the set of vertices $j \in Q_0$  such that there exists a path $\mathsf{p}$ in $Q$ with $s(\mathsf{p})=j$ and $t(\mathsf{p})=i$
(resp. $s(\mathsf{p})=i$ and $t(\mathsf{p})=j$).

\begin{definition} \cite[\S 6.4]{HL10} \label{def: mfree}
We call a positive root $\be=\sum_{i \in I_0}n_i\al_i$ {\it multiplicity free} if $n_i \le 1$ for all $i \in I_0$.
\end{definition}

By \eqref{eq: dim zeta eta} and Definition \ref{def: mfree}, we have
\begin{equation} \label{eq: eta zeta}
\text{ all $\eta_i$ and $\zeta_i$ $(i \in I_0)$ are multiplicity free positive roots.}
\end{equation}

Set $\widehat{\Phi}^+_n \seq \Prt \times \Z$. The bijection $\widehat{\phi}: \Z Q \to \widehat{\Phi}^+_n$
defined by $\mathtt{M}(\beta)[m] \mapsto (\beta,m)$ can be described in the following combinatorial way(\cite[\S 2.2]{HL11}):
\begin{equation}
\begin{aligned} \label{eq:widephi}
& {\rm (i)} \ \widehat{\phi}(i,\xi_i)=(\eta_i,0), \\
& {\rm (ii)} \text{ For a given $\widehat{\phi}(i,p)=(\be,m)$}, \
\begin{cases}
\widehat{\phi}(i,p-2)=(\tau(\be), m) & \text{ if } \tau(\be) \in \Prt, \\
\widehat{\phi}(i,p-2)=(-\tau(\be), m-1)& \text{ if } \tau(\be) \in \Phi^-_n, \\
\widehat{\phi}(i,p+2)=(\tau^{-1}(\be), m)& \text{ if } \tau^{-1}(\be) \in \Prt, \\
\widehat{\phi}(i,p+2)=(-\tau^{-1}(\be), m+1)& \text{ if } \tau^{-1}(\be) \in \Phi^-_n.
\end{cases}
\end{aligned}
\end{equation}
We write the injection $\widehat{\phi}^{-1}|_{\Prt \times 0}$ as $\inp$.
\begin{definition}
The {\it Auslander-Reiten quiver} $\AR=( (\AR)_0,(\AR)_1)$ is the full subquiver of $\Z Q$ whose set of vertices is the same as
$\inp(\Prt) \seq \widehat{\phi}^{-1}(\Prt,0)$. Thus one can identify $(\AR)_0$ with $\Prt$.
\end{definition}

Considering Theorem \ref{thm: Ga}, the following facts are known:
\begin{itemize}
\item[({\rm a})] The vertices $\be \in (\AR)_0$ corresponds to $\mtM(\be)$ in $\mathsf{Ind}(Q)$.
\item[({\rm b})] The arrow $\be \to \be' \in (\AR)_1$ corresponds to the {\it irreducible morphism}
from $\mtM(\be)$ to $\mtM(\be')$.
\end{itemize}
In particular,
the projective cover $\mtP(i)$ of $\mtS(i)$ corresponds to $\eta_i$, and
the injective envelope $\mtI(i)$ of $\mtS(i)$ corresponds to $\zeta_{i^*}$.

For $\be \in \Prt$ with $\tau(\be) \in \Prt$, the following property holds:
\begin{align} \label{eq: ad fn}
\be + \tau(\be) = \sum_{\theta \in X(\be)} \theta
\end{align}
where $X(\be)$ denote the set of positive roots $\theta$ such that there exists an arrow $\theta \to \be$ in $(\AR)_1$.

For $i \in Q_0$, we define
\begin{equation}\label{Def: m_i}
m_i= \max\{ k \in \Z_{\ge 0} \ | \ \tau^k(\eta_i) \in \Prt\}.
\end{equation}

Then the AR-quiver $\AR$ satisfies the following properties:
\begin{align}
& (\AR)_0 = \{ (i,p) \in \Z Q \ | \ \xi_i-2m_i \le p \le \xi_i \}, \label{eq: known characterization} \\
& \xi_{i^*}-2m_{i^*}= \xi_i - \mathtt{h}_n +2. \label{eq: Nakayama equation}
\end{align}
Here $\mathtt{h}_n$ denotes the Coxeter number of $\g_0$.

\subsection{Orders on $\Prt$.} \label{subsec: various order} In this subsection, we shall recall various orders on $\Prt$.

\medskip

We say that an order $<$ on $\Phi_n^+$ is {\it convex} if the order satisfies the following property:
$$\text{For all $\alpha,\beta$ and $\gamma=\alpha+\beta \in \Phi_n^+$, either $\alpha < \gamma < \beta$ or $\beta < \gamma< \alpha$.}$$
It is well-known that a reduced expression $\redez$ of $w_0$ induces a {\it convex total order} $<_{\redez}$ on $\Prt$ as follows (\cite{Bour}):
$$ \beta_z \seq  s_{i_1}s_{i_2} \cdots s_{i_{z-1}}\alpha_{i_z} \text{ and $\beta_x <_{\redez} \beta_y$ if and only if $x < y$}.$$
Moreover, any convex total order is induced by a reduced expression $\redez$ of $w_0$ (\cite{Papi94}).

We say that two reduced expressions $\widetilde{w}=s_{i_1}s_{i_2}\cdots s_{i_{\ell(w)}}$ and $\widetilde{w}'=s_{j_1}s_{j_2}\cdots s_{j_{\ell(w)}}$ of $w \in W_0$
are {\it commutation equivalent}, denoted by $\widetilde{w} \sim \widetilde{w}' \in [\widetilde{w}]$, if $s_{j_1}s_{j_2}\cdots s_{j_{\ell(w)}}$ is obtained
from $s_{i_1}s_{i_2}\cdots s_{i_{\ell(w)}}$ by replacing
$s_{a}s_{b}$ by $s_{b}s_{a}$ for $a$ and $b$ not linked in $\Delta_n$ (see \cite{B99})

For a reduced expression $\redez$ of $w_0$ adapted to $Q$, the following is well-known:
\begin{eqnarray} &&
\parbox{85ex}{
If $\redex_0' \in [\redex_0]$, then $\redex_0'$ is adapted to $Q$. Conversely,
any $\redex_0'$ adapted to $Q$ is in $[\redex_0]$.
} \label{eq: [Q]}
\end{eqnarray}
Thus we can write $[Q]\seq[\redex_0]$ for the $\redez$.

\begin{eqnarray} &&
\parbox{85ex}{Note that a commutation class $[w_0]$ of $w_0$ determines the {\it coarsest} convex partial order $\prec_{[\redex_0]}$ on $\Phi_n^+$.
In particular, the order $\prec_Q \seq \prec_{[Q]}$ is defined by the paths in $\Gamma_Q$ (\cite{B99,R96}). More precisely,
\begin{itemize}
\item for a pair $(\alpha,\beta) \in \Phi_n^+$ with $\gamma=\alpha+\beta\in \Phi_n^+$ and
$\redez \sim \redez'$, then we have
$$   \alpha <_{\redex_0} \gamma <_{\redex_0} \beta \quad \text{ if and only if } \quad \alpha <_{\redex'_0} \gamma <_{\redex'_0} \beta,$$
\item $\alpha \prec_Q \beta$ if and only if there exist paths from $\beta$ to $\alpha$ in $\Gamma_Q$.
\end{itemize}
}\label{eq: [redez]}
\end{eqnarray}

The following theorem provides a way of obtaining all reduced expressions of $w_0$ in $[Q]$ and hence convex total orders
compatible with the convex partial order $\preceq_Q$:
\begin{align} \label{eq: compati}
\alpha \prec_Q \beta \text{ implies } \alpha <_{\widetilde{w}_0} \beta \text{ for any } \widetilde{w}_0 \in [Q].
\end{align}

\begin{theorem} \cite[Theorem 2.17]{B99} \label{Thm: compatible reading}
Any reduced expression in the equivalence class $[Q]$ can be obtained by reading $\AR$ in the following way:
If there exists an arrow $\be \to \al \in (\AR)_1$, we read $\al$ before $\be$.
Replacing vertex $\be$ by $i$ for $\inp(\be)=(i,p)$, we have a sequence $(i_1,i_2,\ldots,i_{\msN})$ giving a reduced expression of
$w_0$ adapted to $Q$,
$$w_0 = s_{i_1}s_{i_2}\cdots s_{i_{\msN}}.$$
\end{theorem}

\begin{remark} \label{Rmk: total orders} We fix $Q$ as the Dynkin quiver associated to type $D_n$ $(n \ge 4)$.
In this remark, we give four canonical readings of $\AR$ which are compatible with Theorem \ref{Thm: compatible reading}.
Thus we have four convex total orders on $\Prt$ which are compatible with the convex partial order $\preceq_Q$:
\begin{enumerate}
\item[({\rm A})]  $<^{U,1}_Q$ (resp. $<^{U,2}_Q$) is the convex total order induced from the following reading:
$$ \text{we read $(i,p)$ before $(i',p') \iff  \begin{cases} d(1,i)-p > d(1,i')-p' \text{ or } \\ d(1,i)-p = d(1,i')-p'
\text{ and } d(1,i) >  d(1,i'), \text{ or }  \\
d(1,i)-p = d(1,i')-p', \ d(1,i) =  d(1,i') \text{ and } i>i'\\
\hspace{40ex} \text{ (resp $i<i'$).}
 \end{cases}$}$$
\item[({\rm B})] $<^{L,1}_Q$ (resp. $<^{L,2}_Q$) is the convex total order induced from the following reading:
$$ \text{we read $(i,p)$ before $(i',p') \iff  \begin{cases} d(1,i)+p < d(1,i')+p' \text{ or } \\ d(1,i)+p = d(1,i')+p'
\text{ and } d(1,i) < d(1,i'), \text{ or }
\\ d(1,i)+p = d(1,i')+p', \ d(1,i) = d(1,i') \text{ and } i>i' \\
\hspace{40ex}  \text{ (resp $i<i'$).}\end{cases}$}$$
\end{enumerate}
Here, $d(i,j)$ denotes the distance between $i$ and $j$ in $\Delta_n$. Note that there are many other readings which are
compatible with Theorem \ref{Thm: compatible reading}.
\end{remark}

\begin{definition} \label{Def: minimal pair}\cite[\S 2.1]{Mc12}.
Let $<$ be any total order on $\Prt$ (need not convex). We say pair $(\al,\be)$ with $\al<\be$ and $\ga=\al+\be \in \Prt$ {\it minimal with respect to $<$}  if
there exists {\it no} pair $(\al',\be')$ such that $\ga=\al'+\be' \text{ and }\al<\al'<\ga<\be'<\be.$
\end{definition}

\begin{remark} \label{rmk: non minimal pair}
For every convex total order which is compatible with $\preceq_Q$, a pair $(\al,\be)$ of $\al+\be=\ga \in \Prt$ can {\it not}
be a minimal pair of $\ga$ if there exists another pair $(\al',\be')$ such that
\begin{eqnarray*} &&
\parbox{75ex}{
$\al'+\be'=\ga$ and there exist paths from $\be$ to $\be'$ and $\al'$ to $\al$.
}
\end{eqnarray*}
\end{remark}

\subsection{Characterization of Auslander-Reiten quiver of finite type $D$}
The combinatorial properties in this subsection might be known to the experts through their computations on
$\AR$ of finite type $D$. More precisely, using a fixed quiver, the {\it reflection functor} on $D^b(\C Q)$-${\rm mod}$ and
the tilting theorem (\cite[Chapter VII]{ASS}), one can observe the descriptions in this subsection.
However, we have a difficulty for finding in standard textbooks and need an explicit description for
the later use. Thus we shall derive the properties from Lemma \ref{Lem: Key Lem} below and the contents in \S \ref{subsec: AR-quiver}.

\medskip

For the rest of this section, $\Delta_{n}$ means the Dynkin diagram of finite type $D$ with the following enumeration:
$$\raisebox{1.3em}{\xymatrix@R=0.1ex{
&&&&*{\circ}<3pt> \ar@{-}[dl]^<{n-1}  \\
*{\circ}<3pt>\ar@{-}[r]_<{1 \ }&*{\circ}<3pt>\ar@{-}[r]_<{2 \ } &\cdots\ar@{-}[r]  &*{\circ}<3pt>
\ar@{-}[dr]_<{n-2} \\  &&&&*{\circ}<3pt>
\ar@{-}[ul]^<{\ \ n}}} \quad \text{ for } \quad n \ge 4.$$
The Coxeter number $\mathtt{h}_n$ is $2n-2$ and the
involution $^*$ induced by $w_0 \in W_0$ is given by $i^*=i$ for $1 \le i \le n-2$ and $(n-1)^*=n-1$, $n^*=n$ if $n$ is even,
$(n-1)^*=n$, $n^*=n-1$ if $n$ is odd.
Note that
\begin{equation} \label{eq: difference n n-1}
|\xi_{n-1}-\xi_n| = 2 \text{ or } 0.
\end{equation}
Thus \eqref{eq: Nakayama equation} and \eqref{eq: difference n n-1} tell that
\begin{equation} \label{eq: m n-1 m n}
m_i=n-2 \ \text{ for $1 \le i \le n-2$} \ \text{ and } \  m_{n-1}+m_n=2n-4 \quad \text{ with } \  m_{n-1},m_n \ge n-3.
\end{equation}

We say that $b \in \{2,\cdots ,n-2\}$ is a right intermediate (resp. a left
intermediate) if
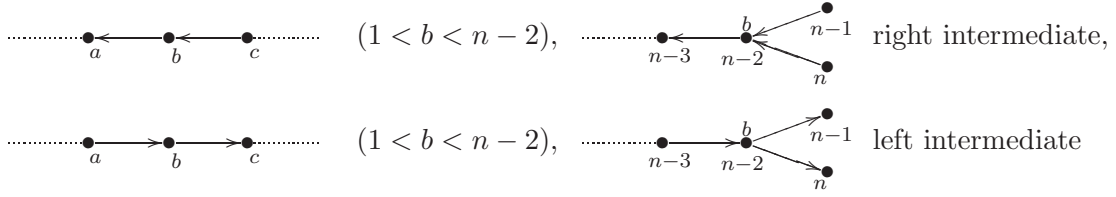
\begin{align*}
&\raisebox{1.2em}{\xymatrix@R=1ex{ & \\ *{ \ }<3pt> \ar@{..}[r]  &*{\bullet}<3pt>
\ar@{<-}[r]_<{a}  &*{\bullet}<3pt>
\ar@{<-}[r]_<{b} &*{\bullet}<3pt>
\ar@{..}[r]_<{c}  & *{ \ }<3pt>
}\quad \raisebox{-1.2em}{\text{($1<b < n-2$),}}
\xymatrix@R=1ex{
&&&*{\bullet}<3pt> \ar@{->}[dl]^<{n-1} \\
*{ \ }<3pt> \ar@{..}[r] & *{ \bullet }<3pt> \ar@{<-}[r]_<{n-3}  &*{\bullet}<3pt>
\ar@{-}[dr]_<{n-2}^<{b \ \ \ \ } \\  &&&*{\bullet}<3pt>
\ar@{->}[ul]^<{\ \ n}}} \text{ right intermediate,}
\\
&\raisebox{1.2em}{\xymatrix@R=1ex{ & \\ *{ \ }<3pt> \ar@{..}[r]  &*{\bullet}<3pt>
\ar@{->}[r]_<{a}  &*{\bullet}<3pt>
\ar@{->}[r]_<{b} &*{\bullet}<3pt>
\ar@{..}[r]_<{c}  & *{ \ }<3pt>
}
\quad \raisebox{-1.2em}{\text{($1<b < n-2$),}}
\xymatrix@R=1ex{
&&&*{\bullet}<3pt> \ar@{<-}[dl]^<{n-1} \\
*{ \ }<3pt> \ar@{..}[r] & *{ \bullet }<3pt> \ar@{->}[r]_<{n-3}  &*{\bullet}<3pt>
\ar@{->}[dr]_<{n-2}^<{b \ \ \ \ } \\  &&&*{\bullet}<3pt>
\ar@{-}[ul]^<{\ \ n}}} \text{ left intermediate}
\end{align*}
in the Dynkin quiver $Q$.

Note that every positive root $\be \in \Prt$ can be written in the following form:
\begin{align*}
\be  = \begin{cases}
\varepsilon_i -  \varepsilon_j = \sum_{ i \le k < j} \al_k  & \ (1 \le i < j \le n), \\
\varepsilon_i +  \varepsilon_n = \sum_{ i \le k \le n-2}  \al_k + \al_n & \ (1 \le i <n), \\
\varepsilon_i +  \varepsilon_j = \sum_{ i \le k < j} \al_k + 2 \sum_{ j \le k \le n-2} \al_k +\al_{n-1}+\al_{n}  & \ (1 \le i < j < n),
\end{cases}
\end{align*}
where $(\varepsilon_a,\varepsilon_b)=\delta_{a,b}$. Thus one can identify $\be \in \Prt=(\AR)_0$ with $\lf a,\pm b \rf$ (see \cite[PLATE IV]{Bour}).
For $\be=\lf a,\pm b \rf \in \Prt$, we call $\ve_a$ and $\pm \ve_{b}$ {\it summands} of $\be$.

\begin{definition}
For an integer $k \in \Z_{\ge 1}$ and a positive root $\be$, we define {\it $k \ge $-support of $\be$}, denoted by $\Supp_{\ge k}(\beta)$, in the following way:
$$\Supp_{\ge k}(\be) \seq \{ i \in I \ | \ n_i \ge k \} \quad \text{ where } \be=\sum_{i \in I}n_i\al_i.$$
\end{definition}

\begin{example} \label{ex: example 1}
Let the quiver $\raisebox{1.2em}{\xymatrix@R=0.5ex{
&&*{\bullet}<3pt> \ar@{->}[dl]^<{ 3} \\
*{ \bullet }<3pt> \ar@{<-}[r]_<{1}  &*{\bullet}<3pt>
\ar@{->}[dr]_<{2}\\
&&*{\bullet}<3pt> \ar@{-}[ul]^<{\ \ \ 4}}}$
of finite type $D_4$ and a height function $\xi$ with $\xi_3=0$ be given.
Then $\AR$ can be drawn by using \eqref{eq:widephi}, or the additive property of dimension vectors
\eqref{eq: ad fn} as follows:
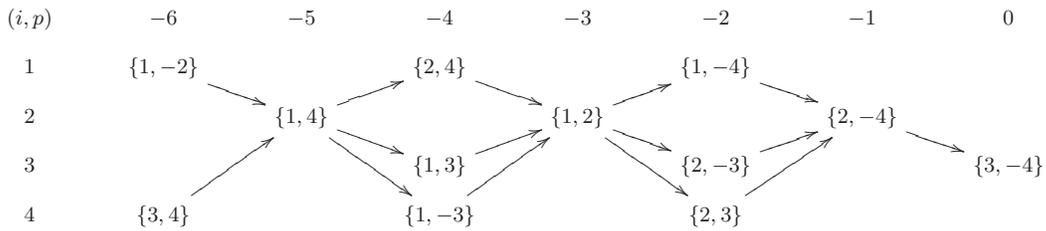
\fontsize{9}{9}\selectfont
$$ \scalebox{0.9}{\xymatrix@R=1ex{
(i,p) & -6 & -5 & -4 & -3 & -2 & -1 & 0 \\
1&\lf 1,-2 \rf \ar@{->}[dr] && \lf 2,4 \rf\ar@{->}[dr] && \lf 1,-4 \rf \ar@{->}[dr]  \\
2&& \lf 1,4 \rf \ar@{->}[dr]\ar@{->}[ddr]\ar@{->}[ur] && \lf 1,2 \rf\ar@{->}[ddr]\ar@{->}[dr]\ar@{->}[ur] && \lf 2,-4 \rf \ar@{->}[dr] \\
3&&& \lf 1,3\rf \ar@{->}[ur] && \lf 2,-3 \rf \ar@{->}[ur] && \lf 3,-4 \rf \\
4& \lf 3,4 \rf \ar@{->}[uur] , && \lf 1,-3 \rf \ar@{->}[uur] && \lf 2,3 \rf \ar@{->}[uur]
}}
$$
\fontsize{11}{11}\selectfont
\end{example}

For $\be \in \Prt$ with $\inp(\be)=(i,p)$, we denote by $\inp_1(\be)=i$  (call it {\it level} of $\be$),
and $\inp_2(\be)=p$.

\begin{definition} \hfill
\begin{enumerate}
\item[({\rm a})] A connected subquiver $\rho$ of $\AR$ is an {\it $S$-sectional path}
if $\rho$ is a concatenation of arrows whose forms are $(i,p) \to (i+1,p+1)$ for $1 \le i \le n-2$, or $(n-2,p) \to (n,p+1)$.
\item[({\rm b})] A connected subquiver $\rho$ of $\AR$ is an {\it $N$-sectional path}
if $\rho$ is a concatenation of arrows whose forms are $(i,p) \to (i-1,p+1)$ for $2 \le i \le n-1$, or $(n,p) \to (n-2,p+1)$.
\item[({\rm c})] A positive root $\be \in \Phi^+$ is {\it contained in the subquiver $\rho$} in $\AR$ if
$\be$ is an end or a start of some arrow in the subquiver $\rho$.
\item[({\rm d})] An $S$-sectional (resp. $N$-sectional) path $\rho$ is {\it maximal} if there is no bigger $S$-sectional (resp. $N$-sectional) path
containing all positive roots in $\rho$.
\item[({\rm e})] A connected subquiver $\varrho$ in $\AR$ is called a {\it swing} if it consists of vertices and arrows in the following way: There exist roots
$\al,\be \in \Phi^+$ and $r,s \le n-2$ such that
$$\raisebox{2em}{\xymatrix@C=2.5ex@R=0.5ex{ &&&&\be\ar@{->}[dr] \\S_r \ar@{->}[r] & S_{r+1} \ar@{->}[r] & \cdots \ar@{->}[r]& S_{n-2}
\ar@{->}[ur]\ar@{->}[dr] && N_{n-2} \ar@{->}[r]& N_{n-3} \ar@{->}[r] & \cdots \ar@{->}[r] & N_s\\
&&&&\al\ar@{->}[ur] }} \text{ where }$$
\begin{itemize}
\item $ \raisebox{1.8em}{\xymatrix@C=2.5ex@R=0.5ex{ &&&&\be
\\ S_{r} \ar@{->}[r] & S_{r+1} \ar@{->}[r] & \cdots \ar@{->}[r]& S_{n-2} \ar@{->}[ur]\ar@{->}[dr] \\
&&&&\al}}$ is an $S$-sectional path ($\inp_1(S_l)=l$),
\item $\raisebox{1.8em}{\xymatrix@C=2.5ex@R=0.5ex{ \be \ar@{->}[dr] \\
& N_{n-2} \ar@{->}[r]& N_{n-3} \ar@{->}[r] & \cdots \ar@{->}[r] & N_{s}\\
\al\ar@{->}[ur]}}$ is an $N$-sectional path ($\inp_1(N_l)=l$),
\item $\be$ is located at $(n-1,u)$ and $\al$ is located at $(n,u)$ for some $u \in \Z$.
\end{itemize}
\item[({\rm f})] A swing $\varrho$ is {\it maximal} if there is no bigger swing
containing all positive roots in $\varrho$.
\end{enumerate}
\end{definition}

\begin{lemma}\cite[Lemma 2.11]{B99},\cite[\S 3.2]{KKK13b} \label{Lem: Key Lem}
\begin{enumerate}
\item[({\rm a})]
Let $k  \in I_0$ be a source, sink, left intermediate or right intermediate in $Q$. Then we have
$$\inp(\al_k) = \begin{cases} (k,\xi_k) & \text{ if $k$ is a source}, \\
(k^*, \xi_{k^*}-2m_{k^*})& \text{ if $k$ is a sink}, \\
(1, \xi_k-k+1) & \text{ if $k$ is a left intermediate, } \\
(1, \xi_k-2n+k+3) & \text{ if $k$ is a right intermediate.} \end{cases}$$
\item[({\rm b})] Assume that $n-2$ is neither a source, sink, left intermediate nor right intermediate in $Q$.
For $\{a,b\}=\{n-1,n\}$, we have
$$\inp(\al_{n-2}) =
\begin{cases}
(a^*, \xi_{n-2}-2n+5) & \text{ if }\raisebox{1em}{\xymatrix@R=0.5ex{
&&*{\bullet}<3pt> \ar@{}[dl]^<{a}  \\
*{ \bullet }<3pt> \ar@{->}[r]_<{n-3}  &*{\bullet}<3pt>
\ar@{<-}[dr]_<{n-2} \ar@{->}[ur] \\  &&*{\bullet}<3pt>
\ar@{-}[ul]^<{\ \ b}}},\\
(a, \xi_{n-2}-1) & \text{ if }\raisebox{1em}{\xymatrix@R=0.5ex{
&&*{\bullet}<3pt> \ar@{->}[dl]^<{a}  \\
*{ \bullet }<3pt> \ar@{<-}[r]_<{n-3}  &*{\bullet}<3pt>
\ar@{->}[dr]_<{n-2}\\   &&*{\bullet}<3pt>
\ar@{-}[ul]^<{\ \ b}}}.
\end{cases}$$
\item[({\rm c})] If $\be \to \al \in (\AR)_1$, then $(\al,\be)=1$.
\item[({\rm d})] For all $i,j \in I_0$,
\begin{align} \label{eq: range}
(i,\xi_j-d(i,j)),\ (i,\xi_j-2m_j+d(i,j))\in \Gamma.
\end{align}
\end{enumerate}
\end{lemma}
\begin{proof}
The proofs are similar with the ones in \cite[Lemma 1.7]{Oh14A}.
\end{proof}

The following lemma comes from \eqref{eq: Nakayama equation}, \eqref{eq: m n-1 m n} and the involution $^*$ on $I_0$:

\begin{lemma} \label{Lem: image}
We have the $m_{n-1}$ and $m_n$ in \eqref{Def: m_i} as follows:
$$
\begin{cases}
m_{n-1}=n-3, \ m_n=n-1 & \text{ if } n \equiv 1 \ ({\rm mod} \ 2) \text{ and } \xi_{n}=\xi_{n-1}+2,
\\
m_{n-1}=n-1, \ m_n=n-3 & \text{ if } n \equiv 1 \ ({\rm mod} \ 2) \text{ and }\xi_{n-1}=\xi_{n}+2,\\
m_{n-1}=m_n=n-2 & \text{ otherwise. }
\end{cases}
$$
\end{lemma}

\begin{lemma} \label{Lem: last 2 line}
For $\al,\be \in \Phi^+$ with $\inp(\al)=(n-1,k)$ and $\inp(\be)=(n,k)$, there exists $1 \le a \le n-1$ such that
$$ \al+\be = 2 \varepsilon_a \quad \text{ and } \quad
\{ \al, \be \} = \begin{cases} \{ \lf a,n \rf , \lf a,-n \rf \} & \text{ if } \xi_n - \xi_{n-1} =0, \\
 \{ \lf a,n-1 \rf , \lf a,-n+1 \rf  \} & \text{ if } \xi_n - \xi_{n-1} = \pm 2.\end{cases}$$
Moreover, if $\inp(\al')=(n-1,k\pm1), \inp(\be')=(n,k\pm1)
\in \Prt $, then we have
\begin{equation} \label{eq: 2 periodic}
 \al+\al' \in \Prt \text{ and } \be+\be' \in \Prt.
\end{equation}
\end{lemma}

\begin{proof}
\textbf{(Case a: when $\xi_n - \xi_{n-1} =0$)} In this case, $\{n-1,n\}$ are both sources or sinks. Assume that they are sinks and $n$ is odd.
By Lemma \ref{Lem: Key Lem} (a) and Lemma \ref{Lem: image}, the situation can be described as follows:
$$ \scalebox{0.8}
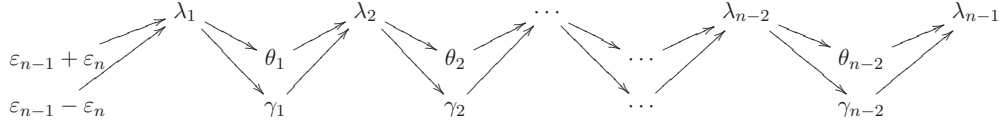
{\xymatrix@R=1ex{
& \lambda_1\ar@{->}[dr]\ar@{->}[ddr] && \lambda_2\ar@{->}[dr]\ar@{->}[ddr]&& \cdots\ar@{->}[dr]\ar@{->}[ddr]&&\lambda_{n-2}\ar@{->}[dr]\ar@{->}[ddr]&&\lambda_{n-1}\\
 \varepsilon_{n-1}+\varepsilon_{n}\ar@{->}[ur]&& \theta_1\ar@{->}[ur]&& \theta_2\ar@{->}[ur] && \cdots\ar@{->}[ur]&& \theta_{n-2}\ar@{->}[ur] && \\
 \varepsilon_{n-1}-\varepsilon_{n}\ar@{->}[uur]&& \ga_1\ar@{->}[uur]&& \ga_2\ar@{->}[uur]&& \cdots\ar@{->}[uur]&& \ga_{n-2}\ar@{->}[uur]
}}
$$
By \eqref{eq: ad fn}, we have $\varepsilon_{n-1}+\varepsilon_{n}+\theta_1=\lambda_1=\varepsilon_{n-1}-\varepsilon_{n}+\ga_1$ and hence $\ga_1-\theta_1=2\ve_n$. Thus
$$ \ga_1=\ve_{a}+\ve_{n} \quad \text{ and } \quad  \theta_1=\ve_{a}-\ve_{n} \quad \text{ for some } \quad a \le n-2. $$

Using this argument successively, we can obtain our first assertion. The
second assertion follows from the additive property of dimension vectors \eqref{eq: ad fn}.
The remaining cases can be proved in the similar way. \\
\noindent
\textbf{(Case b: when $\xi_n - \xi_{n-1} =\pm 2$)} In this case, one of $n-1$ and $n$ is a sink and the another is a source. Assume that $n$ is a sink and
$n$ is odd. Then the neighborhood of $n-2$ in $Q$ can be drawn as follows:
$$\raisebox{1.3em}
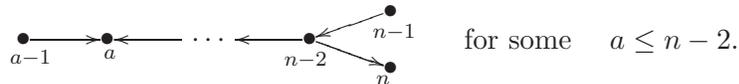
{\xymatrix@R=0.1ex{
&&&&*{\bullet}<3pt> \ar@{->}[dl]^<{n-1}  \\
*{\bullet}<3pt>\ar@{->}[r]_<{a-1}&*{\bullet}<3pt>\ar@{<-}[r]_<{a \ } &\cdots\ar@{<-}[r]  &*{\bullet}<3pt>
\ar@{->}[dr]_<{n-2} \\  &&&&*{\bullet}<3pt>
\ar@{-}[ul]^<{\ \ n}}} \quad \text{ for some } \quad a \le n-2.$$
Then we have
\begin{align*}
& \Dim \mtP(n-2)= (\al_a+\cdots+\al_{n-3})+\al_{n-2}+\al_n=\ve_a+\ve_n,\\
& \Dim \mtP(n-1)= (\al_a+\cdots+\al_{n-3})+\al_{n-2}+\al_{n-1}+\al_n=\ve_a+\ve_{n-1},\\
& \Dim \mtP(n)= \al_n=\ve_{n-1}+\ve_n, \ \Dim \mtI(n-1)= \al_{n-1}=\ve_{n-1}-\ve_n.
\end{align*}
By Lemma \ref{Lem: Key Lem} (a), the situation can be described as follows:
$$ \scalebox{0.8}
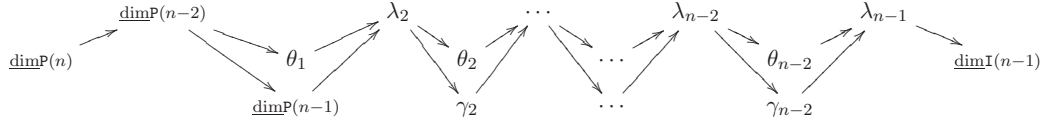
{\xymatrix@C=3ex@R=1ex{
& {\scriptstyle\Dim \mtP(n-2)}\ar@{->}[dr]\ar@{->}[ddr] && \lambda_2\ar@{->}[dr]\ar@{->}[ddr]&& \cdots\ar@{->}[dr]\ar@{->}[ddr]&&
\lambda_{n-2}\ar@{->}[dr]\ar@{->}[ddr]&&\lambda_{n-1}\ar@{->}[dr]\\
 {\scriptstyle\Dim \mtP(n)}\ar@{->}[ur]&& \theta_1\ar@{->}[ur]&& \theta_2\ar@{->}[ur] && \cdots\ar@{->}[ur]&& \theta_{n-2}\ar@{->}[ur]
&& {\scriptstyle\Dim \mtI(n-1)}\\
&& {\scriptstyle\Dim \mtP(n-1)}\ar@{->}[uur]&& \ga_2\ar@{->}[uur]&& \cdots\ar@{->}[uur]&& \ga_{n-2}\ar@{->}[uur]
}}
$$
Thus $\theta_1=\ve_a-\ve_{n-1}$ by the additive property on dimension vectors.
Using the arguments in \textbf{(Case a)}, one can easily verify our assertions. The remaining cases can be proved in
the similar way.
\end{proof}

Throughout this paper, we denote by $\mathtt{t} \in \{ n-1,n \}$ which is determined in Lemma \ref{Lem: last 2 line} and $t'$ the element in
$\{ n-1,n \} \setminus \{ \mathtt{t} \}$; i.e.,
\begin{align} \label{eq: def t}
\mathtt{t} \seq \begin{cases} n-1 \\  n \end{cases} \quad \text{ and } \quad \mathtt{t}'\seq \begin{cases} n & \text{ if } \xi_{n-1}-\xi_{n}=\pm 2,  \\
n-1 & \text{ if } \xi_{n-1}-\xi_{n}=0. \end{cases}
\end{align}

By Lemma \ref{Lem: last 2 line}, we can notice that all roots
containing $\pm \ve_{\mathtt{t}}$ as their summand appear in the level $n-1$ and $n$.

\begin{lemma} \label{Lem: dia sink source} \hfill
\begin{enumerate}
\item[({\rm a})] Let $k \le n-2$ be a source in $Q$ and $\rho$ be the maximal $N$-sectional path containing the simple root $\al_k$. Then all roots in $\rho$
contain $\ve_k$ as their summand and $\rho$ can be drawn as follows:
$$\raisebox{1.8em}{\scalebox{0.9}{\xymatrix@C=4ex@R=0.5ex{ \ve_k \pm \ve_{\mathtt{t}}\ar@{->}[dr] \\
& N_{n-2} \ar@{->}[r]& N_{n-3} \ar@{->}[r] & \cdots \ar@{->}[r] & N_{k}=\al_k, \quad \text{where } \inp_1(N_l)=l.\\
\ve_k \mp \ve_{\mathtt{t}}\ar@{->}[ur] }}}$$
\item[({\rm b})] Let $k \le n-2$ be a sink in $Q$ and $\rho$ be the maximal $S$-sectional path containing the simple root $\al_k$. Then all roots in $\rho$
contain $\ve_k$ as their summand  and $\rho$ can be drawn as follows:
$$
 \raisebox{1.8em}{\scalebox{0.9}{\xymatrix@C=4ex@R=0.5ex{ &&&&\ve_k \pm \ve_{\mathtt{t}}
\\ \al_k=S_{k} \ar@{->}[r] & S_{k+1} \ar@{->}[r] & \cdots \ar@{->}[r]& S_{n-2} \ar@{->}[ur]\ar@{->}[dr] &&, \quad \text{where } \inp_1(S_l)=l. \\
&&&&\ve_k \mp \ve_{\mathtt{t}} }}}
$$
\end{enumerate}
\end{lemma}
\begin{proof} Assume that $k$ is a source. Then the neighborhood of $k$ in $Q$ can be drawn as
$$\xymatrix@R=1ex{*{\bullet}<3pt>
\ar@{<-}[r]_<{k-1}  &*{\bullet}<3pt>
\ar@{->}[r]_<{k} &*{\bullet}<3pt>
\ar@{-}[l]^<{\ \ k+1} } \quad \text{ or } \quad
\raisebox{1em}{\xymatrix@R=0.5ex{
&&*{\bullet}<3pt> \ar@{}[dl]^<{n}  \\
*{ \bullet }<3pt> \ar@{<-}[r]_<{n-3}  &*{\bullet}<3pt>
\ar@{->}[dr]_<{n-2} \ar@{->}[ur]\\  &&*{\bullet}<3pt>
\ar@{-}[ul]^<{\ \ n-1}}} \quad \text{ if $k=n-2$. }
$$
For $a \le k-1$ with a path from $a$ to $k-1$ in $Q$, denoted by $a \overset{\text{path}}{\longrightarrow} k-1$, we have
\begin{align*}
\Dim \mtI(k-1) & = \al_a+\cdots+\al_k=\ve_a-\ve_{k+1} \qquad \qquad \qquad \   \  \text{if $k-1$ exists, }\\
\Dim \mtI(k+1) &  = \begin{cases}
{\rm (i)} \ \al_k+\cdots+\al_b=\ve_k-\ve_{b+1}  \\
 \hs{1ex}\text{ if } {}^\exists b \to b+1 \ (b \le k+1 < n-2) \text{ and
${}^\exists k+1 \overset{\text{path}}{\longleftarrow} b$ in $Q$},\\
{\rm (ii)} \ \al_k+\cdots+\al_{n-2}+\al_n = \ve_k+\ve_n \\
\hs{10ex}
\text{ if ${}^\exists$ $k+1 \overset{\text{path}}{\longleftarrow} n$ and ${}^{\not\exists}$ $k+1 \overset{\text{path}}{\longleftarrow} n-1$ in $Q$}, \\
{\rm (iii)} \ \al_k+\cdots+\al_{n-2}+\al_{n-1} = \ve_k-\ve_n \\
\hs{10ex} \text{ if ${}^\exists$ $k+1 \overset{\text{path}}{\longleftarrow} n-1$ and ${}^{\not\exists}$ $k+1 \overset{\text{path}}{\longleftarrow} n$ in $Q$}, \\
{\rm (iv)} \ \al_k+\cdots+\al_{n-2}+\al_{n-1} = \ve_k+\ve_{n-1} \\
\hs{10ex} \text{ if ${}^\exists$ $k+1 \overset{\text{path}}{\longleftarrow} n-1$ and ${}^{\exists}$ $k+1 \overset{\text{path}}{\longleftarrow} n$ in $Q$},
\end{cases} \\
\Dim \mtI(k) & = \al_k=\ve_{k-1}-\ve_k.
\end{align*}
By Lemma \ref{Lem: Key Lem} (d), the following subquiver is contained in $\AR$:
\begin{equation} \label{eq: subquiver 1}
\scalebox{0.8}{\xymatrix@R=0.5ex{
&&&&&  { C_0=\Dim \mtI(k-1)}\ar@{->}[dr] \\
&&&& { C_1}\ar@{->}[ur]\ar@{->}[dr]  && \al_k \\
&&& { C_{n-k-1}}\ar@{.>}[ur]\ar@{->}[dr] && { M_1=\Dim \mtI(k+1)}\ar@{->}[ur] \\
&& { C_{n-k}} \ar@{->}[ur]\ar@{->}[dr]\ar@{->}[ddr]  && { M_{n-k-1}}\ar@{.>}[ur] \\
{ i=n-1}& { C_{n-k+1} } \ar@{->}[ur] && { M_{n-k}}\ar@{->}[ur] \\
{ i=n} & { C'_{n-k+1} } \ar@{->}[uur] && { M'_{n-k}}\ar@{->}[uur]
}}
\end{equation}
Note that $C_1= \ve_a-\ve_{b+1}$, $\ve_a+\ve_n$, $\ve_a-\ve_{n}$ or $\ve_a+\ve_{n-1}$ corresponding to
${\rm (i)}$, ${\rm (ii)}$, ${\rm (iii)}$ and ${\rm (iv)}$, respectively.

From the subquiver \eqref{eq: subquiver 1}, the additive property of dimension vectors \eqref{eq: ad fn} tells that
\begin{itemize}
\item $ C_0- \al_k=\ve_a-\ve_k= C_1-M_1= \cdots = C_{n-k-1}-M_{n-k-1}$,
\item $C_{n-k-1}+M_{n-k}+M'_{n-k}=C_{n-k}+M_{n-k-1}$,
\item $C_{n-k}=C_{n-k-1}+M_{n-k}=C'_{n-k-1}+M'_{n-k}$.
\end{itemize}
Hence we conclude that
$\ve_a-\ve_k =C_{n-k+1}-M'_{n-k}=C'_{n-k+1}-M_{n-k}.$
By Lemma \ref{Lem: last 2 line}, we can conclude that
$M_{n-k}$ and $M'_{n-k}$ contain $\ve_k$ as their summand.

On the other hand, $\ve_a-\ve_k + M_2=C_2$ and hence $$ M_2 = \ve_a'-\ve_a \quad \text{ or } \quad \ve_k \pm \ve_c \quad \text{for some $a'<a<c$.}$$
By Lemma \ref{Lem: Key Lem} (c), $ M_2$ should be of the form $\ve_k \pm \ve_c$. In this way, we can
conclude that $ M_i$ for $2 \le i \le n-k-1$ contain $\ve_k$ as their summand, which yields our first assertion. For the case when
$k$ is a sink, we can apply the similar argument by observing $\Dim \mtP((k-1)^*)$, $\Dim \mtP(k^*)$ and $\Dim \mtP((k+1)^*)$.
\end{proof}

\begin{proposition} \label{prop: triangle}
Let $\inp(\al)=(n',s)$ and $\inp(\be)=(n'',l)$ such that
$$n',n'' \in \{ n-1,n\}, \quad |s-l|=2k \quad \text{ and } \quad n'-n''\equiv k-1 \ ({\rm mod} \ 2),$$
for some $k \in \Z_{\ge 1}$. Then we have
$$ \inp(\al+\be) = (n-1-k, \dfrac{s+l}{2}) \in \AR.$$
\end{proposition}
\begin{proof}
Consider the subquiver $\widetilde{\Gamma}$ of $\Z Q$ with setting $\al=\theta_1$ or $\ga_1$, and $\be=\theta_{k+1}$ or $\ga_{k+1}$:
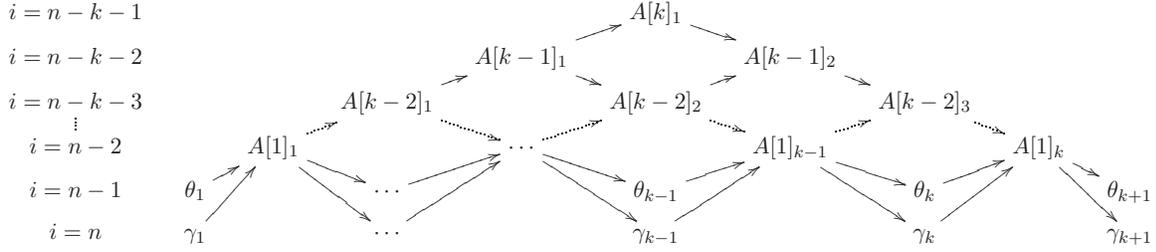
\begin{equation} \label{figure: nfree roots}
\scalebox{0.8}{\xymatrix@C=2.5ex@R=0.5ex{
{i=n-k-1}&&&&& A[k]_{1}\ar@{->}[dr] \\
{i=n-k-2}&&&& A[k-1]_{1}\ar@{->}[ur]\ar@{->}[dr] && A[k-1]_{2}\ar@{->}[dr]\\
{i=n-k-3}&&& A[k-2]_{1}\ar@{.>}[dr]\ar@{->}[ur] && A[k-2]_{2}\ar@{.>}[dr]\ar@{->}[ur]&& A[k-2]_{3}\ar@{.>}[dr]\\
{i=n-2}\ar@{.}[u]&& A[1]_{1}\ar@{->}[dr]\ar@{->}[ddr]\ar@{.>}[ur] && \cdots \ar@{->}[dr]\ar@{->}[ddr]\ar@{.>}[ur]
&& A[1]_{k-1}\ar@{->}[dr]\ar@{->}[ddr]\ar@{.>}[ur] && A[1]_{k}\ar@{->}[dr]\ar@{->}[ddr] \\
{i=n-1}& \theta_{1}\ar@{->}[ur] && \cdots\ar@{->}[ur] && \theta_{k-1}\ar@{->}[ur] && \theta_k \ar@{->}[ur]&& \theta_{k+1} \\
{i=n} &  \ga_{1}\ar@{->}[uur]&& \cdots\ar@{->}[uur] && \ga_{k-1}\ar@{->}[uur] && \ga_k\ar@{->}[uur] && \ga_{k+1}
}}
\end{equation}
If $k=1$, it is trivial. Assume that $k > 1$ and $k$ is even. By \eqref{eq: known characterization} and an induction on $k$,
\begin{align*}
& A[k-2]_2= \theta_2+\theta_k=\ga_2+\ga_k \in \Prt, \\
& A[k-1]_1= \theta_1+\ga_k=\ga_1+\theta_k \in \Prt, \\
& A[k-1]_2= \theta_2+\ga_{k+1}=\ga_2+\theta_{k+1} \in \Prt.
\end{align*}
Thus all $A[i]_j \in \widetilde{\Gamma}$ $(i < k)$ are elements in $\Prt$ and
$$A[k]_1=A[k-1]_1+A[k-1]_2-A[k-2]_2=\theta_1+\theta_{k+1}=\ga_1+\ga_{k+1} \in \Prt$$
by \eqref{eq: 2 periodic}. For the case when $k$ is odd, one can prove in the similar way.
\end{proof}

For $\be=\sum_{i \in I} k_i \al_i \in
\Prt$, the {\it height} of $\be$ is defined by $\het(\be)=\sum_{i
\in I} k_i$.

\begin{corollary} \label{cor: nfree position} Set
\begin{align*}
& i=\max\{ \inp_2(\be) \mid \ \inp_1(\be) \in \{ n-1,n \}, \ \het(\be)\ge 2 \}, \\
& j=\min\{ \inp_2(\be) \mid \ \inp_1(\be) \in \{ n-1,n \}, \ \het(\be)\ge 2 \}.
\end{align*}
Then we have the followings:
\begin{enumerate}
\item[({\rm a})] $i-j=2(n-3)$.
\item[({\rm b})] Every multiplicity non-free positive root $\be$ satisfies the following conditions:
\begin{align} \label{eq: nfree root condition}
 1 < \ell \seq \inp_1(\be) < n-1 \text{ and }j-(n-1-\ell)\le  \inp_2(\be) \le i-(n-1-\ell).
\end{align}
\end{enumerate}
\end{corollary}

\begin{proof} Note that the number of multiplicity non-free positive roots is $(n-3)(n-2)/2$.
By Lemma \ref{Lem: last 2 line}, a simple root $\al$ with $\inp_1(\al) \in  \{n-1,n\}$ is $\al_{n-2}$, $\al_{n-1}$ or $\al_n$.
Then the first assertion follows from Lemma \ref{Lem: Key Lem} (a) and Lemma \ref{Lem: image}.
For each multiplicity non-free positive root $\be=\ve_a+\ve_b$ $(b \le n-2)$,
there exist two pairs of roots $\{ \ve_a-\ve_{\mathtt{t}},\ve_b+\ve_{\mathtt{t}}\}$ and $\{ \ve_a+\ve_{\mathtt{t}},\ve_b-\ve_{\mathtt{t}}\}$ such that their sums are $\be$ and they are located
at level $n-1$ or $n$. By setting $k=n-3$, $\inp_2(a_1)=j$ and $\inp_2(a_{n-2})=i$ in \eqref{figure: nfree roots}, our second assertion follows.
\end{proof}

\begin{corollary} \label{Cor: nfree mfree position}
Assume we have two roots $\al$ and $\be$ in the same sectional path such that $\al$ is multiplicity free located in the level $k < n-1$ and
$\be$ is not multiplicity free. Then $$ \inp_1(\al) < \inp_1(\be).$$
\end{corollary}

\begin{proof}
Recall that $\be$ is contained in the area \eqref{eq: nfree root condition} and $\al$ is not. Thus our assertion follows.
\end{proof}

\begin{corollary} \label{cor: n,n-1} \begin{enumerate}
\item[({\rm a})] Assume that $\xi_n=\xi_{n-1}$. Then there exists the maximal sectional path $\rho$ which can be drawn as follows:
\begin{equation} \label{eq: both ss}
\begin{cases}
\raisebox{1.8em}{\scalebox{0.9}{\xymatrix@C=4ex@R=0.5ex{ \al_{n-1}\ar@{->}[dr] \\
& N_{n-2} \ar@{->}[r]& N_{n-3} \ar@{->}[r] & \cdots \ar@{->}[r] & N_{1}\\
\al_n\ar@{->}[ur]}}} & \text{ if $\{n-1,n\}$ are sinks}, \\
 \raisebox{1.8em}{\scalebox{0.9}{\xymatrix@C=4ex@R=0.5ex{ &&&&\al_{n-1}
\\ S_{1} \ar@{->}[r] & S_{2} \ar@{->}[r] & \cdots \ar@{->}[r]& S_{n-2} \ar@{->}[ur]\ar@{->}[dr] && \\
&&&&\al_{n} }}} & \text{ if $\{n-1,n\}$ are sources.}
\end{cases}
\end{equation}
Here all roots in $\rho$ contain $\ve_{n-1}$ as their summand and $\inp_1(S_l)=\inp_1(N_l)=l$ $(1 \le l \le k-2)$.
\item[({\rm b})] Assume that $|\xi_n-\xi_{n-1}|=2$. Then there exists the maximal sectional pathes $\rho$ and $\rho'$ which can be drawn as follows:
\begin{equation} \label{eq: one ss}
\begin{cases}
\overbrace{\al_{n-1}\to N_{n-2} \to \cdots N_{2} \to N_1}^{\text{ $\rho=$ maximal $N$-sectional path}} \cdots
\underbrace{S_{1} \to S_{2} \to \cdots \to S_{n-2}
\to\al_{n}}_{\text{$\rho'=$maximal $S$-sectional path}} & \text{ if }\xi_n-\xi_{n-1}= 2,  \\
\overbrace{\al_{n}\to N_{n-2} \to \cdots N_{2} \to N_1}^{\text{ $\rho=$ maximal $N$-sectional path}} \cdots
\underbrace{S_{1} \to S_{2} \to \cdots \to S_{n-2}
\to\al_{n-1}}_{\text{$\rho'=$maximal $S$-sectional path}} & \text{ if }\xi_{n-1}-\xi_{n}= 2,
\end{cases}
\end{equation}
where
\begin{itemize}
\item $\inp_1(S_l)=\inp_1(N_l)=l$ $(1 \le l \le k-2)$,
\item $\inp_2(N_1)+2=\inp_2(S_1)$,
\item $N_l$ contains $\ve_n$ as its summand and $S_l$ contains $-\ve_n$ as its summand $(1 \le l \le k-2)$.
\end{itemize}
\end{enumerate}
\end{corollary}

\begin{proof}
We only give a proof when $\{n-1,n\}$ are sinks. The remaining cases can be proved using the argument in this proof. By Lemma \ref{Lem: Key Lem}, Lemma \ref{Lem: image} and Proposition \ref{prop: triangle}, we know that
there is a subquiver in $\AR$ as follows:
$$\scalebox{0.75}
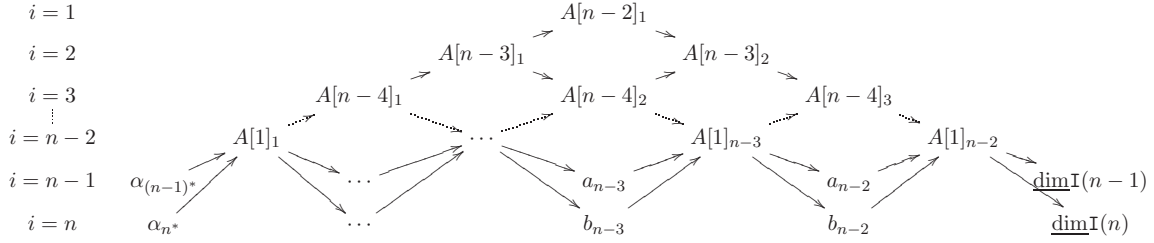
{\xymatrix@C=2ex@R=0.5ex{
{i=1}&&&&& A[n-2]_{1}\ar@{->}[dr] \\
{i=2}&&&& A[n-3]_{1}\ar@{->}[ur]\ar@{->}[dr] && A[n-3]_{2}\ar@{->}[dr]\\
{i=3}&&& A[n-4]_{1}\ar@{.>}[dr]\ar@{->}[ur] && A[n-4]_{2}\ar@{.>}[dr]\ar@{->}[ur]&& A[n-4]_{3}\ar@{.>}[dr]\\
{i=n-2}\ar@{.}[u]&& A[1]_{1}\ar@{->}[dr]\ar@{->}[ddr]\ar@{.>}[ur] && \cdots \ar@{->}[dr]\ar@{->}[ddr]\ar@{.>}[ur]
&& A[1]_{n-3}\ar@{->}[dr]\ar@{->}[ddr]\ar@{.>}[ur] && A[1]_{n-2}\ar@{->}[dr]\ar@{->}[ddr] \\
{i=n-1}& \al_{(n-1)^*}\ar@{->}[ur] && \cdots\ar@{->}[ur] && a_{n-3}\ar@{->}[ur] && a_{n-2} \ar@{->}[ur]&& \Dim \mtI(n-1) \\
{i=n} &  \al_{n^*}\ar@{->}[uur]&& \cdots\ar@{->}[uur] && b_{n-3}\ar@{->}[uur] && b_{n-2}\ar@{->}[uur] &&  \Dim \mtI(n)
}}$$
where $\{a_{k},b_{k}\}=\{ \ve_{i_k} - \ve_{n}, \ve_{i_k} + \ve_{n} \}$ for some $i_k \le n-2$. By Corollary \ref{cor: n,n-1},
$$ A[\ell]_1= \al_{n-1} + (\ve_{i_{\ell+1}}+\ve_n ) = \al_{n} + (\ve_{i_{\ell+1}}-\ve_n )=\ve_{i_{\ell+1}}+\ve_{n-1} \quad (1 \le \ell \le n-2),$$
which yields our assertion.
\end{proof}

\begin{proposition} \label{Prop: maximal path} \hfill
\begin{enumerate}
\item[({\rm a})] For every maximal $S$-sectional path $\rho$ which ends at level $n$ and $n-1$, there exists
$k \le n-2+\delta$ such that all roots in $\rho$ contain $\ve_k$ as their summand.
Here $\delta=1$ if $\{n-1,n\}$ are sources, and $\delta=0$ otherwise.
\item[({\rm b})] For every maximal $N$-sectional path  $\rho$ which starts at level $n$ and $n-1$,
there exists $k \le n-2+\delta$ such that all roots in $\rho$ contain $\ve_k$ as their summand.
Here $\delta=1$ if $\{n-1,n\}$ are sinks, and $\delta=0$ otherwise.
\end{enumerate}
\end{proposition}

\begin{proof} We only give a proof for the first assertion.
Take a subset $E$ of $\{ 1,2,\cdots,n-2 \}$ defined as follows:
$$a \in E \iff \xi_a=\min\{\xi_i \ | \ 1 \le i \le n-2\}.$$
Note that maximal $S$-sectional paths ending at $n-1$ and $n$ are adjacent to each other. \\
\textbf{(Case A: $n-2 \ne a \in E$  )} Then we can check that $a$ is a sink in $Q$.
By Lemma \ref{Lem: dia sink source}, the maximal $S$-sectional path $\rho_a$ containing $\al_a$ satisfies our assertion.

Let $\rho$ be a maximal $S$-sectional path which ends at level $n$ and $n-1$, and is located at the right of $\rho_a$.
By an induction on the distance from $\rho_a$,
all roots in the maximal $S$-sectional path left adjacent to $\rho$ contain $\ve_b$ as their summand.
Thus, by Lemma \ref{Lem: last 2 line}, the situation in $\AR$ can be drawn as follows:
$$\scalebox{0.8}
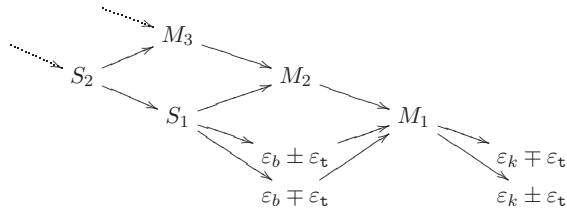
{\xymatrix@R=0.5ex{
&&{}^{ \ \ \ } \ar@{.>}[dr]&& \\
& {}^{ \ \ \ } \ar@{.>}[dr]&&  M_3\ar@{->}[dr]  \\
&&  S_2\ar@{->}[dr]\ar@{->}[ur]&& M_2\ar@{->}[dr] \\
&&&  S_1 \ar@{->}[dr]\ar@{->}[ur]\ar@{->}[ddr]&& M_1\ar@{->}[dr]\ar@{->}[ddr] \\
&&&&   \ve_b \pm \ve_{\mathtt{t}}\ar@{->}[ur] && \ve_k \mp \ve_{\mathtt{t}} \\
&&&&  \ve_b \mp \ve_{\mathtt{t}}\ar@{->}[uur]&& \ve_k \pm \ve_{\mathtt{t}} \\
}}$$
where $1\le k \le n-2$ and all $S_i$ contain $\ve_b$ as their summand. Since $|\xi_j-\xi_{j\pm 1}|=1$ for all $j \in I$,
we can see that all $M_i$ contain $\ve_k$ as their summand by the additive property
of dimension vectors. \\
\noindent
\textbf{(Case B: $\{n-2\}=E$  )}  If $n-2$ is a sink, then we can apply the same argument in \textbf{(Case A)}. Assume that
$n-2$ is not a sink. By our choice, we have the following subquiver $\widetilde{\Gamma}$ in $\Z Q$
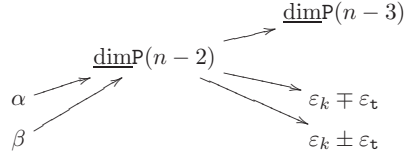
$$\scalebox{0.8}{\xymatrix@R=0.5ex{
&& \Dim\mtP(n-3) \\
& \Dim\mtP(n-2)\ar@{->}[ur]\ar@{->}[dr]\ar@{->}[ddr] \\
 \al \ar@{->}[ur] && \ve_k \mp \ve_{\mathtt{t}} \\
 \be  \ar@{->}[uur]&& \ve_k \pm \ve_{\mathtt{t}} \\
}}
$$
where $\al$ or $\be$ in $\Prt$, and $1 \le k \le n-2$. By the additive property of dimension vectors,
$\Dim\mtP(n-2)$ contains $\ve_k$ as its summand. Thus all roots in the maximal $S$-sectional path containing $\Dim\mtP(n-2)$
$$\scalebox{0.8}
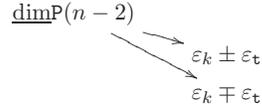
{\xymatrix@C=4ex@R=0.5ex{ \Dim\mtP(n-2)\ar@{->}[dr]\ar@{->}[ddr] \\ &\ve_k \pm \ve_{\mathtt{t}} \\ &\ve_k \mp \ve_{\mathtt{t}} }}$$
contain $\ve_k$ as their summand. Thus we can apply the same argument in \textbf{(Case A)} also.
\end{proof}

\begin{theorem} \label{Thm: V-swing}
 For every maximal swing $\varrho$, there exists $1 \le k \le n-2$ such that
all roots in $\varrho$ contain $\ve_k$ as their summand. Moreover, $\varrho$ contains a simple root $\al_k$ and
is one of the following two forms:
\begin{align}
& \raisebox{1.8em}{\scalebox{0.9}{\xymatrix@C=3ex@R=0.5ex{ &&&&\ve_k \pm \ve_{\mathtt{t}}\ar@{->}[dr]
\\ \Dim \mtP(k)=S_{k} \ar@{->}[r] & S_{k+1} \ar@{->}[r] & \cdots \ar@{->}[r]& S_{n-2} \ar@{->}[ur]\ar@{->}[dr] &&
N_{n-2} \ar@{->}[r]& N_{n-3} \ar@{->}[r] & \cdots \ar@{->}[r] & N_{1}\\
&&&&\ve_k \mp \ve_{\mathtt{t}}\ar@{->}[ur] }}}, \label{eq: type a} \\
&\quad \ \  \raisebox{1.8em}{\scalebox{0.9}{
\xymatrix@C=3ex@R=0.5ex{ &&&&\ve_k \pm \ve_{\mathtt{t}}\ar@{->}[dr] \\S_{1} \ar@{->}[r] & S_{2} \ar@{->}[r] & \cdots \ar@{->}[r]& S_{n-2} \ar@{->}[ur]\ar@{->}[dr] &&
N_{n-2} \ar@{->}[r]& N_{n-3} \ar@{->}[r] & \cdots \ar@{->}[r] & N_{k}=\Dim \mtI(k)\\
&&&&\ve_k \mp \ve_{\mathtt{t}}\ar@{->}[ur] }}}. \label{eq: type b}
\end{align}
\end{theorem}

\begin{proof} By Lemma \ref{Lem: last 2 line} and Proposition \ref{Prop: maximal path}, every maximal swing is of the following form
$$\raisebox{2em}{\xymatrix@C=3ex@R=0.5ex{ &&&&\ve_k-\ve_{\mathtt{t}} \ar@{->}[dr] \\S_r \ar@{->}[r] & S_{r+1} \ar@{->}[r] & \cdots \ar@{->}[r]& S_{n-2}
\ar@{->}[ur]\ar@{->}[dr] && N_{n-2} \ar@{->}[r]& N_{n-3} \ar@{->}[r] & \cdots \ar@{->}[r] & N_s\\
&&&&\ve_k+\ve_{\mathtt{t}}\ar@{->}[ur] }}$$
where  $k \le n-2$ and all $S_i$ ($1 \le u \le r$) and $N_j$ ($1 \le v \le s$) contain $\ve_k$ as their summand.

(\textbf{Case A:} $k$ is a source or sink) Assume that $k$ is a source in $Q$, then $\al_k=N_1$ or $S_1$. We claim that $\al_k=N_1$. By Lemma \ref{Lem: Key Lem},
$\inp_2(\ve_k \pm \ve_{\mathtt{t}})=\xi_k-n+1-k$ or $\xi_k+n-1+k$. However, the later case can not be happen since $\xi_n < \xi_k+n-1+k$
by the assumption that $k$ is a source in $Q$. Thus Proposition \ref{Lem: dia sink source} implies that $\al_k=N_1$ and hence
$\inp_2(\ve_k \pm \ve_{\mathtt{t}})=\xi_k-n+1-k$. Since $\xi_1-2n+4 \le \xi_k-2n+3+k$, our assertion for a source $k$ follows and the maximal swing is of the form
\eqref{eq: type b}. In the similar way, we have our assertion for a sink $k$ and the maximal swing is of the form
\eqref{eq: type a}.

(\textbf{Case B:} $k$ is a left or right intermediate) Assume  that $k$ is a left intermediate in $Q$. Then
Lemma \ref{Lem: Key Lem} and Proposition \ref{Prop: maximal path} tell that $\al_k=N_1$ or $S_1$. By Lemma \ref{Lem: Key Lem} once again,
$\inp_2(\ve_k \pm \ve_{\mathtt{t}})=\xi_k-k+1-(n-2)$ or $\xi_k-k+1+(n-2)$. Since $k$ is a left intermediate,
$\xi_{n-1} < \xi_k+(n-1-k)$ and hence $\al_k=N_1$ and $\inp_2(\ve_k \pm \ve_{\mathtt{t}})=\xi_k-k-n+3$. Note that
$$ \xi_k-k-n+3-(n-1-k)= \xi_k-2n+4 = \xi_k-2 m_k.$$
Since $k$ is a left intermediate, $(k-1,\xi_k-2 m_k-1) \not \in \AR$. Thus our assertion for a left intermediate
$k$ follows and the maximal swing is of the form
\eqref{eq: type a}.  In the similar way, we have our assertion for a right intermediate $k$ and the maximal swing is of the form
\eqref{eq: type b}.

(\textbf{Case C:} remaining cases) For the remaining cases, we can apply the similar argument in \textbf{Case A} and \textbf{Case B} with
Lemma \ref{Lem: Key Lem} (b). Thus one can prove that
\begin{itemize}
\item if $\raisebox{1em}{\xymatrix@R=0.5ex{
&&*{\bullet}<3pt> \ar@{}[dl]^<{a}  \\
*{ \bullet }<3pt> \ar@{->}[r]_<{n-3}  &*{\bullet}<3pt>
\ar@{<-}[dr]_<{n-2}^<{k \ \ \ \ } \ar@{->}[ur] \\  &&*{\bullet}<3pt>
\ar@{-}[ul]^<{\ \ b}}}$, then  the maximal swing is of the form
\eqref{eq: type a},
\item if $ \raisebox{1em}{\xymatrix@R=0.5ex{
&&*{\bullet}<3pt> \ar@{->}[dl]^<{a}  \\
*{ \bullet }<3pt> \ar@{<-}[r]_<{n-3}  &*{\bullet}<3pt>
\ar@{->}[dr]_<{n-2}^<{k \ \ \ \ }\\   &&*{\bullet}<3pt>
\ar@{-}[ul]^<{\ \ b}}}$, then  the maximal swing is of the form
\eqref{eq: type b}.
\end{itemize}
Here $\{ a,b\}=\{n-1,n\}$.
\end{proof}

\begin{remark} \label{rem: swing}
Note that, for each $1 \le a \le n-2$, there are $2n-a-1$ positive roots containing $\ve_a$ as their summand. Thus Theorem
\ref{Thm: V-swing} implies that every root is contained in a maximal swing. With Corollary \ref{cor: n,n-1},
we can say that a maximal swing $\varrho$ is the {\it $a$-swing} if all positive roots having $\ve_a$ as their summand appear in $\varrho$.
\end{remark}

\begin{definition}
For the $a$-swing
$${\scalebox{0.9}{\xymatrix@C=3ex@R=0.5ex{ &&&&\ve_a \pm \ve_{\mathtt{t}}\ar@{->}[dr]
\\ S_{s} \ar@{->}[r] & S_{s+1} \ar@{->}[r] & \cdots \ar@{->}[r]& S_{n-2} \ar@{->}[ur]\ar@{->}[dr] &&
N_{n-2} \ar@{->}[r]& N_{n-3} \ar@{->}[r] & \cdots \ar@{->}[r] & N_{l}\\
&&&&\ve_a \mp \ve_{\mathtt{t}}\ar@{->}[ur] }}}$$
we define
\begin{enumerate}
\item {\it the $S$-part} as $S_s \to S_{s+1} \to \cdots \to S_{n-2}$ and {\it a length of the $S$-part} as $n-2-s$,
\item {\it the $N$-part} as $N_{n-2} \to N_{n-3} \to \cdots \to N_l$ and {\it a length of the $N$-part} as $n-2-l$,
\item the shorter part (resp. longer part) in a canonical way.
\end{enumerate}
\end{definition}

We say that a maximal $S$-sectional (resp. $N$-sectional) path is {\it shallow} if it ends (resp. starts) at level less than $n-1$.

\begin{theorem} \label{thm: short path}
Let $\rho$ be a shallow maximal $S$-sectional $($resp. $N$-sectional$)$ path.
Then there exists $k \le n-2+\delta$ such that all roots in $\rho$ contain $-\ve_k$ as their summand and $\rho$ starts (resp. ends) at level $1$.
Here $\delta=1$ if $\{ n-1,n\}$ are sink or source, $\delta=0$ otherwise.
\end{theorem}

\begin{proof} Assume that $\rho$ is a shallow maximal $S$-sectional path. We claim that it starts at level $1$. If
it does not start at level $1$. Lemma \ref{Lem: Key Lem} (d) implies that it ends at level $n-1$ or $n$.  Thus we can draw $\rho$ as follows:
$$\xymatrix@R=2ex@R=0.5ex{ &S_1 \ar@{->}[r] & S_{2} \ar@{->}[r] & \cdots \ar@{->}[r]& S_{l-1} \ar@{->}[r]& S_l.}$$
Then $S_l=\Dim \mtI(l)$ for some $l \le n-2$. By \eqref{eq: eta zeta}, Corollary \ref{cor: n,n-1} and Remark \ref{rem: swing}, we have
$S_l=\ve_{a_l}-\ve_k$ ($1 \le k \le n-2+\delta$) which is contained in the $a_l$-swing. By the same reason, $S_{l-1}$
is also contained in another $a_{l-1}$-swing ($a_l \ne a_{l-1}$) and
$$S_{l-1} = \ve_{a_{l-1}} \pm \ve_c  \quad \text{ for some } \quad a_{l-1}<c.$$
By Lemma \ref{Lem: Key Lem} (c), $S_{l-1}$ must be the same as $\ve_{a_{l-1}}-\ve_k$. Thus we can prove that all $S_i$ contains $-\ve_k$ successively.
In the similar way, we can prove the case when $\rho$ is a shallow maximal $N$-sectional path.
\end{proof}

For a shallow maximal $N$-sectional (resp. $S$-sectional) path $\rho$, we say that $\rho$ is the shallow maximal $(N,-a)$-sectional
(resp. $(S,-a)$-sectional) path if all of its roots contain $-\ve_a$ as their summand.

\begin{corollary} \label{Cor: Longest} \hfill
\begin{enumerate}
\item[({\rm a})] $1$-component and $2$-component are adjacent in $\AR$.
\item[({\rm b})] The longest root is located at
$$ \inp(\ve_1+\ve_2) =
\begin{cases}
(n-2,\xi_1-n+1)) & \text{ if $1$ is a source}, \\
(n-2,\xi_1-n+3) & \text{ if $1$ is a sink}.
\end{cases}
$$
\end{enumerate}
\end{corollary}

\begin{proof} Note that $1$ is a source or a sink, always. Assume that $1$ is a source. The we have $\inp(\al_1)=(1,\xi_1)$ and
$1$-swing is of the form \eqref{eq: type a}(\eqref{eq: type b} also)  with its $N_{1}=\al_1$.
Note that $2$ is a left intermediate or a sink, since $1$ is a source.

(i) If $2$ is a left intermediate, we have  $\inp(\al_2)=(1,\xi_1-2)$ and $2$-swing is of the form \eqref{eq: type a}
with its $N_{1}=\al_2$. Thus the root located at $(n-2,\xi_1-n+1)$ must contain $\ve_1$ and $\ve_2$ as its
summands, simultaneously. Thus our assertion follows.

(ii) If $2$ is a sink, then $\inp(\al_2)=(2,\xi_1-2n+1)$ and $2$-swing is of the form \eqref{eq: type a} with its $N_{2}=\al_2$.
Thus its $S_1$ is located at $(1,\xi_1-2)$. By the same reason, we can obtain our assertion.

For the case when $1$ is a sink, one can verify by using the similar argument.
\end{proof}

Now, Theorem \ref{Thm: V-swing} and Theorem \ref{thm: short path} provide also a way to compute $\AR$:

\begin{remark} \hfill \label{Alg: easy}
\begin{enumerate}
\item[({\rm a})] Fill out vertices in $\AR$ by using the height function $\xi$ and {\rm Lemma \ref{Lem: image}}.
\item[({\rm b})] Using {\rm Lemma \ref{Lem: Key Lem}} (a) and (b), label the vertices corresponding to simple roots.
\item[({\rm c})] Using {\rm Theorem \ref{Thm: V-swing}} with step (b), we can label the summand of all vertices.
\item[({\rm d})] Using {\rm Lemma \ref{Lem: last 2 line}}, we can complete the labeling of all vertices at level $n-1$ and $n$.
\item[({\rm e})] Using {\rm Theorem \ref{thm: short path}}, we can complete the labeling of all vertices.
\end{enumerate}
\end{remark}

Let $\sigma$ be a subset of $\Phi^{+}_n$ defined as follows:
\begin{equation} \label{eq: sigma}
\sigma \seq \big\{  \be \in \Phi^{+}_n \ | \ \inp_1(\be)=n-1, \  \be \not \in  \{ \al_{n-1},\al_{n}\} \big   \}.
\end{equation}
Then Theorem \ref{Thm: V-swing} tells that $|\sigma|=n-2$ and each element in $\sigma$ is contained in
only one swing.
We set, for $1 \le k \le n-3$,
\begin{eqnarray} &&
\parbox{85ex}{
\begin{itemize}
\item the positive roots in $\sigma=\{\sigma_1,\ldots,\sigma_{n-2}\}$ as $  \inp_{2}(\sigma_{k+1})+2= \inp_{2}(\sigma_{k})$,
\item indices $\{ i_{\sigma_1},i_{\sigma_2},\ldots,i_{\sigma_{n-2}}\}=\{ 1,2,\ldots,n-2 \}$ such that $\sigma_{k}$ is contained in $i_{\sigma_{k}}$-swing.
\end{itemize}}\label{eq:sigma enumeration}
\end{eqnarray}

\begin{corollary} \label{cor: reverse uni} \hfill
\begin{itemize}
\item[({\rm a})] A multiplicity non-free positive root $\lf a, b \rf$ is contained in the longer part of the $b$-swing.
\item[({\rm b})] There exists $1 \le \ell \le n-2$ such that $i_{\sigma_{\ell}}=1$ and
\begin{equation} \label{eq: reverse unimodal}
i_{\sigma_{1}}>i_{\sigma_{2}}>\cdots >i_{\sigma_{\ell}}=1 < i_{\sigma_{\ell-1}} < \cdots < i_{\sigma_{n-2}},
\end{equation}
where the shorter part of $i_{\sigma_{a}}$ $(a < \ell)$ is the $N$-part and the shorter part of $i_{\sigma_{b}}$ $(b > \ell)$ is the $S$-part .
\end{itemize}
\end{corollary}

\begin{proof}
(a) Assume that it is contained in the shorter part, and $S$-part is the shorter part of the $b$-swing.
Then it is contained in the $N$-part of the $a$-swing. Note that
\begin{equation} \label{eq: longer}
\text{the $S$-part and $N$-part of $a$-swing are strictly longer than $S$-part of $b$-swing.}
\end{equation}
Then there is a vertex $(i,p)$ such that
\begin{eqnarray} &&
\parbox{85ex}{
\begin{itemize}
\item $i=\inp_1(\Dim \mtP(b))-1$, $p=\inp_2(\Dim \mtP(b))-1$, (Theorem \ref{Thm: V-swing})
\item $p> \xi_{i}-2(n-2)$ and $(i,p) \not \in \AR$, (by \eqref{eq: longer})
\end{itemize}}\label{eq:cont known char}
\end{eqnarray}
which yields a contradiction to \eqref{eq: known characterization}. By applying the
similar argument, one can complete the proof for the first assertion.

\noindent
(b) Note that there is $1 \le \ell \le n-2$ such that $i_{\sigma_{\ell}}=1$. For $ k < l < \ell$, we first show that
$i_{\sigma_{k}} > i_{\sigma_{l}}$. Note that
the shorter part of $i_{\sigma_{a}}$ $(a < \ell)$ is an $N$-part. Indeed, if the shorter part of $i_{\sigma_{a}}$ is an $S$-part,
the fact that $S$-part and $N$-part of $1$-swing are strictly longer than the shorter part of the $a$-swing yields
a contradiction by the same reason in (a). Thus $N$-parts of $i_{\sigma_{k}}$-swing and $i_{\sigma_{l}}$-swing are shorter parts.

If $i_{\sigma_{k}} < i_{\sigma_{l}}$, then $N$-parts of $i_{\sigma_{k}}$-swing is longer than $N$-parts of $i_{\sigma_{l}}$. Then there exist a
vertex $(i,p)$ satisfying \eqref{eq:cont known char} by setting $b=i_{\sigma_{l}}$. Thus we have a contradiction. For $k>l > \ell$,
we can apply the similar argument and hence $i_{\sigma_k}>i_{\sigma_l}$.
\end{proof}

Let $\kappa$ be a subset of $\Phi^{+}_n$ defined as follows:
\begin{equation} \label{eq: kappa sigma}
\kappa \seq \{  \be \in \Phi^{+}_n \ | \ \inp_1(\be)=1  \}.
\end{equation}
We enumerate the positive roots in $\kappa=\{\kappa_1,\ldots,\kappa_{n-1}\}$ in the following way:
$$  \inp_{2}(\kappa_{i+1})+2= \inp_{2}(\kappa_{i}),  \quad  \text{ for } 1 \le i \le n-2.$$
Then \eqref{eq: m n-1 m n}, Proposition \ref{Prop: maximal path} and Theorem \ref{thm: short path} tell that $|\kappa|=n-1$ and each element in $\kappa$ is contained in
only one shallow maximal path, maximal path sharing $\ve_{\mathtt{t}'}$ or maximal path sharing $-\ve_{\mathtt{t}'}$.
We set, for $1 \le k \le n-1$,
indices $$\{ j_{\kappa_{1}},j_{\kappa_{2}},\ldots,j_{\kappa_{n-1}} \}=\{ -2,\ldots,-n+2, \mathtt{t}', -\mathtt{t}' \}$$
such that $\kappa_{s}$ contains $-\ve_{-j_{\kappa_s}}$ as its summand.

The following lemma can be proved by using Remark \ref{Alg: easy} and the similar argument in Corollary \ref{cor: reverse uni}:
\begin{corollary} \label{Cor: first}
Notice that $\ve_1+\ve_2 $ is the longest root in $\Phi^+$. We have the followings:
\begin{itemize}
\item[({\rm a})] $ \ve_1+\ve_2 =
\begin{cases} \kappa_1 + \kappa_2+ \cdots + \kappa_{n-2} & \text{ if $1$ is a sink }, \\
\kappa_2 + \kappa_3+ \cdots + \kappa_{n-1} & \text{ if $1$ is a source. } \end{cases}$
\item[({\rm b})] There exists $\mathtt{l}$ such that $$ |j_{\kappa_{n-1}}|< \ldots < |j_{\kappa_{\mathtt{l}}}|=\mathtt{t}'=|j_{\kappa_{\mathtt{l}-1}}| > |j_{\kappa_{\mathtt{l}-2}}| > \cdots > |j_{\kappa_{1}}|,$$
and the maximal sectional path sharing $-\ve_{j_{\kappa_s}}$ is the $S$-sectional, if $s \le \mathtt{l}-1$, and
the maximal sectional path sharing $-\ve_{j_{\kappa_s}}$ is the $N$-sectional, otherwise.
\item[({\rm c})] $\displaystyle\sum_{i=1}^{n-1}\kappa_i=2\ve_1$ and
$\left\{ \displaystyle\sum_{i=1}^{\mathtt{l}-1}\kappa_i,\displaystyle\sum_{i=\mathtt{l}}^{n-1}\kappa_i \right\} = \left\{ \ve_1+\ve_{\mathtt{t}'}, \ve_1-\ve_{\mathtt{t}'} \right\}$.
\end{itemize}
\end{corollary}

\section{The generalized quantum affine Schur-Weyl duality}
 In this section,
we briefly recall the basic materials of quantum affine algebras and
 quiver Hecke algebras  following \cite{Oh14,Oh14A}. For precise definitions in this subsection, we refer
\cite{KKK13b,Oh14,Oh14A}.

\subsection{Quantum affine algebras and their finite dimensional integrable modules.}
Set $I=I_0 \bigsqcup \{ 0 \}$. An affine Cartan datum $(\cm,\wl^{\vee},\wl,\Pi^{\vee},\Pi)$ consists of
{\rm (i)} a {\it generalized affine Cartan matrix} $\cm=(a_{i,j})_{i,j\in I}$,
{\rm (ii)} a {\it dual weight lattice} $\wl^{\vee}=\bigoplus_{i=0}^{n} \Z h_i \oplus \Z d $,
{\rm (iii)} a {\it weight lattice} $\wl=\bigoplus_{i=0}^{n} \Z \Lambda_i \oplus \Z \delta$,
{\rm (iv)} a set of {\it simple coroots} $\Pi^{\vee} = \{ h_i \mid i\in I\} \subset \wl^{\vee}$ and
{\rm (v)} a set of {\it simple roots} $\Pi = \{ \al_i \mid i\in I\} \subset \wl$
such that
\begin{eqnarray}&&\parbox{85ex}{
\begin{itemize}
\item $\delta=\sum_{i \in I} \mathsf{a}_ih_i$ is the {\it null root} (\cite[Chapter 4]{Kac}),
\item $\langle h_i, \al_j \rangle  = a_{ij}$ and $\langle h_j, \Lambda_i \rangle =\delta_{ij}$ for all $i,j \in I$,
\item $\Pi^\vee$ and $\Pi$ are linearly independent sets.
\end{itemize} }\label{eq: fw} \end{eqnarray}
Note that $\cm$ is {\it symmetrizable}; i.e., there exists a diagonal matrix $\mathsf{D}\seq{\rm diag}(\mathsf{d}_i \in \Z_{> 0} \ | \ i \in I)$ such that
$\mathsf{D}\cm$ is symmetric.

We denote by $\rl=\bigoplus_{i \in I} \Z \al_i$ the {\it root lattice}, $\rl^+=\bigoplus_{i \in I} \Z_{>0} \al_i$ and
$c=\sum_{i \in I}\mathsf{c}_i\al_i$ the {\it center}.

Let $( \ , \ )$ be a $\Q$-valued symmetric bilinear form on $\wl$ satisfying
$$  \langle h_i,\lambda \rangle = \dfrac{2(\al_i,\lambda)}{\al_i,\al_i} \text{ and } (\delta,\lambda)=\langle c,\lambda \rangle \text{ for any }
\lambda \in \wl.$$

For an indeterminate $q$, we define $ q_i=q^{(\al_i,\al_i)/2}$. Let $\field$ be an algebraic closure of $\C(q)$ in $\cup_{k >0} \C((q^{1/k}))$.

The {\it quantum affine algebra} $\Ug$ associated with an affine Cartan datum $(\cm,\wl^{\vee},\wl,\Pi^{\vee},\Pi)$
is the associate $\field$-algebra with generators $e_i$, $f_i$ $(i \in I)$ and $q^h$ $(h \in \wl^\vee)$ with certain relations (see e.g.
\cite[Definition 2.1]{Oh14A}). The $\Up$, also called a {\it quantum affine algebra}, is a subalgebra of $\Ug$ generated by
$e_i$, $f_i$ and $K_i^{\pm 1}$ $(i \in I)$, where $K_i\seq q_i^{h_i}$.

We choose $0 \in I$ as the leftmost vertices in the tables in \cite[pages 54, 55]{Kac} except $A^{(2)}_{2n}$-case in which we take the
longest simple root as $\alpha_0$. Note that the subalgebra of $\Ug$ generated by $e_i$, $f_i$ and $K_i^{\pm 1}$ $(i \in I_0)$
is isomorphic to the universal enveloping algebra of finite simple Lie algebra $\g_0$

The algebra $\Up$ has a Hopf algebra structure with the comultiplication
\begin{align} \label{eq: comul}
\Delta(K_i)=K_i\otimes K_i, \quad \Delta(e_i)=e_i\otimes K_i^{-1}+1\otimes
e_i,\quad \Delta(f_i)=f_i\otimes 1+K_i\otimes f_i.
\end{align}

Let $\wl_{{\rm cl}}=\wl/\Z \delta$ and ${\rm cl}:\wl \to \wl_{{\rm cl}}$ be the canonical map.
We say that a $\Up$-module $M$ is {\it integrable} if
\begin{itemize}
\item $ M = \bigoplus_{\lambda \in \wl_{{\rm cl}}} M_\lambda$,
where $M_{\lambda}= \{ u \in M \ | \ \text{$K_i u =q_i^{\langle h_i , \lambda \rangle} u$ for all $i\in I$} \}$,
\item the action of $e_i$ and $f_i$ on $M$ are locally nilpotent for all $i \in I$.
\end{itemize}

Let $\Cat_\g$ be a category of finite dimensional integrable $\Up$-modules. By \eqref{eq: comul}, $\Cat_\g$ is a stable by tensor
product; i.e., for $M,N \in \Cat_\g$, we have $M \otimes N \in \Cat_\g$.

Let $M$ be a simple module in $\Cat_\g$. Then it is known that there exist a unique non-zero vector $v_M$ (up to non-zero constant multiple)
of weight $\lambda \in \wl_{{\rm cl}}^0 \seq \{ \lambda \in \wl_{{\rm cl}} \ | \ \langle c,\lambda \rangle =0 \}$ such that $\langle \lambda,h_i \rangle \ge 0$ $(i \in I_0)$ and
all weights of $M$ are contained in $\lambda - \sum_{i \in I_0} \Z_{\ge 0} \al_i$.

Set $\varpi_i={\rm gcd}(\mathsf{c}_0,\mathsf{c}_i)^{-1}(\mathsf{c}_0\Lambda_i-\mathsf{c}_i\Lambda_0) \in \wl$ for $i \in I_0$. Then there exists a unique simple
$\Up$-module $\Vi$ $(i \in I_0)$ in $\Cat_\g$ (up to isomorphism) satisfying the following properties:
\begin{itemize}
\item The weights of $\Vi$ are contained in the convex hull of $W_0 \varpi_i$,
\item $\dim \Vi_{\varpi_i}=1$,
\item For any $\mu \in W_0 \varpi_i$, there exists a non-zero vector $u_\mu$ such that
$$ u_{s_i\mu} =\begin{cases} f_i^{(\langle h_i,\mu \rangle)} u_\mu & \text{if } \langle h_i,\mu \rangle \ge 0, \\
e_i^{(\langle h_i,\mu \rangle)} u_\mu & \text{if } \langle h_i,\mu \rangle \le 0,\end{cases} \quad \text{ for all } i \in I.$$
\end{itemize}
We call $\Vi$ the {\it $i$th fundamental representation} (\cite[\S 1.3]{AK}).

For $M \in \Cat_\g$, the {\it affinization $\Maff$} is a $\wl$-graded $\Up$-module $$\Maff = \soplus_{\lambda \in \wl} (\Maff)_\lambda$$ for which
$(\Maff)_\lambda = M_{{\rm cl}(\lambda)}$ and the action of $e_i$ and $f_i$ are defined in a way that
they commutes with the canonical map ${\rm cl}: \Maff \to M$.

For $x \in \field^\times$ and $M \in \Cat_\g$, we define
$$ M_x \seq \Maff / (z_M-x)\Maff \in \Cat_\g$$
where $z_M$ denotes the $\Up$-automorphism of $M_\aff$ of weight $\delta$  (see, \cite[\S 2.2]{KKK13a}).

\begin{definition} \label{def: CQone} \cite{HL11} (see also. \cite[\S 3.3]{KKK13b})
Let $Q$ be the Dynkin quiver of finite type $D_n$ (resp. $A_n$). Set $\g=D^{(1)}_{n}$ and $\g_0=D_{n}$ (resp. $\g=A^{(1)}_{n}$ and $\g_0=A_{n}$).
For every $\be \in \Prt$ associated to $\g_0$, we define the $\Up$-module $\VQbe$ as follows:
\begin{align} \label{V_Q(beta)}
\VQbe \seq \Vi_{(-q)^{p}} \quad \text{ for } \quad \inp(\be)=(i,p).
\end{align}
The subcategory $\Cat^{(1)}_Q$ is the smallest abelian full subcategory of $\Cat_\g$ satisfying
\begin{itemize}
\item[({\rm a})] $\VQbe \in \Cat^{(1)}_Q$ for all $\be \in \Prt$,
\item[({\rm b})] it is stable by taking
subquotient, tensor product and extension.
\end{itemize}
\end{definition}

A simple integrable $\Up$-module $M \in \Cat_\g$ is {\it good} if $M$ has a {\it bar involution}, {\it a crystal basis} and a {\it global basis} (see
\cite{Kas02} for precise definitions). For example, $\Vi_{(-q)^p}$ for $(i,p) \in \Z Q$ is a good module.

For a good module $M$ and $N$, there exist a $\Up$-homomorphism
$$ \Rnorm_{M_1,M_2}: \Maff \otimes \Naff \to \field(z_M,z_N)  \otimes_{\field[z_M^{\pm 1},z_N^{\pm 1}]} \Naff \otimes \Maff $$
such that
$$ \Rnorm_{M,N} \circ z_M = z_M \circ \Rnorm_{M,N}, \ \Rnorm_{M,N} \circ z_N = z_N \circ \Rnorm_{M,N} \text{ and }
\Rnorm_{M,N} (v_M \otimes v_N) = v_N \otimes v_M.$$

The {\it denominator} $d_{M,N}$ of $\Rnorm_{M,N}$ is the unique non-zero monic polynomial $d(u) \in \field[u]$ of smallest degree such that
\begin{equation}\label{definition: dm,n}
d_{M,N}(z_N/z_M)\Rnorm_{M,N}(\Maff \otimes \Naff) \subset (\Naff \otimes \Maff).
\end{equation}

The module $\Vi$ has a left dual $\Vi^*$ and a right dual ${}^*\Vi$ with the duality $\Up$-morphisms:
\begin{equation} \label{equation: psta}
\begin{aligned}
& \Vi^* \otimes \Vi  \overset{{\rm tr}}{\longrightarrow} \field \quad \text{ and } \quad
\Vi \otimes {}^*\Vi  \overset{{\rm tr}}{\longrightarrow} \field, \quad \text{where} \\
& \Vi^*\seq  V(\varpi_{i^*})_{(p^*)^{-1}}, \ {}^*\Vi\seq  V(\varpi_{i^*})_{p^*} \ \  \text{ and } \ \
p^* \seq (-1)^{\langle \rho^\vee ,\delta \rangle}q^{(\rho,\delta)}.
\end{aligned}
\end{equation}
Here $\rho$ and $\rho^\vee$ denote elements in $\wl$ and $\wl^\vee$ such that
$\langle h_i,\rho \rangle=1$ and $\langle \rho^\vee,\al_i  \rangle=1$ for all $i \in I$.

The following theorem tells that the denominators $d_{i,j}(z)\seq d_{\Vi,\Vj}(z)$ of $\Rnorm_{\Vi,\Vj}$ ($i,j \in I_0$) provide important information about the category $\Cat_\g$:

\begin{theorem} \cite{AK,Kas02} \label{Thm: bpro}
\begin{enumerate}
\item The zeros of $d_{i,j}(z) \in \C[[q^{1/m}]]q^{1/m}$ for some $m \in \Z_{>0}$.
\item $ \Vi_{a_i} \otimes  \Vj_{a_j}$ is simple if and only if
$d_{i,j}(a_i/a_j) \ne 0$ and $d_{i,j}(a_j/a_i) \ne 0$.
\item For any simple integrable $U'_q(\g)$-module $M$, there
exists a unique finite sequence $($up to permutation$)$
$$\left( (i_1,a_1),\ldots, (i_l,a_l)\right) \text{ in } (I_0 \times \field^\times)^l$$
such that
$d_{i_k,i_{k'}}(a_{k'}/a_k) \not=0$ for $1\le k<k'\le l$ and $M$ appears as the head of $\tens_{i=1}^{l}V(\varpi_{i_k})_{a_k}$.
\end{enumerate}
\end{theorem}

\subsection{The Dorey's type morphisms and denominators $d_{k,l}(z)$ when $\g=D^{(1)}_{n}$ and $\g=D^{(2)}_{n+1}$.}
In this subsection, we recall some family of morphisms in $\Cat_\g$, called by {\it Dorey's type morphisms}, and
denominators $d_{k,l}(z)$ of $\Rnorm_{k,l}(z)$ for untwisted and twisted quantum affine algebra of type $D$.
After that we observe their similarities.

We say that an element in $\Hom_{\Up}(\Vi_{(-q)^a}\otimes \Vj_{(-q)^b}, V(\varpi_k)_{(-q)^b})$ is a
{\it Dorey's type morphism} \cite{CP96}. This kind of morphisms
were studied at \cite{CP96} for untwisted affine types $\g=A^{(1)}_{n}$, $B^{(1)}_{n}$, $C^{(1)}_{n}$ and $D^{(1)}_{n}$. Recently, the author investigated
such type morphisms for twisted affine types $\g=A^{(2)}_{2n}$, $A^{(2)}_{2n-1}$ and $D^{(2)}_{n+1}$ in \cite{Oh14}.

\begin{theorem} \label{Thm: Dorey D_n(1)} \cite[Theorem 6.1]{CP96} $($see also. \cite[Appendix A]{KKK13b}$)$
Let $\g$ be of type $D^{(1)}_{n}$ and $(i,x),(j,y),(k,z) \in I_0 \times {\bf k}^\times$.  Then
$$ \Hom_{U_q'(\g)}\big( V(\va_i)_x\tens V(\va_j)_y, V(\va_k)_z \big) \ne 0 $$
if and only if one of the following conditions holds:
\begin{eqnarray}&&
\left\{\parbox{73ex}{
\begin{enumerate}
\item[{\rm (i)}] $\ell \seq \max(i,j,k) \le n-2$, $s+m =\ell$ for $\{ s,m \} \seq \{ i,j,k\} \setminus \{ \ell \}$
and
$$ \big( x/z,y/z \big) =
\begin{cases}
\big( (-q)^{-j},(-q)^{i} \big), & \text{ if } \ell = k,\\
\big( (-q)^{-j},(-q)^{-i+2n-2} \big), & \text{ if } \ell = i,\\
\big( (-q)^{j-2n+2},(-q)^{i} \big), & \text{ if } \ell = j,
\end{cases}
$$
\item[{\rm (ii)}] $i+j \geq n$, $k=2n-2-i-j$, $\max(i,j,k) \le n-2$, and $x/z=(-q)^{-j}$, $y/z=(-q)^i$,
\item[{\rm (iii)}] $s \seq \min(i,j,k) \le n-2$, $ \{ m,\ell \}  \seq \{ i,j,k\} \setminus \{ s \} \subset \{ n-1,n\}$ and
\begin{itemize}
\item $n-s  \equiv  \begin{cases} \ell-m \ \ {\rm mod} \ 2 & \text{ if } s = k, \\
\ell-m^*  \ {\rm mod} \ 2 & \text{ if } s \ne k, \end{cases}$
\item $\big( x/z,y/z \big) = \begin{cases}
\big( (-q)^{-n+k+1},(-q)^{n-k-1} \big), & \text{ if } s= k,\\
\big( (-q)^{-n+i+1},(-q)^{2i} \big), & \text{ if } s = i,\\
\big( (-q)^{-2j},(-q)^{n-j-1} \big), & \text{ if } s = j.
\end{cases}$
\end{itemize}
\end{enumerate}
}\right. \label{eq: p ij D}
\end{eqnarray}
\end{theorem}

\begin{theorem} \label{Thm: Dorey D_n+1(2)} \cite[Theorem 3.6, Theorem 3.8]{Oh14} Let $\g$ be of type $D^{(2)}_{n+1}$
and $(i,x),(j,y),(k,z) \in I_0 \times {\bf k}^\times$.  Then
$$ \Hom_{U_q'(\g)}\big( V(\va_i)_x\tens V(\va_j)_y, V(\va_k)_z \big) \ne 0 $$
if one of the following conditions holds:
\begin{eqnarray}&&
\left\{\parbox{73ex}{
\begin{enumerate}
\item[{\rm (i$'$)}] $\ell \seq \max(i,j,k) \le n-1$, $s+m =\ell$ for $\{ s,m \} \seq \{ i,j,k\} \setminus \{ \ell \}$
and
$$ \big( x/z,y/z \big) =
\begin{cases}
\big( (-q^2)^{-j/2},(-q^2)^{i/2} \big), & \text{ if } \ell = k,\\
\big( (-q^2)^{-j/2},(-q)^{-i/2+n} \big), & \text{ if } \ell = i,\\
\big( (-q^2)^{j/2-n},(-q)^{i} \big), & \text{ if } \ell = j,
\end{cases} \quad \text{ $($up to sign$)$}
$$
\item[{\rm (iii$'$)}]  $s \seq \min(i,j,k) \le n-1$, $\{ m,\ell\} \seq \{i,j,k\} \setminus \{s\} \subset \{n\}$ and
$$\big( x/z,y/z \big) = \begin{cases}
\big( \pm \sqrt{-1}(-q^2)^{\frac{-n+k}{2}},\mp \sqrt{-1}(-q)^{\frac{n-k}{2}} \big), & \text{ if } s= k,\\
\big( \pm \sqrt{-1}(-q)^{\frac{-n+i}{2}},(-q^2)^{i} \big), & \text{ if } s = i,\\
\big( (-q^2)^{-j},\pm \sqrt{-1}(-q^2)^{\frac{n-j}{2}} \big), & \text{ if } s = j.
\end{cases}$$
\end{enumerate}
}\right. \label{eq: p ij 2 D}
\end{eqnarray}
\end{theorem}

Now we fix a Dynkin quiver $Q$ of finite type $D_{n+1}$ and take a height function $\xi$ on $Q$ satisfying
$$\xi_{n} \equiv \xi_{n+1} \equiv 0 \ ({\rm mod} \ 2).$$

For $1 \le i \le n+1$ and $p \in \Z$, we set
\begin{equation} \label{eq: star}
\begin{aligned}
& \{1,\ldots,n \} \ni i^\star \seq \begin{cases} i & \text{ if } 1 \le i \le n-1,\\
n & \text{ if } i=n \text{ or } n+1,  \end{cases}  \\
& ((-q)^p)^\star \seq  \begin{cases} (\sqrt{-1})^{\delta(n+1 \equiv i ({\rm mod} \ 2) )+1}  (-q)^p & \text{ if } 1 \le i \le n-1, \\
(-1)^i(-q)^p& \text{ if } i=n \text{ or } n+1. \end{cases}
\end{aligned}
\end{equation}
With the map \eqref{eq: star}, we can associate an injective map
\begin{align*}
& ^\star \colon \big\{ V(\va_i)_{a(-q)^p} \in\Cat_{D^{(1)}_{n+1}} \mid  1 \le i \le n+1, \ a \in \C^\times, \ p \in \Z \big\} \\
& \hspace{28ex} \longrightarrow \big\{ V(\va_i)_{ a  (-q)^p} \in\Cat_{D^{(2)}_{n+1}} \mid  1 \le i \le n, \ a \in \C^\times, \ p \in \Z \big\}
\end{align*}
given by (see \cite{Her10})
$$ (V(\va_i)_{(-q)^{p}})^\star \seq V(\va_{i^\star})_{((-q)^{p})^\star} \in \Cat_{D^{(2)}_{n+1}}.$$

For each $\be \in \Prt $ with $\inp(\be)=(i,p) $,
we define
\begin{equation} \label{eq: mathsf V}
 \mathsf{V}_Q(\be) \seq (\VQbe)^\star = \begin{cases} \Vi_{(\sqrt{-1})^{\delta(n+1 \equiv i ({\rm mod} \ 2) )+1}  (-q)^p } & \text{ if } 1 \le i \le n-1, \\
                                          V(\varpi_{n})_{(-1)^i(-q)^p} & \text{ if }  i=n \text{ or } n+1,\end{cases}
\end{equation}
which is the good module over $U'_q(D^{(2)}_{n+1})$.

For $1 \le i \le n-1$, it is known that (\cite[(1.7)]{AK})
\begin{equation} \label{eq: + iso -}
\Vi \simeq \Vi_{-1} \in \Cat_{D_{n+1}^{(2)}}.
\end{equation}

\begin{remark} \label{rmk: Dn+1,1 Dn+1,2}
By replacing $n$ in Theorem \ref{Thm: Dorey D_n(1)} with $n+1$ and considering \eqref{eq: + iso -} together, one can observe that
\begin{itemize}
\item[{\rm (a)}] $p^*$ of $D^{(1)}_{n+1}$ is $(-q)^{2n}$ and $p^*$ of $D^{(2)}_{n+1}$ is $-(-q^2)^{n}$,
\item[{\rm (b)}] for a surjective $U_q'(D^{(1)}_{n+1})$-homomorphism of types {\rm (i)} and {\rm (iii)} in Theorem \ref{Thm: Dorey D_n(1)}
$$\Vi_a \tens \Vj_b \twoheadrightarrow V(\varpi_k)_c, $$
we have a $U_q'(D^{(2)}_{n+1})$-homomorphism of types {\rm (i$'$)} and {\rm (iii$'$)} in Theorem \ref{Thm: Dorey D_n+1(2)}
$$V(\varpi_{i^\star})_{a^\star} \tens V(\varpi_{j^\star})_{b^\star} \twoheadrightarrow V(\varpi_{k^\star})_{c^\star}.$$
\end{itemize}
\end{remark}

Now we record the denominators $d_{k,l}(z)$ for $\g=D^{(1)}_{n}$ and $D^{(2)}_{n+1}$.

\begin{theorem} \label{Thm; denominator} \cite[Appendix A]{KKK13b} \cite[Section 4]{Oh14} \hfill
\begin{enumerate}
\item[({\rm a})] For $1\leq k,l \leq n$ and $\g=D^{(1)}_{n}$, we have
\begin{equation} \label{eq: denominator 1}
\begin{aligned}
& d_{k,l}(z)= \begin{cases}
\displaystyle \prod_{s=1}^{\min (k,l)}(z-(-q)^{|k-l|+2s})(z-(-q)^{2n-2-k-l+2s}) & \text{ if } 1 \le k,l \le n-2, \allowdisplaybreaks \\
\quad \displaystyle \prod_{s=1}^{k}(z-(-q)^{n-k-1+2s}) &  \text{ if }  1 \le k \le n-2 < l, \allowdisplaybreaks \\
\quad \displaystyle \prod_{s=1}^{\lfloor \frac{n-1}{2} \rfloor} (z-(-q)^{4s}) &  \text{ if }  \{k,l\}=\{n,n-1\},  \allowdisplaybreaks  \\
\quad \displaystyle \prod_{s=1}^{\lfloor \frac{n}{2} \rfloor} (z-(-q)^{4s-2}) &  \text{ if }  k=l \in \{ n-1, n\}.
 \end{cases}
\end{aligned}
\end{equation}
\item[({\rm b})] For $1\leq k,l \leq n$ and $\g=D^{(2)}_{n+1}$, we have
\begin{equation} \label{eq: denominator 2}
\begin{aligned}
& d_{k,l}(z)= \begin{cases}
\displaystyle \prod_{s=1}^{\min(k,l)} \hspace{-1ex} (z^2  -  (-q^2)^{|k-l|+2s})
(z^2 - (-q^2)^{2n-k-l+2s}) & \text{ if } 1 \le k,l \le n-1,  \allowdisplaybreaks  \\
\quad \displaystyle \prod_{s=1}^{k}(z^2+(-q^{2})^{n-k+2s}) & \hspace{-5ex} \text{ if }  1 \le k \le n-1< l=n,  \allowdisplaybreaks  \\
\quad \displaystyle \prod_{s=1}^{n} (z+(-q^2)^{s}) & \text{ if } k=l=n.
 \end{cases}
\end{aligned}
\end{equation}
\end{enumerate}
\end{theorem}

From \eqref{eq: denominator 1} and \eqref{eq: denominator 2}, one can observe that $d_{k,l}(z)$ has a zero of multiplicity $2$ when
\begin{itemize}
\item[(i)] $\g$ is of type $D^{(1)}_{n}$ $(n \ge 4)$ and $z=(-q)^s$ where (\cite[Lemma 3.2.4]{KKK13b})
\begin{equation} \label{eq: dpole 1}
2 \le k,l \le n-2, \ k+l > n-1, \ 2n-k-l \le s \le k+l \text{ and } s \equiv k+l \mod \ 2.
\end{equation}
\item[(ii)] $\g$ is of type $D^{(2)}_{n+1}$ $(n \ge 3)$ and $z=(-q^2)^{s/2}$ where (\cite[Corollary 4.16]{Oh14})
\begin{equation} \label{eq: dpole 2}
2 \le k,l \le n-1, \ k+l > n, \ 2n+2-k-l \le s \le k+l \text{ and } s \equiv k+l \mod \ 2.\end{equation}
\end{itemize}

We want to emphasize that
\begin{eqnarray}&&
\parbox{85ex}{
\begin{enumerate}
\item[{\rm (a)}] if we replace $n$ in \eqref{eq: dpole 1} with $n+1$, we can obtain the condition \eqref{eq: dpole 2},
\item[{\rm (b)}] the condition \eqref{eq: p ij D} {\rm (ii)} satisfies the condition \eqref{eq: dpole 1}.
\end{enumerate}
}
\end{eqnarray}

\subsection{Quiver Hecke algebras and the generalized quantum affine Schur-Weyl duality functors.}
Let $\field$ be a field. For a given symmetrizable Cartan matrix $\cm$, we choose a polynomials
$\calQ_{i,j}(u,v) \in \field[u,v]$, for $i,j \in I$, of the form
\begin{align} \label{eq: condition A}
\calQ_{i,j}(u,v) =
\begin{cases}
\sum_{ p(\al_i,\al_i)+q(\al_j,\al_j)+2(\al_i,\al_j)=0} t_{i,j;p,q}u^pv^q & \text{ if } i \ne j, \\
0  & \text{ if } i = j,
\end{cases}
\end{align}
where $t_{i,j;p,q} \in \field$ are such that $t_{i,j;-a_{ij},0} \ne 0$, and $t_{i,j;p,q} =t_{j,i;q,p}$. Thus $\calQ_{i,j}(u,v)=\calQ_{j,i}(v,u)$.
The symmetric group $\mathfrak{S}_m= \langle \mathfrak{s}_1,\mathfrak{s}_2,\ldots,\mathfrak{s}_{m-1} \rangle$ acts on $I^m$ by place permutations.

For $n \in \Z_{\ge 0}$ and $\beta \in \rl^+$ such that $\het(\beta) = n$, we set
$$I^{\beta} = \{\nu = (\nu_1, \ldots, \nu_n) \in I^{n} \ | \ \alpha_{\nu_1} + \cdots + \alpha_{\nu_n} = \beta \}.$$

For $\beta \in \rl^+$, we denote by $R(\beta)$ the quiver Hecke algebra at $\beta$ associated
with $(\cm,\wl, \Pi,\wl^{\vee},\Pi^{\vee})$ and
$(Q_{i,j})_{i,j \in I}$. It is a $\Z$-graded $\field$-algebra generated by
the generators $\{ e(\nu) \}_{\nu \in  I^{\beta}}$, $ \{x_k \}_{1 \le
k \le n}$, $\{ \tau_m \}_{1 \le m \le n-1}$ with certain defining relations
(see e.g \cite[Definition 2.7]{Oh14A}).

Let $\Rep(R(\be))$ be the category of finite dimensional graded $R(\be)$-modules and $[\Rep(R(\be))]$ be the Grothendieck group of $\Rep(R(\beta))$.
Then $[\Rep(R(\be))]$ has a natural $\Z[q,q^{-1}]$-module structure induced by the grading shift. More precisely, $(qM)_k = M_{k+1}$ for
$M \in \Rep(R(\be))$ such that $M = \soplus_{k \in \Z} M_k$.

For $M \in \Rep(R(\be))$ and $N \in \Rep(R(\gamma))$ for $\beta,\gamma \in \rl^+$, we have the module
$M \cvl N \in \Rep(R(\be+\gamma))$ induced by the {\it convolution product}. Then $[\Rep(R)] \seq \soplus_{\beta \in \rl^+} [\Rep(R(\be))]$ has a natural
algebra structure induced by the convolution product. For $M \in \Rep(\be)$ and $M_k \in \Rep(\be_k)$ $(1 \le k \le n)$, we denote by
$$ M^{\cvl 0} \seq \field, \quad M^{\cvl r} = \underbrace{ M \cvl \cdots \cvl M }_r, \quad \dct{k=1}{n} M_k = M_1 \cvl \cdots \cvl M_n.$$

\medskip

For a fixed reduced expression $\eqcw$ of $w_0$,
let us consider the convex total order $<_{\redez}$ on $\Prt$
$$ \be_{1} <_{\redez} \be_{2} <_{\redez}  \cdots <_{\redez} \be_{\msN}.$$

For sequences $\um, \  \um' \in \Z_{\ge 0}^{\mathsf{N}}$, we define an order $\le^\tb_{\redex_0}$ as follows (\cite{Mc12}):
\begin{eqnarray*}&&
\parbox{65ex}{$\um'=(m'_1,\ldots,m'_N) <^\tb_{\redex_0}  \um=(m_1,\ldots,m_N) $ if and only if there exist integers $k$, $s$ such that $1 \le k \le s \le N$,
$m_t'=m_t$ $(t<k)$, $m'_k<  m_k$, and $m_t'=m_t$ $(s<t\le N)$,
$m'_{s}<  m_{s}$.}
\end{eqnarray*}
Note that $\le^\tb_{\redex_0}$ is a partial order on sequences.

\begin{theorem} \cite[Theorem 4.7]{BKM12}\cite[Theorem 3.1]{Mc12} (see also \cite{Kato12}) \label{thm: BkMc}
Let $\g_0$ be a finite simple Lie algebra and $R \seq \soplus_{\beta \in \rl^+} R(\be)$ be the quiver Hecke algebra corresponding to $\g_0$.
For each positive root $\beta \in \Prt$, there exists a simple module $S_{\redez}(\beta)$ satisfying following properties:
\begin{enumerate}
\item[({\rm a})]  For every $\um \in \Z_{\ge 0}^{\mathsf{N}}$, there exists a non-zero $R$-module homomorphism
\begin{align*}
&\rmat{\um} \colon
\overset{\to}{S}_{\redez}(\um) \seq
S_{\redez}(\beta_1)^{\cvl m_1}\cvl\cdots \cvl S_{\redez}(\beta_N)^{\cvl m_N}\\
&\hs{20ex}{\Lto}\overset{\gets}{S}_{\redez}(\um) \seq
S_{\redez}(\beta_N)^{\cvl m_N}\cvl\cdots \cvl S_{\redez}(\beta_1)^{\cvl m_1}.
\end{align*}
and ${\rm Im}(\rmat{\um}) \simeq {\rm hd}\left(\overset{\to}{S}_{\redez}(\um)\right) \simeq {\rm soc}\left(\overset{\gets}{S}_{\redez}(\um)\right)$ is simple.
\item[({\rm b})] For any sequence $\um\in \Z_{\ge 0}^{\mathsf{N}}$, we have
$[\overset{\to}{S}_{\redez}(\um)] \in [{\rm Im}(\rmat{\um})] + \displaystyle\sum_{\um' <^{\tb}_\redez \um } \Z_{\ge 0} [{\rm Im}(\rmat{\um'})].$
\item[({\rm c})] For any simple $R(\beta)$-module $M$, there exist a unique sequence $\um\in \Z_{\ge 0}^{\mathsf{N}}$ such that $M \simeq {\rm Im}(\rmat{\um}) \simeq {\rm hd}\big( \overset{\to}{S}_{\redez}(\um) \big).$
\item[({\rm d})] For a minimal pair $(\beta_k,\beta_l)$ of $\beta_j=\beta_k+\beta_l$ with respect to $<_\redez$, there exists an exact sequence
of $R$-modules
$$ 0 \to S_{\redez}(\beta_j)
\to S_{\redez}(\beta_k) \cvl S_{\redez}(\beta_l) \overset{{\mathbf r} \ne 0}{\Lto} S_{\redez}(\beta_l) \cvl S_{\redez}(\beta_k)
\to S_{\redez}(\beta_j) \to 0$$
such that ${\rm Im}(\rmat{})$ is simple.
\end{enumerate}
\end{theorem}

For any pair of reduced expressions $\redez$ and
$ \widetilde{w_0}' \in [\redez]$, we have
$$S_{\redez}(\beta) \simeq S_{\redez'}(\beta) \quad \text{ for all } \beta \in \Phi^+_{n},$$
by \eqref{eq: [redez]} and Theorem \ref{thm: BkMc} ({\rm d}).
Thus we denote by $S_Q(\beta)$ the simple $R(\beta)$-module $S_{\redez}(\beta)$ for any reduced expression $\redez \in [Q]$.

\medskip

For a given quiver $\Gamma = (\Gamma_0,\Gamma_1)$ and $i,j \in \Gamma_0$,
we denote by
$$\text{$\Gamma[i,j]$ the number of arrows $\mathsf{a}$ in $\Gamma_1$ such that $s(\mathsf{a})=i$ and $t(\mathsf{a})=j$.}$$

Define a generalized symmetric Cartan matrix $\cm^\Gamma=(a^\Gamma_{i,j})_{i,j\in \Gamma_0}$ and a
set of polynomials $(\calQ^\Gamma_{i,j}(u,v))_{i,j\in \Gamma_0}$ associated to the given quiver $\Gamma$ as follows:
\begin{equation} \label{equation: quiver cm}
\begin{aligned}
a^\Gamma_{i,j} = \begin{cases} \qquad \quad 2 \\
-\Gamma[i,j]-\Gamma[j,i] \end{cases}
\quad \text{ and } \quad
\calQ^\Gamma_{i,j}(u,v)= \begin{cases} (u-v)^{\Gamma[i,j]}(v-u)^{\Gamma[j,i]} & \text{ if } i \ne j \in \Gamma_0, \\ \qquad \qquad 0 & \text{ otherwise.}\end{cases}
\end{aligned}
\end{equation}
We denote by $R^\Gamma$ the quiver Hecke algebra associated with the polynomials $(\calQ^\Gamma_{i,j}(u,v))_{i,j\in \Gamma_0}$.

\medskip

From now on, we recall the {\it quantum affine Schur-Weyl duality functors} which were introduced in \cite{KKK13a,KKK13b}.

\medskip
Let $\{ V_s \}_{s \in \mathcal{S}}$ be a family of good $U_q'(\g)$-modules labeled by an index set $\mathcal{S}$.
For a triple $(J,X,s)$ consisting of
$$ \ {\rm (i)} \text{ $J$ is an index set}, \ {\rm (ii)} \text{ a map $X:J \to \field^\times$}, \ {\rm (iii)} \text{ a map $s:J \to \mathcal{S}$},$$
we can associate a quiver $Q^J=(Q^J_0,Q^J_1)$ as follows:
$$ \text{$Q^J_0=J$ and $Q^J[i,j]$ is the order of the zero of $d_{V_{s(i)},V_{s(j)}}$ at $X(j)/X(i)$ for $i, j\in J$.} $$

\begin{theorem} \cite{KKK13a} \label{thm: duality functor} There exists a functor
$$\F : \Rep(R^{Q^J}) \rightarrow \Cat_\g$$
which satisfies the following properties:
\begin{enumerate}
\item[({\rm a})] $\F$ is a tensor functor. Namely, there exist $U_q'(\g)$-module isomorphisms
$$\F(R^{Q^J}(0)) \simeq \field  \text{ and } \F(M_1 \cvl M_2) \simeq \F(M_1) \tens \F(M_2) \ \ \text{for any $M_1, M_2 \in \Rep(R^{Q^J})$}.$$
\item[({\rm b})] If the underlying graph of $Q^J$ is a Dynkin diagram of finite type $A$, $D$ and $E$, then
the functor $\F$ is exact.
\end{enumerate}
\end{theorem}
We call the functor $\F$ the {\it generalized quantum affine Schur-Weyl duality functor}.

\begin{theorem} \cite{KKK13b} \label{thm: duality functor on Q}
For $\g=D^{(1)}_{n}$ $($resp. $A^{(1)}_{n})$
and $Q$ a Dynkin quiver of finite type $A_n$ $($resp. $D_n)$,
take $$ J \seq \left\{ (i,p) \in \Z Q \ | \ \widehat{\phi}(i,p) \in \Pi_n \times \{ 0 \}  \right\}.$$
We also take two maps $s: J \to \{ V(\varpi_i) \ | \ i \in I_0 \}$ and $X: J \to  \field^\times$ given by
$$ s(i,p)=V(\varpi_i) \ \text{ and } \ X(i,p)= (-q)^p  \ \text{ for } (i,p) \in J.$$
\begin{itemize}
\item[({\rm a})] The underlying graphs of $Q^J$ and $Q$ coincide. Thus the functor
$$\F^{(1)}_Q: \Rep(R^{Q^J}) \rightarrow \Cat^{(1)}_Q \ \text{is exact.} $$
\item[({\rm b})] The functor $\F^{(1)}_Q$ sends simples to simples, bijectively. In particular, $$\F^{(1)}_Q(\SQbe) \simeq \VQbe.$$
\end{itemize}
\end{theorem}

\begin{remark} \label{rem: faithful}
In the above theorem, the functor $\F^{(1)}_Q$ is a faithful functor by the following argument:
For any non-zero homomorphism $f$ in $\Rep(R^{Q^J})$, ${\rm Im}(f)$ does not vanish since $\F^{(1)}_Q$ sends simples to simples.
Thus $\F^{(1)}_Q$ is faithful.
\end{remark}

\section{Dorey's rule on $\AR$ and minimal pair.} \label{sec: Dorey minimal}

In this section, we fix the Dynkin quiver $Q$ of finite type $D_n$, $\g$ as $D^{(1)}_{n}$ and the quiver Hecke algebra $R$ as of finite type $D_n$.
We prove that how $\AR$ codifies the Dorey's
rule for $D^{(1)}_{n}$ and all pair $(\al,\be)$ of a multiplicity free positive root $\ga \in \Prt$ are minimal with suitable total orders.
Moreover, we show that there exist $(n-b-1)$-many {\it non-minimal pairs} $(\al,\be)$ of a {\it non-multiplicity free} $\ga=\ve_a+\ve_b \in \Prt$. To do this,
we need to apply the results in Section \ref{sec: Combinatorics}.

\begin{definition} \label{def: upper ray} Let  $\ga$ be a vertex in $\AR$ with $1<\inp_1(\ga) \le n$. (Equivalently, $\ga \in \Prt$.)
The {\it upper ray of $\ga$} is a concatenation of an $S$-sectional path and an $N$-sectional path satisfying the following properties:
\begin{itemize}
\item  $\underbrace{\mtS_1 \to \cdots \to \mtS_{a} \to \mtS_{a+1}}_{\text{$S$-sectional path}}=\ga
=\overbrace{\mathtt{N}_{b+1} \to \mathtt{N}_{b} \to \cdots \to  \mathtt{N}_1}^{\text{$N$-sectional path}}$.
\item There is no vertex $\mtS_0 \in (\AR)_0$ such that $\mtS_0 \to \mtS_1$ is an $S$-sectional path in $\AR$,
\item There is no vertex $\mathtt{N}_0 \in (\AR)_0$ such that $\mathtt{N}_1 \to \mathtt{N}_0$ is an $N$-sectional path in $\AR$.
\end{itemize}
\end{definition}

\begin{lemma} \label{Lem one of two}
 For every $\ga \in \Prt$ with $1<\inp_1(\ga) \le n$, write its upper ray as in {\rm Definition \ref{def: upper ray}}.
Then we have
\begin{align} \label{eq: upper ray 1}
\inp_1(\mtS_1) \text{ or } \inp_1(\mathtt{N}_1)=1.
\end{align}
\end{lemma}

\begin{proof}
(a) Assume that $\ga=\ve_a-\ve_b$ for $a < b \le n-2$. Then, by Theorem \ref{Thm: V-swing} and  Theorem \ref{thm: short path},
it is contained in the $a$-swing and the shallow maximal $(S,-b)$-sectional path, or the shallow maximal $(N,-b)$-sectional path.
Thus our assertion follows from Theorem \ref{thm: short path}.

(b) Assume that $\ga=\ve_a+\ve_b$ for $a < b \le n-2$. Then, by Theorem \ref{Thm: V-swing}, it is located at the intersection of
the $a$-swing and the $b$-swing. Then our assertion follows from Corollary \ref{cor: reverse uni}.

(c) Assume that $\ga=\ve_a\pm\ve_b$ for $a < b$ and $b=n-1$ or $n$. Then our assertion follows from Lemma \ref{Lem: last 2 line}  and
Corollary \ref{cor: n,n-1}.
\end{proof}

\begin{proposition} \label{prop: upper ray}
 For a pair $(\al,\be)$ of $\al+\be=\ga \in \Prt$, assume that the pair $(\al,\be)$ is contained in
the upper ray of $\ga$. Then we have
$$\VQbe \otimes \VQal \twoheadrightarrow \VQga \quad \text{ and } \quad \SQbe \cvl \SQal \twoheadrightarrow \SQga.$$
\end{proposition}

\begin{proof}
In this proof we give the proof only for the positive roots which are of the form $\ga=\ve_a-\ve_b$.
For the other positive roots, one can prove by applying the similar strategy given in this proof.

By the assumptions, $\be$ is in $S$-part of the upper ray, $\al$ is in $N$-part of the upper ray and
$\{ \al,\be \} = \{ \ve_a-\ve_c, \ve_c-\ve_b \}$.
We assume once more that $\al=\ve_a-\ve_c$. Then we have the situation in $\AR$ as follows:
$$
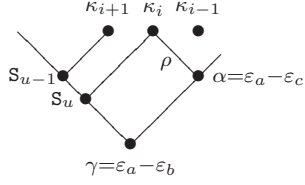
{\xy (-15,0)*{}="T1"; (15,0)*{}="T2"; (3,0)*{}="C1"; (-3,0)*{}="C0"; (9,0)*{}="C2";
(-15,-3)*{}="L1"; (-9,-6)*{}="L2"; (-6,-9)*{}="L3"; (12,-3)*{}="R1"; (9,-6)*{}="R2";
(0,-15)*{}="B";
"T1"; "B" **\dir{-};"R1"; "B" **\dir{-};"R2"; "C1" **\dir{-}?(.3)+(-2.5,0)*{\scriptstyle \rho};
"L2"; "C0" **\dir{-};
"L3"; "C1" **\dir{-};
"R2"*{\bullet};"B"*{\bullet};"C1"*{\bullet};"C0"*{\bullet};
"C2"*{\bullet};"L2"*{\bullet};"L3"*{\bullet};
"B"+(0,-3)*{\scriptstyle \ga=\ve_a-\ve_b };
"R2"+(8,0)*{\scriptstyle \al=\ve_a-\ve_c};
"C0"+(0,3)*{\scriptstyle \kappa_{i+1} };
"C1"+(0,3)*{\scriptstyle \kappa_i };
"C2"+(0,3)*{\scriptstyle \kappa_{i-1} };
"L2"+(-4,0)*{\scriptstyle \mtS_{u-1}};
"L3"+(-3,0)*{\scriptstyle \mtS_{u}};
\endxy}
$$
Then we have a multiplicity free positive root $\kappa_i=\ve_b-\ve_c$ for some $b < c$.
Note that, in this case, the maximal $S$-sectional path $\rho$ containing $\al$ is shallow.
Now we claim that $\kappa_{i+1}$ exists and contains $\ve_c$ as its summand.
Note that one of $\kappa_{i+1}$ and $\kappa_{i-1}$ exists. Assume that
\begin{equation} \label{as: kappa}
\text{$\kappa_{i-1}$ exists and contains $\ve_c$ as its summand.}
\end{equation}
Then it must be contained in the $S$-part of the $c$-swing. If it is in the $N$-part of the $c$-swing, it has an intersection with
$\rho$. Thus the positive root located at the intersection must be $\ve_c-\ve_c$, which yields a contradiction.

Under the assumption \eqref{as: kappa}, $\kappa_{i-1}$ must be in the $S$-part of the $c$-swing.
In this case, $\ve_a+\ve_c$ is located at the intersection of the $S$-part of the $c$-swing and $N$-part of upper ray. Thus we have $$\inp_1(\ve_a+\ve_c)<\inp_1(\ve_a-\ve_c)$$ which yields
a contradiction to Corollary \ref{Cor: nfree mfree position}. Hence $\kappa_{i-1}$ can not contain $\ve_c$ as its summand.

Now we prove the existence of $\kappa_{i+1}$. If it is not exists, then
$$\kappa_i=\kappa_{n-1}=\Dim \mtP(1)=\ve_1-\ve_c.$$
Then $\kappa_{i-1}$ must contain $\ve_c$ as its summand, which is impossible.

Now we need to prove that $\kappa_{i+1}$ is contained in the $N$-part of the $c$-swing. If it is not contained in the $N$-part of the $c$-swing,
then there is no intersection the shallow maximal $(N,-b)$-sectional path and the $c$-swing. It implies
that $\ve_c-\ve_b \not \in \Prt$ which is impossible. Thus we can conclude that $\kappa_{i+1}$ is contained in the $N$-part of the $c$-swing.
Thus we have
$$\mtS_{u-1}=\ve_c-\ve_b=\be$$
by Theorem \ref{Thm: V-swing} and Theorem \ref{thm: short path}.
Thus we have $$ \inp_1(\be)+\inp_1(\al)=\inp_1(\ga)$$
which implies our assertion by \eqref{eq: p ij D} {\rm (i)}.
\end{proof}

\subsection{Multiplicity free positive roots of the form $\ve_a-\ve_b$ ($a < b \le n-2$)} \label{sec: vea-veb}

\begin{lemma} \label{lem: mfree pair position 1}
Assume $\ga=\ve_a-\ve_b$, $\al=\ve_c-\ve_b$ for $a<c<b \le n-2$ and $\inp_1(\ga)<\inp_1(\al)$.
\begin{enumerate}
\item[({\rm a})] If $\ga$ is contained in the $N$-part of the $a$-swing, then $\be=\ve_a-\ve_c$ is
contained in the $S$-part of the $a$-swing.
\item[({\rm b})] If $\ga$ is contained in the $S$-part of the $a$-swing, then $\be=\ve_a-\ve_c$ is
contained in the $N$-part of the $a$-swing.
\end{enumerate}
\end{lemma}

\begin{proof}
(a) Assume that $\be$ is in the $N$-part of the $a$-swing. By the convexity of $\prec_Q$ and Theorem \ref{Thm: V-swing},
$\inp_1(\ga)<\inp_1(\be)$ and hence we have
$$
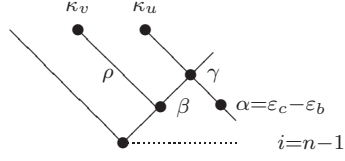
{\xy (-15,0)*{}="T1"; (15,0)*{}="T2"; (3,0)*{}="C1"; (-6,0)*{}="C0"; (-1,0)*{}="C00";
(-15,-3)*{}="L1"; (-9,-6)*{}="L2"; (-6,-9)*{}="L3"; (12,-3)*{}="R1"; (9,-6)*{}="R2";
(0,-15)*{}="B"; (-14,-15)*{}="B0";
"T1"; "B" **\dir{-};"R1"; "B" **\dir{-};"R2"+(6,-6); "C1" **\dir{-};
"B"; "B"+(15,0) **\dir{.};
"R2"+(-4,-4.7); "C0" **\dir{-};
"R2"*{\bullet};"B"*{\bullet};"C1"*{\bullet};"C0"*{\bullet};
"R2"+(4,-4)*{\bullet};
"R2"+(-4,-4.3)*{\bullet};
"B"+(25,0)*{\scriptstyle i=n-1};
"R2"+(3,0)*{\scriptstyle \ga};
"C0"+(0,3)*{\scriptstyle \kappa_{v} };
"C1"+(0,3)*{\scriptstyle \kappa_u };
"R2"+(12,-4)*{\scriptstyle \al=\ve_c-\ve_b};
"R2"+(-1,-4.3)*{\scriptstyle \be};
"R2"+(-11,0)*{\scriptstyle \rho};
\endxy}
$$
where $\kappa_u=\ve_{i_u}-\ve_b$, $\kappa_v=\ve_{i_v}-\ve_c$ and $\rho$ is a part of the shallow maximal $(S,-c)$-sectional path.
Then $\kappa_{u-1}$ or $\kappa_{u+1}$ must contain $\ve_b$ as its summand.

(i) Assume $\kappa_{u-1}$ (resp. $\kappa_{u+1}$) contains $\ve_b$ and is contained in the $S$-part of the $b$-swing. Then $\ve_a+\ve_b$ contained in $N$-part of the $a$-swing
and $\inp_1(\ve_a+\ve_b) < \inp_1(\ve_a-\ve_b)$ (resp. $\inp_1(\ve_a+\ve_b) < \inp_1(\ve_a-\ve_c)$),
which is a contradiction to Corollary \ref{Cor: nfree mfree position}.

(ii) Assume $\kappa_{u-1}$ (resp. $\kappa_{u+1}$) contains $\ve_b$ and is contained in the $N$-part of the $b$-swing. Then we have an intersection of
the $b$-swing and the shallow maximal $(S,-b)$-sectional path (resp. the shallow maximal $(S,-c)$-sectional path), which is impossible.

Thus there is no shallow maximal $(S,-c)$-sectional path and hence our assertion follows. By applying the similar argument, one can prove (b).
\end{proof}

\begin{proposition} \label{prop: mfree 1}
Assume $\ga=\ve_a-\ve_b$ $(a<c<b \le n-2)$. Assume that there exists $\al=\ve_c-\ve_b$ such that $a<c<b \le n-2$
and $\inp_1(\ga)<\inp_1(\al)$. Then we have
$$ \inp_1(\ga)+ \inp_1(\be)=\inp_1(\al), \qquad \text{ where $\be=\ga-\al$}.$$
Moreover,
\begin{enumerate}
\item[({\rm a})] if there are two $(\al,\be)$ and $(\al',\be')$ satisfying the assumption, then we have
$$|\inp_2(\al) - \inp_2(\be)| = |\inp_2(\al') - \inp_2(\be')|=2n-2+\inp_1(\ga).$$
\item[({\rm b})] there is a surjective $\Up$-module homomorphism
$$\VQbe \otimes \VQal \twoheadrightarrow \VQga \quad\text{or}\quad \VQal \otimes \VQbe \twoheadrightarrow \VQga$$
and hence there is a surjective $R$-module homomorphism
$$\SQbe \cvl \SQal \twoheadrightarrow \SQga \quad\text{or}\quad \SQal \cvl \SQbe \twoheadrightarrow \SQga.$$
\end{enumerate}
\end{proposition}

\begin{proof}
Assume that $\ga$ is contained in the $S$-part of the
$a$-swing. Then Lemma \ref{lem: mfree pair position 1} tells that
$\ga$ and $\al$ are contained in the shallow maximal $N$-sectional path:
$$  N_r \to N_{r-1} \to \cdots \to N_{2} \to N_1,$$
where $\inp_1(N_l)=l$, $\al=N_t$ and $\ga=N_s$ for $1 \le s < t \le r$. Set $\ell=t-s$.

We first claim that the $S$-part of the $c$-swing is shorter than or equal to the $N$-part of the $c$-swing.
If $S$-part of the $c$-swing is longer than the $N$-part of the $c$-swing, we have the following situation:
$$
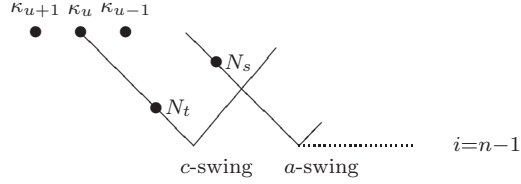
{\xy (-15,0)*{}="T1"; (15,0)*{}="T2"; (3,0)*{}="C1"; (-6,0)*{}="C0"; (-1,0)*{}="C00";
(-29,0)*{}="T0";
(-15,-3)*{}="L1"; (-9,-6)*{}="L2"; (-6,-9)*{}="L3"; (12,-3)*{}="R1"; (9,-6)*{}="R2";
(0,-15)*{}="B"; (-14,-15)*{}="B0";
"T1"; "B" **\dir{-};"B"+(3,3); "B" **\dir{-};"C00"+(-2,-2); "B0" **\dir{-};
"B"; "B"+(15,0) **\dir{.};"T0"; "B0" **\dir{-};
"C00"+(-10,-4)*{\bullet};"C00"+(-7,-4)*{\scriptstyle N_s};
"C00"+(-18,-10)*{\bullet};"C00"+(-15,-10)*{\scriptstyle N_t};
"T0"*{\bullet};"T0"+(6,0)*{\bullet};"T0"+(-6,0)*{\bullet};
"B"+(25,0)*{\scriptstyle i=n-1};
"B"+(3,-3)*{\scriptstyle a\text{-swing}};
"B0"+(3,-3)*{\scriptstyle c\text{-swing}};
"T0"+(0,3)*{\scriptstyle \kappa_u};
"T0"+(6,3)*{\scriptstyle \kappa_{u-1}};
"T0"+(-6,3)*{\scriptstyle \kappa_{u+1}};
\endxy}
$$
where $\kappa_u=\ve_c-\ve_{j_u}$. If $\kappa_{u-1}=\ve_{i_{u-1}}-\ve_c$, then any shallow maximal sectional path
containing $\kappa_{u-1}$ can not exist. On the other hand, if $\kappa_{u+1}=\ve_{i_{u}}-\ve_c$,
any shallow maximal sectional path containing $\kappa_{u+1}$ can not have an intersection with the $a$-swing. Hence we have the claim.

Now the situation in $\AR$ can be drawn as follows:
$$
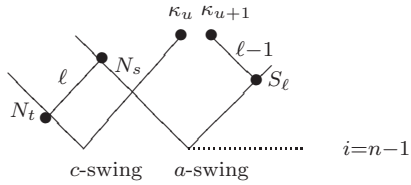
{\xy (-15,0)*{}="T1"; (15,0)*{}="T2"; (3,0)*{}="C1"; (-6,0)*{}="C0"; (-1,0)*{}="C00";
(-29,0)*{}="T0";
(-15,-3)*{}="L1"; (-9,-6)*{}="L2"; (-6,-9)*{}="L3"; (12,-3)*{}="R1"; (9,-6)*{}="R2";
(0,-15)*{}="B"; (-14,-15)*{}="B0";
"T1"; "B" **\dir{-};"T2"+(-4,-4); "B" **\dir{-};"C00"; "B0" **\dir{-};
"B"; "B"+(15,0) **\dir{.};"T0"+(5,-5); "B0" **\dir{-};
"C00"+(-10.5,-3);"C00"+(-18,-11) **\dir{-}?(.3)+(-3,0)*{\scriptstyle \ell};
"C00"+(13,-6)*{\scriptstyle S_\ell};
"C00"*{\bullet};"C00"+(0,3)*{\scriptstyle \kappa_u};
"C00"+(4,0)*{\bullet};"C00"+(6,3)*{\scriptstyle \kappa_{u+1}};
"C00"+(4,0);"C00"+(10,-6) **\dir{-}?(.3)+(4,0)*{\scriptstyle \ell-1};
"C00"+(10,-6)*{\bullet};
"C00"+(-10.5,-3)*{\bullet};"C00"+(-7,-4)*{\scriptstyle N_s};
"C00"+(-18,-11)*{\bullet};"C00"+(-21,-10)*{\scriptstyle N_t};
"B"+(25,0)*{\scriptstyle i=n-1};
"B"+(3,-3)*{\scriptstyle a\text{-swing}};
"B0"+(3,-3)*{\scriptstyle c\text{-swing}};
\endxy}
$$
where $\kappa_{u}=\ve_c-\ve_{j_u}$. Then one can observe that $\kappa_{u+1}=\ve_{i_{u+1}}-\ve_{c}$ and
$$S_\ell = \ve_a-\ve_c.$$ Thus
\begin{align}\label{eq: rst1}
 \inp_1(\ga)+ \inp_1(\be)=\inp_1(\al) \quad\text{ and }\quad |\inp_2(\al) - \inp_2(\be)| = 2n-2+\inp_1(\ga).
 \end{align}

The assertion {\rm (b)} follows from the second condition in \eqref{eq: p ij D} {\rm (i)} and \eqref{eq: rst1}.
\end{proof}

\begin{remark} \label{rmk: mfree pair}
Let $\ga$ be a multiplicity free positive root and $(\al,\be)$ be a pair of $\ga$.
Then Corollary \ref{cor: n,n-1} and Theorem \ref{thm: short path} tell that
one of $\al$ and $\be$ should have contained in the shallow maximal sectional path $\rho$
which contains $\ga$ together.
\end{remark}

\begin{theorem}
For every pair  $(\al,\be)$ of $\ga=\ve_a-\ve_b$ $(a<b \le n-2)$, there is a total order $<$ such that
\begin{itemize}
\item[({\rm a})] it is a compatible with the convex partial order $\preceq_Q$,
\item[({\rm b})] $(\al,\be)$ is a minimal pair of $\ga$ with respect to the total order.
\end{itemize}
\end{theorem}

\begin{proof}
First, we claim that the following situation can not happen:
$$ \text{ $\ve_a-\ve_c$ and $\ve_a-\ve_b$ are in the same part of the $a$-swing and $\inp_1(\ve_a-\ve_c)>\inp_1(\ga)$.}$$
Assume that $\ve_a-\ve_c$ and $\ve_a-\ve_b$ are contained in $N$-part of the $a$-swing. By Theorem \ref{thm: short path},
there are shallow maximal $S$-sectional paths $\rho$ and $\rho'$ containing $\ga$ and $\ve_a-\ve_c$, respectively.
Then we can draw the situation in $\AR$ as follows:
$$
{\xy (-12,-3)*{}="T1"; (13,-2)*{}="T2"; (0,-15)*{}="B";
(11,-4)="R1";(6,-9)="R2"; (2,-5)="C";
(7,0)="L1";(-3,0)="L2";
"L1"+(5,0); "L1"+(25,0) **\dir{.};
"L1"+(30,0)*{\scriptstyle i=1};
"T1"; "B" **\dir{-};"T2"; "B" **\dir{-};
"R1"; "L1" **\dir{-};"R2"; "L2" **\dir{-};"L1"; "C" **\dir{-};
"R1"*{\bullet};"R2"*{\bullet};"L1"*{\bullet};"L2"*{\bullet};"C"*{\bullet};
"R1"+(7,0)*{\scriptstyle \ve_a-\ve_b};
"R2"+(7,0)*{\scriptstyle \ve_a-\ve_c};
"L1"+(2,3)*{\scriptstyle N=\kappa_r};
"L2"+(2,3)*{\scriptstyle M=\kappa_s};
"C"+(-3,0)*{\scriptstyle L};
\endxy}
$$
By \eqref{eq: ad fn}, we have
$$\ve_a-\ve_b+L=\ve_a-\ve_c+N \quad \text{ and hence }\quad \het(N)>\het(L).$$
However \eqref{eq: ad fn} also tells that $L=\sum_{k=r}^{s}\kappa_k$. Thus it can not happen. By the similar argument, they can not contained in $S$-part of the
$a$-swing together whenever $\inp_1(\ve_a-\ve_c)>\inp_1(\ga)$.
\begin{eqnarray} &&
\parbox{75ex}{
Note that any sum of pair $(\al',\be')$, in the level $n-1$ and $n$, can not be of the form $\ve_a-\ve_b$.
} \label{eq: m free fact 1}
\end{eqnarray}

(a) Now we deal with the following case first:
$$\text{$\inp_1(\ve_a-\ve_c)<\inp_1(\ga)$ are in the same part of the $a$-swing.}$$
Assume that they are contained in $N$-part of the $a$-swing. Then, by previous observations, one can prove that
$$
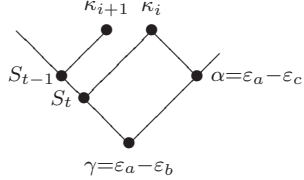
{\xy (-15,0)*{}="T1"; (15,0)*{}="T2"; (3,0)*{}="C1"; (-3,0)*{}="C0";
(-15,-3)*{}="L1"; (-9,-6)*{}="L2"; (-6,-9)*{}="L3"; (12,-3)*{}="R1"; (9,-6)*{}="R2";
(0,-15)*{}="B";
"T1"; "B" **\dir{-};"R1"; "B" **\dir{-};"R2"; "C1" **\dir{-};"L2"; "C0" **\dir{-};
"L3"; "C1" **\dir{-};
"R2"*{\bullet};"B"*{\bullet};"C1"*{\bullet};"C0"*{\bullet};"L2"*{\bullet};"L3"*{\bullet};
"B"+(0,-3)*{\scriptstyle \ga=\ve_a-\ve_b };
"R2"+(8,0)*{\scriptstyle \al=\ve_a-\ve_c};
"C0"+(0,3)*{\scriptstyle \kappa_{i+1} };
"C1"+(0,3)*{\scriptstyle \kappa_i };
"L2"+(-4,0)*{\scriptstyle S_{t-1}};
"L3"+(-3,0)*{\scriptstyle S_{t}};
\endxy}
$$
where
\begin{itemize}
\item $\kappa_{i+1}$ is a multiplicity free positive root containing $\ve_c$ as its summand.
\item $\kappa_{i+1}$ is contained in $N$-part of the $c$-swing.
\end{itemize}
Hence we have $S_{t-1}=\be=\ve_c-\ve_b$. Thus $(\al,\be)$ is contained in the upper ray of $\ga$.

With Remark \ref{rmk: mfree pair} and \eqref{eq: m free fact 1}, one can check that the total order $<^{U,1}_Q$ and $<^{U,2}_Q$
make $(\al,\be)$ minimal. Similarly, we can prove for the case when they are in the $S$-part of the $a$-swing.

(b) Now we deal with the following case:
$$\text{$(\ve_a-\ve_c)$ and $\ga$ are in the different parts of the $a$-swing.}$$
Assume that $\ga$ is in $N$-part of the $a$-swing and $\rho$ is the shallow maximal $S$-sectional path
containing $\ga$. Then one can prove that
$$ \inp_1(\ga)< \inp_1(\ve_c-\ve_b) \text{ implies } \ve_a-\ve_c \text{ is in the $N$-part of the $a$-swing}$$
by using Theorem \ref{Thm: V-swing}, Theorem \ref{thm: short path} and Corollary \ref{Cor: first}. Thus
the pair $(\al,\be)$ satisfies the equations in Proposition \ref{prop: mfree 1}.
Then one can check that the total order $<^{L,1}_Q$ and $<^{L,2}_Q$ make $(\al,\be)$ minimal as in (a).

By applying the similar arguments in this proof, one can prove for the remaining cases.
\end{proof}

\begin{example} In Example \ref{ex: example 1}, the four convex total orders in Remark \ref{Rmk: total orders} are given as follows:
\begin{itemize}
\item [$(<^{U,1}_Q)$]: \fontsize{10}{10}\selectfont $\lf 3,-4\rf<\lf 2,-4\rf<\lf 1,-4\rf<\lf 2,3\rf<\lf 2,-3\rf<\lf 1,2\rf<\lf 2,4\rf<\lf 1,-3\rf<\lf 1,3\rf<\lf 1,4\rf<\lf 1,-2\rf<\lf 3,4\rf  $. \fontsize{11}{11}\selectfont
\item [$(<^{U,2}_Q)$]: \fontsize{10}{10}\selectfont $\lf 3,-4\rf<\lf 2,-4\rf<\lf 1,-4\rf<\lf 2,-3\rf<\lf 2,3\rf<\lf 1,2\rf<\lf 2,4\rf<\lf 1,3\rf<\lf 1,-3\rf<\lf 1,4\rf<\lf 1,-2\rf<\lf 3,4\rf   $. \fontsize{11}{11}\selectfont
\item [$(<^{L,1}_Q)$]: \fontsize{10}{10}\selectfont $\lf 3,-4\rf <\lf 2,-4\rf <\lf 2,-3\rf <\lf 2,3\rf <\lf 1,-4\rf <\lf 1,2\rf <\lf 1,-3\rf <\lf 1,3\rf <\lf 2,4\rf <\lf 1,4\rf <\lf 3,4\rf <\lf 1,-2\rf $. \fontsize{11}{11}\selectfont
\item [$(<^{L,2}_Q)$]: \fontsize{10}{10}\selectfont $\lf 3,-4\rf <\lf 2,-4\rf <\lf 2,3\rf <\lf 2,-3\rf <\lf 1,-4\rf <\lf 1,2\rf <\lf 1,3\rf <\lf 1,-3\rf <\lf 2,4\rf <\lf 1,4\rf <\lf 3,4\rf <\lf 1,-2\rf $. \fontsize{11}{11}\selectfont
\end{itemize}
\end{example}

For the rest of this paper, we say that {\it a pair $(\alpha,\beta)$ is minimal} when $\alpha+\beta \in \Phi_n^+$ and the pair is a minimal pair with respect to
a suitable convex total order which is induced by some reduced expression $\redez \in [Q]$.

\subsection{Multiplicity free positive roots of the form $\ve_a \pm \ve_{n-1}$ and $\ve_a \pm \ve_{n}$} \label{sec: vea pm vet}
As we observed in Lemma \ref{Lem: last 2 line} and Corollary \ref{cor: n,n-1},
position of the subset of multiplicity free positive roots
$$\{ \ve_a+\ve_n, \ve_a-\ve_{n}, \ve_b+\ve_{n-1}, \ve_b-\ve_{n-1} \in \Prt \mid a \le n-1, \ b \le n-2\} $$
depends on the value $\xi_{n-1}-\xi_n$ and the parity of $n$. In short, $\inp_1(\al_{n-1}),\inp_1(\al_{n}) \in \{n-1,n \}$ and
\begin{center}
\begin{tabular}{| c | c | c | c  | } \hline
Case&$|\xi_{n-1}-\xi_n|$ & $\inp_1(\ve_a\pm\ve_n)$ & $\inp_1(\ve_b\pm\ve_{n-1})$ \\ \hline
A&$0$ & $n-1$ or $n$ & less than $n-1$ \\ \hline
B&$2$ & less than $n-1$ & $n-1$ or $n$ \\ \hline
\end{tabular}
\end{center}

\begin{remark} \label{rmk: obs 1}
Let us consider the roots $\{ \ve_b+\ve_{n-1} \mid b \le n-2 \}$ for
Case A. By Corollary \ref{cor: n,n-1}, the set forms the maximal
$N$-sectional path starting at $n-1$ and $n$ if $\{n-1,n\}$ are sink
in $Q$, or the maximal $S$-sectional path ending at $n-1$ and $n$ if
$\{n-1,n\}$ are sources in $Q$ (see \eqref{eq: both ss}).

For the set  $\{ \ve_b-\ve_{n-1} \mid b \le n-2 \}$ of Case A, Theorem
\ref{thm: short path} tells that the set forms the shallow
$(N,-(n-1))$-maximal path or $(S,-(n-1))$-maximal path (see Theorem \ref{thm: short path}).

For sets $\{ \ve_a+\ve_{n} \mid a \le n-1 \}$ and $\{ \ve_a-\ve_{n}
\mid a \le n-1 \}$) of Case B, Corollary \ref{cor: n,n-1} tells that
one of two set forms the maximal $S$-sectional path starting at
level $1$ and ending at level one of $n-1$ and $n$, and the another
set forms the maximal $N$-sectional path ending at level $1$ and
ending at level one of $n-1$ and $n$  (see \eqref{eq: one ss}).
\end{remark}

\begin{proposition} \label{pro: short n-1,n} \hfill
\begin{enumerate}
\item[({\rm a})] Take any root $\ga$ in $\{  \ve_b+\ve_{n-1}, \ve_b-\ve_{n-1} \in \Prt \mid b \le n-2\}$
when $|\xi_{n-1}-\xi_n|=0$. Then all pair $(\al,\be)$ of $\ga$ is a minimal pair.
Moreover, one of the conditions in \eqref{eq: p ij D} {\rm (i)} and {\rm (iii)} holds when
$\phi^{-1}(\gamma)=(k,z)$, $\phi^{-1}(\beta)=(j,y)$ and $\phi^{-1}(\alpha)= (i,x)$.
\item[({\rm b})] Take any root $\ga$ in $\{  \ve_a+\ve_{n}, \ve_a-\ve_{n} \in \Prt \mid a \le n-1 \}$
when $|\xi_{n-1}-\xi_n|=2$. Then all pair $(\al,\be)$ of $\ga$ is a minimal pair.
Moreover, one of the conditions in \eqref{eq: p ij D} {\rm (i)} and {\rm (iii)} holds when
$\phi^{-1}(\gamma)=(k,z)$, $\phi^{-1}(\beta)=(j,y)$ and $\phi^{-1}(\alpha)= (i,x)$.
\end{enumerate}
\end{proposition}

\begin{proof}
(a) Note that
\begin{itemize}
\item any sum of pair $(\al,\be)$ with $\inp_1(\al)$, $\inp_1(\be) \in \{ n-1,n\}$
can not be $\ga=\ve_b-\ve_{n-1}$,
\item $\inp_1(\ga) \le n-2$.
\end{itemize}
Thus we can apply the same arguments
in \S \ref{sec: vea-veb} for $\ga=\ve_b-\ve_{n-1}$ to prove our assertion. More precisely,
every pair $(\al,\be)$ of $\ga=\ve_b-\ve_{n-1}$ satisfies one of the conditions in Proposition \ref{prop: upper ray} and Proposition \ref{prop: mfree 1}.

For $\ga=\ve_b+\ve_{n-1}$, there exist only two pairs $$\big\{ \{ \al,\be \}
, \ \{ \al',\be' \} \big\} = \big\{ \{
\ve_{n-1}-\ve_{n}, \ve_b+\ve_{n} \}, \ \{ \ve_{n-1}+\ve_{n},
\ve_b-\ve_{n} \} \big\}$$
such that $\inp_1(\al),\inp_1(\al'),\inp_1(\be),\inp_1(\be') \in \{ n-1,n\}$. Then $\{\ga,\ve_{n-1}-\ve_n\}$ and
$\{\ga,\ve_{n-1}+\ve_n\}$ are contained in the same maximal
sectional paths, respectively. Then $(\al,\be)$ and
$(\al',\be')$ become minimal pairs
with respect to the total order $<^{L,1}_Q$ or $<^{L,2}_Q$.
Moreover, one can check that it satisfies $\phi^{-1}_1(\gamma) < n-2 < \phi^{-1}_1(\al),\phi^{-1}_1(\be) $ and the third condition in
\eqref{eq: p ij D} {\rm (iii)}. For the other pairs of $\ga$, we can apply the same arguments in \S \ref{sec: vea-veb} to prove our assertion. \\
(b) In this case, each $\ga=\ve_a-\ve_n$ (resp. $\ve_a+\ve_n$) has a unique pair $\{ \al,\be \}$ in level $n-1$ and $n$ as
$\{ \ve_{n-1}-\ve_{n}, \ve_a-\ve_{n-1} \}$ (resp. $\{ \ve_{n-1}+\ve_{n}, \ve_a-\ve_{n-1} \}$). Then $(\al,\be)$ becomes a minimal pair
with respect to the total orders $<^{L,1}_Q$ and $<^{L,2}_Q$. The remaining assertions can be proved in the similar way of {\rm (a)}.
For the other pairs, we can apply the same arguments in \S \ref{sec: vea-veb} to prove our assertion.
\end{proof}

\begin{proposition} Take a positive root of the form $\ga=\ve_a-\ve_{\mathtt{t}}$ or $\ve_a+\ve_{\mathtt{t}}$. Then
all pair $(\al,\be)$ of $\ga$ satisfying one of the following two conditions:
$$ \begin{cases} n-1- \dfrac{|\inp_2(\be)-\inp_2(\ga)|}{2} = |\inp_2(\ga)-\inp_2(\al)|&
\text{ if } \inp_1(\be) \in \{n-1,n\}, \\
n-1- \dfrac{|\inp_2(\al)-\inp_2(\ga)|}{2} = |\inp_2(\ga)-\inp_2(\be)|
&\text{ if } \inp_1(\al) \in \{n-1,n\}.
\end{cases} $$
Moreover, one of the conditions in \eqref{eq: p ij D} {\rm (iii)} holds when $\phi^{-1}(\gamma)=(k,z)$ with $k=n$, $\phi^{-1}(\beta)=(j,y)$ and $\phi^{-1}(\alpha)= (i,x)$.
\end{proposition}

\begin{proof}
Note that, in these cases, one of $\al$ and $\be$ should be contained in the level $n-1$ or $n$. For $\ga=\ve_a\pm\ve_{\mathtt{t}}$, assume that
$\inp_1(\be) \in \{ n-1,n\}$ and $\inp_2(\be) < \inp_2(\ga)$; i.e.,
$$  \inp_2(\ga)-\inp_2(\be)=2k \text{ for some } 1 \le k \le n-2 \text{ and }
\be= \ve_c\pm\ve_{\mathtt{t}} \ (a < b) .$$
By Corollary \ref{cor: reverse uni} (a), the $N$-part of the $b$-swing is longer than the $S$-part of the $b$-swing. Thus we have the subquiver in $\AR$ as follows:
$$
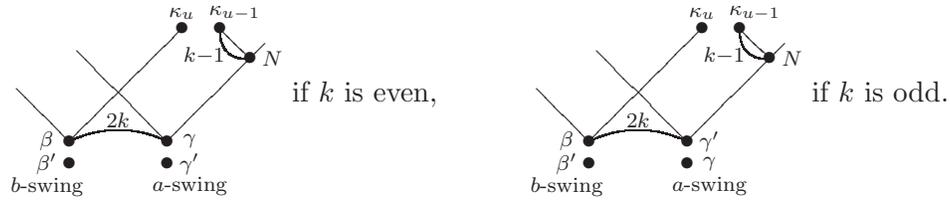
\raisebox{2.5em}{{\xy (-15,0)*{}="T1"; (-2,0)*{}="T2"; (15,0)*{}="T3"; (28,0)*{}="T4"; (3,0)*{}="K1";(7,0)*{}="K2";
(0,-15)*{}="B1"; (13,-15)*{}="B2";
(0,-18)*{}="B11"; (13,-18)*{}="B21";
(15,0)*{}="K2";(20,0)*{}="K1"; (24,-4)*{}="V1";
"T1"+(8,-8); "B1" **\dir{-};"T3"; "B1" **\dir{-};"T2"+(3,-3); "B2" **\dir{-};"T4"+(-2,-2); "B2" **\dir{-};
"V1"; "K1" **\dir{-};
"V1"; "K1" **\crv{(20,-4.3)}?(.4)+(-3.5,0)*{\scriptstyle k-1};
"B1";"B2";**\crv{(6.5,-12)}?(.5)+(0,1.3)*{\scriptstyle 2k};
"K2"*{\bullet};"K1"*{\bullet};"B2"*{\bullet};"B1"*{\bullet};"V1"*{\bullet};"B21"*{\bullet};"B11"*{\bullet};
"B1"+(-3,-6)*{\scriptstyle b\text{-swing}};
"B2"+(3,-6)*{\scriptstyle a\text{-swing}};
"K2"+(0,2)*{\scriptstyle \kappa_u};
"V1"+(3,0)*{\scriptstyle N};
"B1"+(-3,0)*{\scriptstyle \be};
"B2"+(3,0)*{\scriptstyle \ga};
"B11"+(-3,0)*{\scriptstyle \be'};
"B21"+(3,0)*{\scriptstyle \ga'};
"K1"+(2,2)*{\scriptstyle \kappa_{u-1}};
\endxy}} \text{ if $k$ is even, } \quad
\raisebox{2.5em}{{\xy (-15,0)*{}="T1"; (-2,0)*{}="T2"; (15,0)*{}="T3"; (28,0)*{}="T4"; (3,0)*{}="K1";(7,0)*{}="K2";
(0,-15)*{}="B1"; (13,-15)*{}="B2";
(0,-18)*{}="B11"; (13,-18)*{}="B21";
(15,0)*{}="K2";(20,0)*{}="K1"; (24,-4)*{}="V1";
"T1"+(8,-8); "B1" **\dir{-};"T3"; "B1" **\dir{-};"T2"+(3,-3); "B2" **\dir{-};"T4"+(-2,-2); "B2" **\dir{-};
"V1"; "K1" **\dir{-};
"V1"; "K1" **\crv{(20,-4.3)}?(.4)+(-3.5,0)*{\scriptstyle k-1};
"B1";"B2";**\crv{(6.5,-12)}?(.5)+(0,1.3)*{\scriptstyle 2k};
"K2"*{\bullet};"K1"*{\bullet};"B2"*{\bullet};"B1"*{\bullet};"V1"*{\bullet};"B21"*{\bullet};"B11"*{\bullet};
"B1"+(-3,-6)*{\scriptstyle b\text{-swing}};
"B2"+(3,-6)*{\scriptstyle a\text{-swing}};
"K2"+(0,2)*{\scriptstyle \kappa_u};
"V1"+(3,0)*{\scriptstyle N};
"B1"+(-3,0)*{\scriptstyle \be};
"B2"+(3,0)*{\scriptstyle \ga'};
"B11"+(-3,0)*{\scriptstyle \be'};
"B21"+(3,0)*{\scriptstyle \ga};
"K1"+(2,2)*{\scriptstyle \kappa_{u-1}};
\endxy}} \text{ if $k$ is odd.}
$$
Then we can prove that $\kappa_{u-1}$ contains $-\ve_c$ as its summand and $N=\ve_a-\ve_b=\al$, by applying the technique in
Proposition \ref{prop: upper ray}. Thus our assertions follow.
For the other cases, we can prove by applying the similar arguments in this proof.
\end{proof}

\begin{theorem}
For each positive root $\ga$ containing $\pm \ve_{n-1}$ or $\pm \ve_{n}$ as its summand, every pair $(\al,\be)$ of $\ga$ is minimal and
there exists a surjective homomorphism
$$\VQbe \otimes \VQal \twoheadrightarrow \VQga \quad \text{ and } \quad \SQbe \otimes \SQal \twoheadrightarrow \SQga.$$
\end{theorem}

\begin{proof}
Let us consider when $\inp_1(\ga) \in \{ n-1,n\}$ first. Assume that $\inp_1(\be) \in \{ n-1,n\}$ and $\inp_1(\al) \not \in \{ n-1,n\}$.
Then the previous proposition tells that $<^{U,1}_Q$ or $<^{U,2}_Q$ make $(\al,\be)$ minimal.
By Theorem \ref{thm: BkMc} (c), we have such a surjection.
Similarly, if $\inp_1(\al) \in \{ n-1,n\}$ and $\inp_1(\be) \not \in \{ n-1,n\}$, then
$<^{L,1}_Q$ or $<^{L,2}_Q$ makes $(\al,\be)$ minimal and we have such a surjection.

When $\inp_1(\ga) \le n-2$, Proposition \ref{pro: short n-1,n} tells that our assertion holds.
\end{proof}

Thus we have a conclusion as follows:
\begin{eqnarray} &&
\parbox{85ex}{ For every multiplicity free positive root $\ga$, every pair $(\al,\be)$ of $\ga$ is minimal
with respect to a suitable total order which is compatible with $\preceq_Q$.
}\label{conclusion: mfre}
\end{eqnarray}

\begin{theorem} \label{thm: free}
 Let $\ga$ be a multiplicity free positive root $\ga$. For every pair $(\al,\be)$ of $\ga$, we have
\begin{enumerate}
\item[({\rm a})] $\VQbe \otimes \VQal$ and $\SQbe \cvl \SQal$ have composition length two,
\item[({\rm b})] there exist surjections $\VQbe \otimes \VQal \twoheadrightarrow \VQga$ and
$\SQbe \cvl \SQal \twoheadrightarrow \SQga$.
\end{enumerate}
\end{theorem}

\begin{proof}
It is an immediate consequence of Theorem \ref{thm: BkMc}, Theorem \ref{thm: duality functor on Q} and \eqref{conclusion: mfre}.
\end{proof}

\begin{remark} \label{rmk: non-adapted}
For a reduced expression $\redez$ of the longest element $w_0$ of $D_4$
$$\widetilde{w}_0=s_1s_2s_3s_1s_2s_4s_1s_2s_3s_1s_2s_4,$$
one can easily check that it is {\it not} adapted to any Dynkin quiver $Q$ of type $D_4$.
Moreover, for a multiplicity free positive root
$\al_2+\al_3+\al_4$, we have
$$\al_2+\al_3 \prec_{[\redez]} \al_3 \prec_{[\redez]} \al_2+\al_3+\al_4  \prec_{[\redez]} \al_2+\al_4  \prec_{[\redez]} \al_4$$
by considering all reduced expressions in $[\redez]$. Equivalently, the pair $(\al_2+\al_3 , \al_4)$ can not be minimal with respect to any convex total order
compatible with  $\prec_{[\redez]}$ even though $\al_2+\al_3+\al_4$ is multiplicity free.
\end{remark}

\subsection{Multiplicity non-free positive root of the form $\ve_a + \ve_b$ for $b \le n-2$.}
By Theorem \ref{Thm: V-swing}, a multiplicity non-free positive root of the form $\ve_a + \ve_b$ for $b \le n-2$ is located at
the intersection of the $a$-swing and the $b$-swing. More precisely,
\begin{itemize}
\item $\ve_a + \ve_b$ is located at the $S$-part of the $a$-swing and  the $N$-part of the $b$-swing or
\item $\ve_a + \ve_b$ is located at the $N$-part of the $a$-swing and  the $S$-part of the $b$-swing.
\end{itemize}
For a pair $(\al,\be)$ of $\al+\be=\ga=\ve_a+\ve_b$ ($b \le n-2$), if
$\al$ or $\be$ is in the same sectional paths containing $\ga$, we can prove that
$$\text{$(\al,\be)$ is a minimal pair of $\ga$,}$$
by applying the same arguments in \S \ref{sec: vea-veb} and \S \ref{sec: vea pm vet}.

Comparing with multiplicity free positive roots in the previous subsections, multiplicity non-free positive roots of the form $\ve_a + \ve_b$ for $b \le n-2$  can
have a pair $(\al,\be)$ which are not in the same sectional paths containing $\ve_a + \ve_b$.

In this subsection, we prove that there exist $(n-b-1)$-may pairs $(\al,\be)$ of $\ve_a+\ve_b$ which can not be minimal
with respect to {\it any} convex total order compatible with $\preceq_Q$. However, we have still surjections
$$ \VQbe \otimes \VQal \twoheadrightarrow V_Q(\ve_a+\ve_b) \quad\text{and}\quad \SQbe \cvl \SQal \twoheadrightarrow S_Q(\ve_a+\ve_b)$$
for the non-minimal pair $(\al,\be)$. Here, $(\al,\be)$ is {\it non-minimal} means that $(\al,\be)$ can not be a minimal pair of $\ga=\al+\be$ for any
convex total order compatible with $\preceq_Q$.

\begin{proposition} \label{prop: nfree}
Assume that there exists a pair $(\al,\be)$ of  $\al+\be=\ga=\ve_a+\ve_b$ $(b\le n-2)$ such that $\al$ and $\be$ are not in the same sectional paths containing $\ga$.
Then we have
\begin{align} \label{eq: nfree}
  \inp_1(\ga) = 2n-2-\inp_1(\al)-\inp_1(\be).
\end{align}
Moreover,
\begin{enumerate}
\item[({\rm a})] we have surjections
$$ \VQbe \otimes  \VQal \twoheadrightarrow \VQga
\quad\text{and}\quad \SQbe \cvl \SQal \twoheadrightarrow \SQga,$$
\item[({\rm b})] the pair $(\al,\be)$ is a non-minimal pair.
\end{enumerate}
\end{proposition}

\begin{proof}
Note that
$\{ \al,\be \} = \{ \ve_a \pm \ve_c, \ve_b \mp \ve_c \} $ for some $a<b<c$.  Assume that $\be=\ve_b-\ve_c$ and $\ga$ is at the intersection of the
$N$-part of the $b$-swing and the $S$-part of the $a$-swing. By assumptions, we have the following subquiver in $\AR$ as follows:
$$
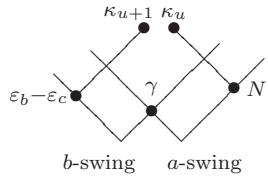
{\xy (-15,0)*{}="T1"; (-7,0)*{}="T2"; (15,0)*{}="T3"; (23,0)*{}="T4"; (3,0)*{}="K1";(7,0)*{}="K2";
(0,-15)*{}="B1"; (8,-15)*{}="B2";
(-6,-9)*{}="V1";(15,-8)*{}="V2"; (4,-11)*{}="I";
"T1"+(6,-6); "B1" **\dir{-};"T3"+(-2,-2); "B1" **\dir{-};"T2"+(3,-3); "B2" **\dir{-};"T4"+(-4,-4); "B2" **\dir{-};
"V1"; "K1" **\dir{-};"V2"; "K2" **\dir{-};
"V1"*{\bullet};"K1"*{\bullet};"K2"*{\bullet};"V2"*{\bullet};"I"*{\bullet};
"K1"+(-2,2)*{\scriptstyle \kappa_{u+1}};
"V1"+(-5,0)*{\scriptstyle \ve_b-\ve_c};
"V2"+(3,0)*{\scriptstyle N};
"I"+(0,3)*{\scriptstyle \ga};
"K2"+(0,2)*{\scriptstyle \kappa_{u}};
"B1"+(-3,-3)*{\scriptstyle b\text{-swing}};
"B2"+(3,-3)*{\scriptstyle a\text{-swing}};
\endxy}
$$
The existence such $\kappa_{u+1}=\ve_e-\ve_c$ is guaranteed by Theorem \ref{thm: short path} and Corollary \ref{cor: n,n-1}. Now we prove
that $\kappa_{u}$ exists and $\kappa_{u}$ contains $\ve_c$ as its summand. if $\kappa_{u}$ does not exist, then Corollary \ref{cor: reverse uni} (b) tells that
$\kappa_{u+1}=\Dim \mtI(1)$.
Then $\be$ should be the same as $\be=\ve_1-\ve_c$ which is a contradiction to the assumption of $b$. Thus $\kappa_{u}$ exists.

Now we claim that $\kappa_{u}$ contains $\ve_c$ as its summand. If $\kappa_{u+2}$ does not exists, $\kappa_{u}$ contains $\ve_c$.
Thus we assume that $\kappa_{u+2}$ exists. If $\kappa_{u+2}$ contains $\ve_c$ as its summand and is in the $S$-part of the $c$-swing,
we have an intersection of the shallow maximal $(N,-c)$-sectional path and the $c$-swing. But it can not happen. If $\kappa_{u+2}$ contains $\ve_c$ as its summand
and is in the $S$-part of the $c$-swing, $\inp_1(\ve_b+\ve_c)<\inp_1(\ve_b-\ve_c)$. But it can not happen, since, by the argument in Proposition
\ref{prop: upper ray}, $\inp_1(\ve_b+\ve_c)$ should be larger than $\inp_1(\ve_b-\ve_c)$. Thus our claim follows.

Using the technique in the previous paragraph, one can prove that $\kappa_{u}$ is in the $S$-part of the $c$-swing. Then the root $N$ at the
intersection of the $c$-swing and the $a$-swing becomes $\ve_a+\ve_c$. Observing the locations of $\al$, $\be$ and $\ga$, we have
\begin{align*}
&\big(\inp_1(\be)-1\big) + \big(\inp_2(\kappa_{u})-\inp_2(\kappa_{u+1}) \big)+ \big(\inp_1(\al)-1\big) \\
& \hspace{20ex} =\big(n-1 - \inp_1(\be)\big) + 2\big(n-1 - \inp_1(\ga)\big) + \big(n-1 - \inp_1(\al)\big).
\end{align*}
Since $\big(\inp_2(\kappa_{u})-\inp_2(\kappa_{u+1}) \big)=2$, we have \eqref{eq: nfree}.
By  \eqref{eq: p ij D} {\rm (ii)}, we have a surjective homomorphism ({\rm a}).

Note that there exists a pair $(\ve_b-\ve_{\mathtt{t}},\ve_a+\ve_{\mathtt{t}})$ of $\ve_a+\ve_b$.
Since there exist sectional paths from $\be$ to $\ve_b-\ve_{\mathtt{t}}$ and from $\ve_a+\ve_{\mathtt{t}}$ to $\al$,
$(\al,\be)$ can not be a minimal pair
with respect to any total order compatible with $\preceq_Q$ by Remark \ref{rmk: non minimal pair}.

For the other cases, we can apply the same argument given in this proof.
\end{proof}

\begin{corollary} \label{cor: double means}
For a non-minimal pair $(\al,\be)$ in {\rm Proposition \ref{prop: nfree}}, $d_{\VQbe,\VQal}(z)$
has a zero of multiplicity $2$ at $z=(-q)^{|\inp_2(\al)-\inp_2(\be)|}$.
\end{corollary}

\begin{proof}
From Proposition \ref{prop: nfree}, we can easily check that
$$\inp_1(\al),\inp_1(\be) \le n-2, \quad
\inp_1(\al)+\inp_1(\be) \ge n \quad \text{and hence} \quad \inp_1(\al),\inp_1(\be) >1.$$ Since
\begin{align*}
 \inp_2(\al) & -\inp_2(\be) \\
& = n-1- \inp_1(\al) +n-1- \inp_1(\be)+2\left(n-1- \big(2n-2-\inp_1(\al)-\inp_1(\be) \big) \right) \\
& = \inp_1(\al)+\inp_1(\be),
\end{align*}
our assertion follows from \eqref{eq: dpole 1}.
\end{proof}

\begin{theorem} \label{thm: non-free}
For a multiplicity non-free positive root $\ga=\ve_a+\ve_b$, $(n-b-1)$-many non-minimal pairs of $\ga$ exist.
\end{theorem}

\begin{proof} Note that
\begin{itemize}
\item a pair $(\al,\be)$ of $\ga$ is of the form $(\ve_a \pm \ve_c, \ve_b \mp \ve_c )$
or $(\ve_b \mp \ve_c,\ve_a \pm \ve_c)$ for some $b<c$,
\item there are $(n-b-1)$-roots in the shorter part of the $b$-swing.
\end{itemize}
Let $a=i_{\sigma_{k}}$, $b=i_{\sigma_{l}}$ $(k \ne l)$ and $1=i_{\sigma_{\ell}}$ as in \eqref{eq:sigma enumeration}. Then either
$l<k \le \ell$, $l>k \ge \ell$, $k \ge \ell \ge l$ or  $l \ge \ell \ge k$ by Corollary \ref{cor: reverse uni} (b).

\noindent
(i) Assume that $l>k \ge \ell$. By Corollary \ref{cor: reverse uni}, for any positive root $\ve_b \mp \ve_c$ in the $S$-part of the $b$-swing,
we have the following subquiver of $\AR$ as follows:
$$
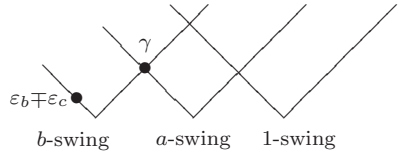
{\xy (-15,0)*{}="T1"; (-2,0)*{}="T2"; (15,0)*{}="T3"; (28,0)*{}="T4"; (3,0)*{}="K1";(7,0)*{}="K2";
(0,-15)*{}="B1"; (13,-15)*{}="B2"; (25,-15)*{}="B3"; (40,0)*{}="T5";
(6.6,-8.4)*{}="I1";
(15,0)*{}="K2";(10,0)*{}="K1"; (-2.5,-12.5)*{}="V1";
(21.5,-6.5)*{}="V2";
"T1"+(8,-8); "B1" **\dir{-};"T3"; "B1" **\dir{-};"T2"+(3,-3); "B2" **\dir{-};"T4"; "B2" **\dir{-};
"K1"; "B3" **\dir{-};
"T5"; "B3" **\dir{-};
"V1"*{\bullet};"I1"*{\bullet};
"B1"+(-3,-3)*{\scriptstyle b\text{-swing}};
"B2"+(0,-3)*{\scriptstyle a\text{-swing}};
"B3"+(2,-3)*{\scriptstyle 1\text{-swing}};
"V1"+(-5,0)*{\scriptstyle \ve_b \mp \ve_c};
"I1"+(0,3)*{\scriptstyle \ga};
\endxy}
$$
Thus it suffices to show that $\ve_a\pm \ve_c$ is in the $N$-part of the $a$-swing.
Assume $\be=\ve_a\pm \ve_c$ is contained in the $S$-part.
If $\inp_1(\be)< \inp_1(\ga)$, it is a contradiction to the convexity of $\prec_Q$.
If $\inp_1(\be)> \inp_1(\ga)$, then Corollary \ref{cor: nfree position} tells that
$\be=\ve_a+ \ve_c$ and hence $c$-swing is located between $a$-swing and $b$-swing. However, it is a contradiction to Corollary \ref{cor: reverse uni} (b).
Hence $\be$ must contained in the $N$-part of the $a$-swing. We can prove for the case $l<k \le \ell$ by using the similar argument.

\noindent
(ii) Assume $k \ge \ell \ge l$. By Corollary \ref{cor: reverse uni}, for any positive root $\ve_b \mp \ve_c$ in the $N$-part of the $b$-swing,
we have the following subquiver of $\AR$ as follows:
$$
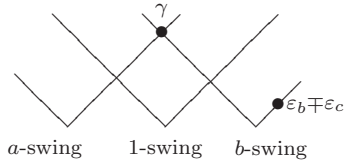
{\xy (-15,0)*{}="T1"; (-2,0)*{}="T2"; (15,0)*{}="T3"; (28,0)*{}="T4"; (3,0)*{}="K1";(7,0)*{}="K2";
(0,-15)*{}="B1"; (13,-15)*{}="B2"; (25,-15)*{}="B3"; (40,0)*{}="T5";
(12.5,-2.5)*{}="I1";
(15,0)*{}="K2";(10,0)*{}="K1"; (-2.5,-12.5)*{}="V1";
(21.5,-6.5)*{}="V2";
"T1"+(8,-8); "B1" **\dir{-};"T3"; "B1" **\dir{-};"T2"; "B2" **\dir{-};"T4"; "B2" **\dir{-};
"K1"; "B3" **\dir{-};
"B3"+(6,6); "B3" **\dir{-};
"I1"*{\bullet};
"B1"+(-3,-3)*{\scriptstyle a\text{-swing}};
"B2"+(0,-3)*{\scriptstyle 1\text{-swing}};
"B3"+(2,-3)*{\scriptstyle b\text{-swing}};
"B3"+(3,3)*{\bullet}; "B3"+(8,3)*{\scriptstyle \ve_b \mp \ve_c};
"I1"+(0,3)*{\scriptstyle \ga};
\endxy}
$$
Assume $\be=\ve_a\pm \ve_c$ is contained in the $N$-part.
If $\inp_1(\be)> \inp_1(\ga)$, it is a contradiction to the convexity of $\prec_Q$.
If $\inp_1(\be)< \inp_1(\ga)$, then Corollary \ref{cor: nfree position} tells that
$\be=\ve_a+ \ve_c$ and hence $c$-swing is located either between $a$-swing and $1$-swing, or between $c$-swing and $1$-swing.
However, both cases yields a contradiction to Corollary \ref{cor: reverse uni} (b).
Hence $\be$ must contained in the $S$-part of the $a$-swing. We can prove for the case $l \ge \ell \ge k$ by using the similar argument.
\end{proof}

\begin{corollary}
For a pair $(\al,\be)$ of $\al+\be=\ga \in \Phi^+$, assume that $\al$ or $\be$ $($or both$)$ are contained in the same sectional path$($s$)$ of $\ga$. Then
$(\al,\be)$ is minimal.
\end{corollary}

\begin{proof}
The proof is an immediate consequence of the preceding theorem.
\end{proof}

\begin{corollary} \label{cor: num min nonmin}
For a positive roots $\ga \in \Phi^+$,
\begin{itemize}
\item[{\rm (i)}] the number of minimal pairs of $\ga$ is $|\Supp_{\ge 1}(\ga)|-1$,
\item[{\rm (ii)}] the number of non-minimal pairs of $\ga$ is $|\Supp_{\ge 2}(\ga)|$.
\end{itemize}
Hence the number of pairs $(\al,\be)$ of $\ga=\al+\be$ is the same as ${\rm ht}(\ga) -1$, where ${\rm ht}(\ga)=\sum_{i \in I}n_i$ for
$\ga=\sum_{i \in I}n_i\al_i$.
\end{corollary}

\begin{proof}
For a multiplicity free root $\ga$, $|\Supp_{\ge 1}(\ga)|={\rm ht}(\ga)$. Then out assertion follows from the easy observation that
the number of pairs of $\ga=|\Supp_{\ge 1}(\ga)|-1$ and \eqref{conclusion: mfre} in this case.

For a multiplicity non-free root $\ga=\varepsilon_a +  \varepsilon_b = \sum_{ a \le k < b} \al_k + 2 \sum_{ j \le k \le n-2} \al_k +\al_{n-1}+\al_{n},$
one of the pair $(\al,\be)$ of $\ga$ should be the one of the following forms:
$$  \ve_b \pm \ve_k \ \ (b < k \le n) \text{ or } \ve_c +\ve_b  \ \ (a < c <b).$$
Thus the number of pairs of $\ga$ is
$$2(n-b)+(b-a-1)=2n-a-b-1,$$ which is the same as ${\rm ht}(\ga)-1$. Then our assertion follows from Theorem \ref{thm: non-free}
and the observations that $|\Supp_{\ge 1}(\ga)|=n-a+1$ and $|\Supp_{\ge 2}(\ga)|=n-b-1$.
\end{proof}

\begin{example} In Example \ref{ex: example 1}, the longest root $ \ga=\ve_1+\ve_2 = \al_1+2\al_2+\al_3+\al_4$ has
$${\rm ht}(\ga)=5, \ \  |\Supp_{\ge 1}(\ga)|=4 \ \text{ and } \ |\Supp_{\ge 2}(\ga)|=1.$$
One can easily check that
\begin{itemize}
\item $(\{ 1,-4 \}, \{ 2,-4 \})$, $(\{ 1, 3 \}, \{ 2, -3 \})$, $(\{ 1,-3 \}, \{ 2,3 \})$, $(\{ 1,4 \}, \{ 2,-4 \})$ are pairs of $\ga$,
\item $(\{ 1,-4 \}, \{ 2,-4 \})$, $(\{ 1, 3 \}, \{ 2, -3 \})$, $(\{ 1,-3 \}, \{ 2,3 \})$ are minimal pairs of $\ga$,
\item $(\{ 1,4 \}, \{ 2,-4 \})$ is  only the non-minimal pairs of $\ga$.
\end{itemize}
\end{example}

By Proposition \ref{prop: nfree} and Corollary \ref{cor: double means}, we know that the surjective homomorphism
$$\VQbe \otimes \VQal \twoheadrightarrow \VQga \quad \text{ for a non-minimal pair $(\al,\be)$ of $\ga$} $$
arises from \eqref{eq: p ij D} {\rm (ii)}. Conversely, for every minimal pair $(\al,\be)$ of $\ga$,
the surjective homomorphism $\VQbe \otimes \VQal \twoheadrightarrow \VQga$ arises from
the other surjective homomorphisms in Theorem \ref{Thm: Dorey D_n(1)}. Thus we have the following Theorem by \eqref{eq: dpole 1}:

\begin{theorem} \label{thm: surj free} \hfill
\begin{itemize}
\item[{\rm (1)}] $d_{\VQbe,\VQal}(z)$ for a minimal pair $(\al,\be)$ of $\gamma \in \Prt$
has a zero of multiplicity $1$ at $z=(-q)^{|\inp_2(\al)-\inp_2(\be)|}$,
\item[{\rm (2)}] $d_{\VQbe,\VQal}(z)$ for a non-minimal pair $(\al,\be)$ of $\gamma \in \Prt$
has a zero of multiplicity $2$ at $z=(-q)^{|\inp_2(\al)-\inp_2(\be)|}$.
\end{itemize}
\end{theorem}

\section{Further applications}

\subsection{Category $\Cat^{(2)}_Q$}
Using the map $^\star$ in \eqref{eq: mathsf V}, we define the category $\Cat^{(2)}_Q$ in $\Cat_{D^{(2)}_{n+1}}$  as follows (see also \cite{Her10}):

\begin{definition} \label{def: CQtwo} Let $\Prt$ be the set of positive roots of finite type $D_{n+1}$.
The subcategory $\Cat^{(2)}_Q$ is the smallest abelian full subcategory of $\Cat_{D^{(2)}_{n+1}}$ satisfying
\begin{itemize}
\item[({\rm a})] $\mathsf{V}_Q(\be) \in \Cat^{(2)}_Q$ for all $\be \in \Prt$,
\item[({\rm b})] it is stable by taking submodule, subquotient, tensor product and extension.
\end{itemize}
\end{definition}

\begin{proposition} \label{Prop: Dorey on Gamma 2}
For every positive root $\ga \in \Phi^+_{n+1}$ with $\het(\ga) \ge 2$ and every minimal pair $(\al,\be)$ of $\ga$,
there exists a surjective $U'_q(D_{n+1}^{(2)})$-module homomorphism
\begin{align} \label{eq: 2 surj}
\mathsf{V}_Q(\be) \otimes \mathsf{V}_Q(\al) \twoheadrightarrow \mathsf{V}_Q(\ga).
\end{align}
\end{proposition}

\begin{proof}
As we observed in Theorem \ref{thm: surj free}, for every minimal pair $(\al,\be)$ of $\ga$,
the surjective homomorphism $\VQbe \otimes \VQal \twoheadrightarrow \VQga$
arises from \eqref{eq: p ij D} {\rm (i)} or {\rm (iii)}. Thus Remark \ref{rmk: Dn+1,1 Dn+1,2} tells that
\eqref{eq: 2 surj} always exists for a minimal pair  $(\al,\be)$ of $\ga$, which yields our assertion.
\end{proof}

From Proposition \ref{Prop: Dorey on Gamma 2}, the condition (a) in Definition \ref{def: CQtwo} can be re-written as follows also:
\begin{itemize}
\item[({\rm a}$'$)] $\mathsf{V}_Q(\al_k)$ for all $\al_k \in \Pi_{n+1}$.
\end{itemize}

\subsection{Convolution of $\SQbe$ for $\be$'s in the sectional path.} This subsection can be regarded as an analogue of
\cite[\S 3.3]{Oh14A}.

\begin{proposition} \label{Prop: commuting}
For any Dynkin quiver $Q$ and $\al$, $\be \in \Prt$ contained in the same sectional path,
$$ \text{ $\SQal \cvl \SQbe \simeq \SQbe \cvl \SQal$ is simple.} $$
\end{proposition}

\begin{proof}
By Theorem \ref{Thm: V-swing} and Theorem \ref{thm: short path}, we have
\begin{align*}
&  |\inp_2(\al) -\inp_2(\be)| = |\inp_1(\al) -\inp_1(\be)|
\qquad \qquad \quad\text{ if $\inp_1(\al),\inp_1(\be) \le n-2$,} \\
&  |\inp_2(\al) -\inp_2(\be)| = n-1-k  \qquad \quad\qquad\qquad\quad \ \  \text{ if $\{\inp_1(\al),\inp_1(\be)\} = \{n-1,k\}$ or $\{n,k\}$,}\\
&  |\inp_2(\al) -\inp_2(\be)| =0 \qquad \quad\qquad\qquad\quad \qquad\qquad \ \text{ if $\{\inp_1(\al),\inp_1(\be)\} = \{n-1,n\}$.}
\end{align*}
By \eqref{eq: denominator 1} and Theorem \ref{Thm: bpro}, we have
$$ \text{ $\VQal \otimes \VQbe \simeq \VQbe \otimes \VQal$ is simple.} $$
Thus our assertion follows from Theorem \ref{thm: duality functor on Q}.
\end{proof}

\begin{corollary} \label{cor: commuting}
 Choose $a \in \{ 1,2,\ldots, n-2\}$ and let $\{ \be_{j_k} \mid 1 \le k \le a-1 \}$ be the set of all positive root
containing $-\ve_a$ as their summand. Then we have
$$ \text{ $\dct{k=1}{a-1} S_Q(\be_{j_k})$ is simple.}$$
\end{corollary}

\begin{proof}
Our assertions are immediate consequences of Theorem \ref{thm: short path} and Proposition \ref{Prop: commuting}.
\end{proof}

\bibliographystyle{amsplain}


\end{document}